\documentclass[journal]{IEEEtran}
\usepackage{graphicx}
\usepackage{siunitx}
\usepackage{tikz,tikzscale,tikz-cd,circuitikz}
\usepackage{pgfplots}
\usepackage{setspace}

\usetikzlibrary{shapes,arrows,calc,fit}

\tikzstyle{block} = [draw=none,rectangle,minimum height=1.75em,minimum width=6em]
\tikzstyle{gain} = [draw,rectangle,minimum height=1em,minimum width=1em]
\tikzstyle{sum} = [draw,circle,node distance=0.5cm]
\tikzstyle{signal} = [coordinate]
\tikzstyle{pinstyle} = [pin edge={to-,thin,black}]

\usepackage{pgfplots}
\pgfplotsset{compat=newest,every axis/.append style={
                    label style={font=\footnotesize},
                    tick label style={font=\footnotesize}  
                    }}

\usepgfplotslibrary{fillbetween}

\usepackage{cite}

\ifCLASSINFOpdf
\else
\fi
\ifCLASSOPTIONcompsoc
  \usepackage[caption=false,font=normalsize,labelfont=sf,textfont=sf]{subfig}
\else
  \usepackage[caption=false,font=footnotesize]{subfig}
\fi
\usepackage{url}

\usepackage{amsmath}
\usepackage{amsmath,amssymb,ifthen,color}
\usepackage{rgsMacros,rgsEnvironments}

\usepackage{enumitem}

\usepackage{hyperref}

\newcommand{\interior}{\mathop{\rm int}\nolimits}
\newcommand{\closure}{\mathop{\rm cl}\nolimits}
\newcommand{\Anorm}[1]{\left|#1\right|_{\cal A}}

\renewcommand{\T}{{\cal T}}
\newcommand{\natspplus}{\nats_{>0}}
\newcommand{\xinitial}{x_{\circ}}

\usepackage{soul}

\usepackage{ifthen}
\newboolean{Report}
\newboolean{Journal}
\setboolean{Report}{true}  
\newcommand{\NotForReport}[1]{\ifthenelse{\boolean{Report}}{}{#1}}
\newcommand{\IfReport}[2]{\ifthenelse{\boolean{Report}}{#1}{{\color{black}#2}}}

\setboolean{Journal}{false}  
\newcommand{\NotForJournal}[1]{\ifthenelse{\boolean{Journal}}{}{{\color{black}#1}}}
\newcommand{\IfJournal}[2]{\ifthenelse{\boolean{Journal}}{#1}{#2}}

\newcommand{\ubar}[1]{\bar{#1}}
\usepackage{algorithm,algpseudocode}

\newcommand{\XinitialSet}{X_\circ}
\usepackage[normalem]{ulem}

\begin{document}
%
\title{Model Predictive Control of Hybrid Dynamical Systems}
%
%

\author{Ricardo~G.~Sanfelice,~\IEEEmembership{Fellow,~IEEE}, and~Berk~Alt{\i}n,~\IEEEmembership{Member,~IEEE}
\thanks{R. G. Sanfelice and B. Alt{\i}n are with the Department
of Electrical and Computer Engineering, University of California, Santa Cruz,
CA, 94107 USA e-mail: \texttt{ricardo@ucsc.edu}, \texttt{berkaltin@ucsc.edu}. 
Research funded in part by Research by R. G. Sanfelice partially supported by NSF Grants no. CNS-2039054 and CNS-2111688, by AFOSR Grants nos. FA9550-23-1-0145, FA9550-23-1-0313, and FA9550-23-1-0678, by AFRL Grant nos. FA8651-22-1-0017 and FA8651-23-1-0004, by ARO Grant no. W911NF-20-1-0253, by DoD Grant no. W911NF-23-1-0158, and by UC Alianza MX Grant no. SRPUCMX24-01.}
}

%
%

\markboth{IEEE Transactions on Automatic Control,~Vol.~0, No.~0, July~2026}%
{Alt{\i}n and Sanfelice: Model Predictive Control of Hybrid Dynamical Systems}
%



\maketitle

\begin{abstract}
The problem of controlling hybrid dynamical systems
using model predictive control (MPC) is formulated and sufficient conditions
for asymptotic stability of a set are provided.
Hybrid dynamical systems are modeled in terms of hybrid equations, involving
a differential equation and a difference equation with inputs and constraints.
The proposed hybrid MPC algorithm uses a suitable prediction and control horizon construction inspired by hybrid time domains.
Structural properties of the hybrid optimization problem,
its feasible set, and its value function are provided. 
Checkable conditions to guarantee asymptotic stability of a set are provided. 
These conditions are given in terms of 
properties on the stage cost, terminal cost, and the existence of static state-feedback laws, related through a control Lyapunov function condition.  Examples illustrate the results throughout the paper.
\end{abstract} 


%
\IEEEpeerreviewmaketitle


\section{Introduction}
\label{sec:intro}


\subsection{Background and Motivation}

Model Predictive Control (MPC) has established itself as a versatile and powerful feedback strategy for constrained control of dynamic systems, with extensive application across industries.
Since its conception as a solution to control problems in emerging industry back in the late 1970s/early 1980s \cite{Richalet.ea.76.IFAC,Cutler.Ramaker.80.ACC},
the iterative, optimization-based algorithmic structure of MPC has appealed to practitioners aiming at solving  myriad of control problems.\footnote{Some of the history of MPC is reported in \cite{Morari.Lee.99.CCE} and \cite{Qin.Badgwell.03.CEP}.}
MPC is a receding horizon control scheme that 
optimizes the selection of the control inputs
through the use of a dynamic model of the process to predict the effect of future control actions.
Without attempting to survey such a broad literature,
MPC schemes for linear and nonlinear systems have undergone substantial development, leading to widespread adoption and rigorous theoretical foundations~\cite{Rawlings.Mayne.09,Borrelli.ea.17,GrunePannek2017}. Early continuous-time MPC strategies, such as~\cite{Mayne.ea.90.TAC,Chen.Allgower.98.Automatica}, illustrate how continuous-time dynamics can be incorporated into  MPC, often through discretization or continuous-time optimization formulations.

Within the broader context of MPC, the term ``hybrid MPC'' typically refers to control of systems that feature both continuous and discrete elements, such as logic-based constraints or mixed-valued states and inputs,
or that have a discontinuous right-hand side due to the process dynamics or the control algorithm being nonsmooth~\cite{maynesurvey,camacho, Borrelli.ea.17}. 
A well-known foundational framework for such systems was developed in~\cite{BemporadMorari1999}, where the authors propose mixed logical dynamical (MLD) models and corresponding mixed-integer MPC formulations. 
These methods address discrete modes and logic rules, and have influenced a large body of work on hybrid MPC~\cite{Bemporad2002,Bemporad2006,Borrelli.ea.17,camacho,maynesurvey}.  
However, following \cite{154}, the use of the term ``hybrid'' in this paper aligns better with the definition introduced in the hybrid systems community \cite{vanderSchaftSchumacher00,AubinLygerosQuincampoixSastrySeube02,65,220}. 
These models capture hybrid behavior beyond discrete logic and variables, including impulsive effects, resets, switching, and asynchronous updates, which arise naturally in applications such as networked systems, robotics, and cyber-physical systems.

Several works have addressed model predictive control for classes of systems exhibiting hybrid behaviors, with emphasis on switched and impulsive dynamics. For example, \cite{mhaskar2008robust}  presents a robust MPC formulation for switched systems under uncertain switching schedules, ensuring constraint satisfaction through conservative design techniques. In \cite{sopasakis}, MPC schemes for linear impulsive systems, addressing the discrete-time occurrence of impulses and demonstrating stability under specific assumptions, are developed. The work in \cite{pereira} proposes a general optimization-based control framework for impulsive systems, focusing on trajectory generation and feasibility. In \cite{muller2012model}, the authors  investigate switched nonlinear systems and derive conditions under which MPC guarantees stability assuming an average dwell-time condition. Similarly, \cite{colaneri2007robust} studies robust MPC for discrete-time switched systems, considering uncertainties and constraints. As a difference to settings based on MLD models, mixed-valued states and inputs, or nonsmooth dynamics, this paper employs the framework for hybrid control systems in \cite{220}, in which a hybrid plant is described by the interaction of continuous-time dynamics and discrete-time transitions governed by a unified set-theoretic formalism. 
Specifically, a {\em hybrid plant} is given by
\begin{equation}
	\HS : 
	\left\{
	\begin{aligned}
		\dot{x}		&= f(x,u)		& (x,u)&\in C\\
		x^+				&= g(x,u)		& (x,u)&\in D
	\end{aligned}
	\right.
	\label{eqn:Hp}
\end{equation}
where $x$ is the state, $u$ is the control input, $C$ is the flow set,
$f$ is the flow map,
$D$ is the jump set,
and 
$g$ is the jump map.
This formalism encapsulates numerous models of a hybrid nature already studied in the context of MPC, sample-and-hold control \cite{findeisen2006nonlinear,magni}, event/self-triggered control \cite{Tabuada.07}, and switching systems \cite{mhaskar2005predictive}, which, very importantly, as stated in \cite[Section 2.2.5]{maynesurvey}, ``it is bound to influence further research on hybrid MPC.'' 

The primary objective of this paper is to lay down the theoretical foundations of a MPC framework for hybrid dynamical systems, with the meaning of ``hybrid'' as clarified above. In contrast to the aforementioned works, the present approach directly incorporates hybrid time domains and trajectory evolution into the prediction and constraint structure, allowing for general, continuous-time hybrid dynamics beyond piecewise-affine, impulsive, or switched-only systems.
Although modern MPC developments, such as economic and dissipativity-based MPC, turnpike phenomena, and formulations without terminal ingredients, offer powerful alternatives in continuous settings, their extension to hybrid systems is nontrivial and remains an open research direction. This paper focuses on transferring and formalizing early MPC principles in the hybrid context, providing a foundation for future integration of these newer techniques.

\subsection{Contributions}

Although the concept of applying MPC to hybrid systems has existed in literature for several decades, an MPC framework for hybrid plants of the form \eqref{eqn:Hp} is not yet available. 
This paper provides a comprehensive presentation of  
hybrid MPC for such systems, and formally introduces the structural conditions and theoretical tools needed to ensure:
\begin{enumerate}
\item
Appropriate structure and properties of the hybrid prediction horizon;
\item Forward invariance of the terminal constraint set 
and of the feasible set (Proposition~\ref{prop:aroundA} and Proposition~\ref{prop:T1T2});
\item 
That the value function serves as a Lyapunov function via bounds in terms of the terminal and stage costs (Lemma~\ref{lem:valuecont} and Lemma~\ref{lem:descent});
\item 
Positive
definiteness of the value function (multiple results in Section~\ref{sec:posdef} and Section~\ref{sec:posdef-Sufficient});
\item Asymptotic stability of a closed set via hybrid MPC, including special cases when jumps or flows are persistent (Theorem~\ref{thm:AS-MPC});
\item Two examples showing in detail how to apply the results.
\end{enumerate}
The results are illustrated in two examples throughout the paper, one pertaining to control of a bouncing ball and another one about sample-and-hold control.  More details and an additional example can be found in \cite{HybridMPCTR}. The overall framework was first outlined in the conference paper \cite{179}, which 
informally presents an initial version of the framework presented here but
does not include any mathematical result.
The conference paper \cite{193} presents a preliminary version of the
basic assumptions. These assumptions are relaxed here; in particular, the set to asymptotically stabilize does no longer
need to be compact, expanding the applicability of our results to a broader class of 
stabilization problems (e.g., consensus, tracking, and estimation),
and a condition relating the terminal constraint set and the set to asymptotically
stabilize is removed (see (O4) in \cite{193}).  
A result revealing properties of the value function was stated without proof in \cite[Lemma 5.4 and Lemma 5.5]{193}, but significantly improved here, under weaker assumptions (cf. Lemma~\ref{lem:valuecont} and Lemma~\ref{lem:descent}) and allowing the stage costs and the terminal cost to be not be positive definite. 
The latter results can be found in Section~\ref{sec:posdef} and Section~\ref{sec:posdef-Sufficient},
which further extend the relaxations in the conference paper \cite{228} -- in fact, only a version
of items~\ref{item2-pd}  and \ref{item3-pd} in Theorem~\ref{thm:PDofJ} are considered therein, 
which are further relaxed in those sections.  Note that \cite{193} and \cite{228} 
do not include i) any mathematical proofs, ii) the intermediate results required to establish the properties of the value function and asymptotic stability, and iii) the many illustrations via examples, given by two examples revisited throughout the paper multiple times.

\subsection{Organization of the Paper}

\IfJournal{
Section~\ref{sec:background} introduces notation and basic concepts of
 hybrid systems. 
Section~\ref{sec:overview} presents an overview of hybrid MPC and its main ingredients and Section~\ref{sec:assumptions} introduces required assumptions.
Section~\ref{sec:ocp} reveals the main properties of the 
constrained optimal control problem associated with hybrid MPC,
Section~\ref{sec:posdef} characterizes properties of the value function, and 
Section~\ref{sec:stability} establishes asymptotic stability.
Section~\ref{sec:examples} presents implementations of hybrid MPC in examples.
}
{
Section~\ref{sec:background} introduces notation and basic concepts of
 hybrid systems. 
Section~\ref{sec:overview} presents an overview of hybrid MPC and its main ingredients: 
the cost functional, the prediction horizon, the constrained optimal control problem,
and the recursive implementation of hybrid MPC.
The assumptions required for hybrid MPC to be asymptotically stabilizing
are presented in Section~\ref{sec:assumptions}.
Section~\ref{sec:ocp} reveals the main properties of the 
constrained optimal control problem associated with hybrid MPC.
Section~\ref{sec:posdef} characterizes properties guaranteeing positive definiteness of the value function.
Section~\ref{sec:stability} formalizes the asymptotic stabilization properties of hybrid MPC.
Section~\ref{sec:examples} presents  examples to which the main results
are applied.
Open problems and future work are discussed in Section~\ref{sec:conclusion}.
}

\section{Background on Hybrid Control Systems}
\label{sec:background}

Throughout the paper, we use~$\reals$ to represent real numbers and~$\realsgeq$ its nonnegative subset. The set of nonnegative integers is denoted~$\naturals$. The notation~$S_1\subset S_2$ indicates~$S_1$ is a subset of~$S_2$, not necessarily proper. The 2-norm is denoted~$|\cdot|$. The distance of a vector~$x\in\reals^n$ to a nonempty set~$\A\subset\reals^n$ is denoted~${\Anorm{x}:=\inf_{a\in\A}|x-a|}$. We denote by~$\A+\delta\ball$ the set of all~$x\in\reals^n$ such that $|x-a|\leq\delta$ for some $a\in\A$. The interior and closure of a set~${S\subset\reals^n}$ are denoted~${\interior S}$ and~${\closure S}$, respectively. We denote by~$\Pi:\reals^{n}\times\reals^{m}\to\reals^{n}$ the standard projection onto~$\reals^{n}$ such that~$\Pi(x,y)=x$. A strictly increasing continuous function~$\alpha:\realsgeq\to\realsgeq$ with $\alpha(0)=0$ is said to be a class-$\classK$ function. An unbounded class-$\classK$ function is said to be a class-$\classK_{\infty}$ function.

\subsection{Hybrid Control Systems}

This paper focuses on the control of hybrid plants given as $\HS$
 in Section~\ref{sec:intro},
where~$x\in\reals^n$ is the state and~$u\in\reals^m$ is the control input. The data of the hybrid plant $\HS$ is given by\begin{itemize}
\item the \textit{flow set}~$C\subset\reals^n\times\reals^m$; 
\item the \textit{flow map}~${f:C\to\reals^n}$ describes the continuous evolution of~$x$ when~$(x,u)$ belongs to  $C$;
\item the \textit{jump set}~${D\subset\reals^n\times\reals^m}$;
\item the \textit{jump map}~$g:D\to\reals^n$ describes the discrete evolution of~$x$ when~$(x,u)$ belongs to $D$.
\end{itemize}
At times, the data of $\HS$ is explicitly denoted as
$\HS = (C,f,D,g)$. 
Furthermore, when convenient, the input $u$ is partitioned as $(u_c,u_d)$, where 
$u_c$ collects the components of $u$ that affect the flows and $u_d$ the components that affect jumps.

To simplify the formulation and development of the forthcoming results, the following mild standing assumption on the flow set $C$ is imposed.  
\begin{assumption}[Standing Assumption]
The set~$C$ is closed.
\end{assumption}
In particular, as clarified below, assuming $C$ to be closed simplifies
the notion of solution pairs. For example, with $C$ closed,
solutions are defined so that the state trajectory component of each
solution stays in the projection of $C$ onto $\reals^n$ over intervals of ordinary time,
except potentially
at the initial and end times of each such interval.  It also simplifies the conditions that the pair needs to satisfy at the initial (hybrid) time.

\begin{example}[Bouncing Ball Control]
\label{ex:bouncingball}
Consider the hybrid model of a ball bouncing vertically on a horizontal flat surface with height~$x_1$ and velocity~$x_2$. When~$x_1\geq 0$, the motion of the (point-mass) ball is represented by the second order differential equation
\begin{equation}
	\dot{x}_1=x_2,\qquad  \dot{x}_2=-\gamma
\label{eq:freefall}
\end{equation}
where~$\gamma>0$ is the gravitational constant. Impacts with the surface are assumed to occur instantaneously, namely, 
the compression of the ball at collisions is neglected
and the exchange of kinetic energy between the ball and the ground takes zero time.  This behavior is conveniently captured by the
difference equation
\begin{equation}
	x_1^+=x_1,\qquad  x_2^+=-\lambda x_2+u
\label{eq:bounce}
\end{equation}
where $\lambda\in[0,1]$ is the coefficient of restitution
and 
${u\geq 0}$ is a control input that affects the velocity after impacts.
Such an instantaneous change occurs when~$x_1=0$ and $x_2\leq 0$. 
In the autonomous case with dissipative jumps~(i.e., when~$u=0$ and~$\lambda \in [0,1)$), asymptotic stability of the origin can be certified using the total energy function~
\begin{equation}\label{eqn:BBenergy}
W(x):=\gamma x_1+\frac{x_2^2}{2}\qquad \forall x \in \reals^2
\end{equation}
 Although~$W$ is constant during flows, it decreases at each impact away from the origin. 
The dynamics can be represented in the form of~\eqref{eqn:Hp} by incorporating the constraints therein to~\eqref{eq:freefall}-\eqref{eq:bounce}. 
This system can be written as a hybrid plant $\HS$ in \eqref{eqn:Hp}, as follows:
\begin{equation}\nonumber
	\HS : 
	\left\{
	\begin{aligned}
		\matt{\dot{x}_1\\ \dot{x}_2}		&= f(x,u) := \matt{x_2\\ - \gamma}		& (x,u)&\in C\\
		\matt{x_1^+ \\x_2^+}				&= g(x,u)	:= \matt{0 \\ -\lambda x_2+u}	& (x,u)&\in D
	\end{aligned}
	\right.
\end{equation}
with $C:=\{(x,u) \in \reals^2 \times \realsgeq : x_1\geq 0\}$
and
$D:=\{(x,u) \in  \reals^2 \times \realsgeq : x_1=0, x_2\leq 0, u\geq 0\}$.
Note that the flow map is defined on $C$ and the jump map is defined on $D$. \hfill $\bigtriangleup$
\end{example}

We refer to an input signal~$u$ and the corresponding state trajectory~$x$ collectively as a solution pair~$(x,u)$ to~$\HS$. A solution pair~$(x,u)$ is parametrized by~$(t,j)\in\realsgeq\times\nats$, where~$t$ denotes flow time and~$j$ is the number of jumps. The domain of~$x$, denoted~$\dom x$, is a \textit{hybrid time domain}.
A set $S \subset \realsgeq\times\nats$
is a {compact hybrid time domain} if
$$S = \bigcup_{j=0}^{J} 
\left([t_j,t_{j+1}]\times\{j\}\right)$$
for some $J \in \nats$ and
for some finite sequence 
of times $\{t_j\}_{j=0}^{J+1}$ satisfying $0=t_0\leq t_1 \leq t_2 \leq \ldots \leq t_J \leq t_{J+1}$,
where $t_j$ is the ordinary time of the~$j$-th jump.
Then, $\dom x \subset \realsgeq \times \nats$ is a hybrid time domain if it is the union of a nondecreasing sequence of compact hybrid time domains, namely, $\dom x$ is the
(possibly infinite)
 union of compact hybrid time domains $S_0, S_1, \ldots, S_i, \ldots$ with the property that 
the sequence $\{S_i\}_{i = 0}^{\overline{i}}$, $\overline{i} \in \nats \cup \{\infty\}$, satisfies
$S_0 \subset S_1 \subset S_2 \subset \ldots \subset S_i \subset \ldots$\ .
The domain of~$u$ is denoted~$\dom u$ in a similar fashion.  Since the state and input
pair needs to satisfy the dynamics imposed by the hybrid plant at the hybrid times
for which they are defined, the domain of the input and of the state are assumed to be the same. 
Under this assumption, for simplicity, we write $\dom (x,u)$
to denote the hybrid time domain of the pair $(x,u)$.

\begin{definition}[Solution Pair]
\label{defn:solpair}
Given a pair of functions~$x:\dom x\to\reals^n$ and~$u:\dom u\to\reals^m$,~$(x,u)$ is said to be a solution pair to~\eqref{eqn:Hp} if
$\dom(x,u):=\dom x=\dom u$ is a hybrid time domain and
the following hold:
\begin{itemize}
\item[(S0)]
$(x(0,0),u(0,0))\in C\cup D$,
	\item[(S1)] For each~$j \in \nats$ such that~$I^j:=\{t:(t,j)\in\dom(x,u)\}$ has a nonempty interior, 
	\begin{itemize}
\item[(S1.1)]	the function~${t\mapsto x(t,j)}$ is locally absolutely continuous,
\item[(S1.2)] the function~${t\mapsto u(t,j)}$ is Lebesgue measurable and locally essentially bounded, 
$t\in\interior I^j$, and
\item[(S1.3)] for almost all~$t\in I^j$,
\IfJournal{
	\begin{equation}
\hspace{-0.5in}(x(t,j),u(t,j)) \in C, \quad 
		\dot{x}(t,j)  = f(x(t,j),u(t,j))
	\label{eq:flowdyn}
	\end{equation}
	}{	\begin{equation}
			\begin{aligned}
(x(t,j),u(t,j)) & \in C, \\
		\dot{x}(t,j) & = f(x(t,j),u(t,j))
		\end{aligned}
	\label{eq:flowdyn}
	\end{equation}}
	\end{itemize}
	\item[(S2)]  For each~$(t,j)\in\dom(x,u)$ {such} that~$(t,j+1)\in\dom(x,u)$,
\IfJournal{
	\begin{equation}
		\label{eq:jumpdyn}
\hspace{-0.5in}(x(t,j),u(t,j))\in D, \quad
			x(t,j+1)		= g(x(t,j),u(t,j))
	\end{equation}
}{
	\begin{equation}
		\label{eq:jumpdyn}
		\begin{aligned}
			(x(t,j),u(t,j))&\in D,\\
			x(t,j+1)					&= g(x(t,j),u(t,j))
		\end{aligned}
	\end{equation}
	}
\end{itemize}
\end{definition}
	
\begin{remark}
The condition in (SO) assures that the initial condition for the state and for the initial value of input can lead to evolution forward in hybrid time.
Since $C$ is closed, the condition in (SO) does not involve a closure on $C$. 
When $C$ is not closed, 
it might be possible to have flow when $x(0,0)$ is on the boundary of the projection of $C$ to the state space, as long as $u$ belongs to the interior of the projection of $C$ to the input space after $(t,j) = (0,0)$. 
The condition in (S1.1) guarantees enough regularity of the state trajectory so that a derivative can be calculated during intervals of flow, at least, almost everywhere.
A function $x$ that is defined on a hybrid time domain and satisfies (S1.1)
is called a {\em hybrid arc}.
The condition in (S1.2) on the input $u$ further assures integrability during flows.
A function $u$ that is defined on a hybrid time domain and satisfies (S1.2)
is called a {\em hybrid input}.
Closedness of $C$ also permits simplifying the two properties in (S1.3), by requiring them to both hold for almost all $t$, rather than requiring
that $(x(t,j),u(t,j))\in C$ holds for all $t$'s in the interior of $I^j$.
Note that a solution $(x,u)$ might belong to both $C$ and $D$,
in which case either flow or jump might be possible. 
However, when the input $u$ is given, flow intervals and jump times
are determined by its hybrid time domain $\dom u$.
\end{remark}

Throughout the paper, the set of solution pairs to~$\HS$ starting from the set~$\XinitialSet\subset\reals^n$ is denoted~$\widehat{\sol}_{\HS}(\XinitialSet)$. That is,~$(x,u)\in\widehat{\sol}_{\HS}(\XinitialSet)$ implies~$x(0,0)\in \XinitialSet$. Given a solution pair~$(x,u)$,~$(T,J)\in \dom(x,u)$ is said to be the terminal (hybrid) time of~$(x,u)$ if~$T\geq t$ and~$J\geq j$ for all~$(t,j)\in \dom(x,u)$. 
A truncation of a solution pair~$(x,u)$ up to some $(T,J)\in \dom(x,u)$
is given by a solution $(x',u')$ that is equal to $(x,u)$
with domain $\dom (x',u')$ given by
\begin{equation}\label{eqn:GenJumpTimes}
\dom (x,u) \cap ([0,T]\times\{0,1,\ldots,J\}) =: \bigcup_{i=0}^{J} ([s_i,s_{i+1}]\times\{i\})
\end{equation}
for some nondecreasing sequence $\{s_i\}_{i=0}^{J+1}$, which is said to be the sequence of {\em generalized jump times} associated with the truncation $(x',u')$.
Furthermore, a solution pair~$(x,u)$ is said to be
\IfJournal{
{\em maximal} if it cannot be further extended;
{\em complete} if~$\dom(x,u)$ is unbounded;
and to have a 
{\em compact domain} if $\dom(x,u)$ is a compact subset of $\realsgeq \times \nats$ with a hybrid time domain structure, namely, with a domain $\dom(x,u)$ equal to ${\bigcup}_{j=0}^J\left([t_j,t_{j+1}]\times\{j\}\right)$ for some finite nondecreasing sequence~$\{t_j\}_{j=0}^{J+1}$ with $t_0=0$ and  $J \in \nats$.
}{
\begin{itemize}
\item {\em maximal} if it cannot be further extended;
\item {\em complete} if~$\dom(x,u)$ is unbounded;
\end{itemize}
and to have a 
\begin{itemize}
\item 
{\em compact domain} if $\dom(x,u)$ is a compact subset of $\realsgeq \times \nats$ with a hybrid time domain structure, namely, with a domain $\dom(x,u)$ equal to ${\bigcup}_{j=0}^J\left([t_j,t_{j+1}]\times\{j\}\right)$ for some finite nondecreasing sequence~$\{t_j\}_{j=0}^{J+1}$ with $t_0=0$ and  $J \in \nats$.
\end{itemize}
}
In addition, a complete solution pair~$(x,u)$ is said to have 
\begin{itemize}
\item {\em persistent flows} if $\dom(x,u)$ is unbounded in the~$t$ direction---that is, there exists no~$T\in\realsgeq$ such that~$t\leq T$ for all~$(t,j)\in \dom(x,u)$;
	\item {\em persistent jumps} if $\dom(x,u)$ is unbounded in the~$j$ direction---that is, there exists no~$J\in\nats$ such that~$j\leq J$ for all~$(t,j)\in \dom(x,u)$.
\end{itemize}
For simplicity, at times we use $(x,u)$ to denote a solution pair or points in $\reals^n\times\reals^m$, 
explicitly making it clear what is meant by saying ``solution'' or ``points'' when not clear from context.

\begin{example}[Bouncing Ball Control (revisited)]
\label{ex:bouncingball-revisited}
For the bouncing ball with a control input in Example~\ref{ex:bouncingball}, the amount of ordinary time needed to exhibit an impact (or jump) from any given initial state~$x=(x_1,x_2)$ with~$x_1\geq 0$ is finite.
Hence, every complete solution pair to this hybrid plant has persistent jumps.  
For a complete characterization of its solutions, see 
\cite[Example 2.32]{220} and \cite[Example~2.12]{65}. 
Due to persistence of jumps, the total energy $W$ in \eqref{eqn:BBenergy} serves as a Lyapunov function for the bouncing ball system in Example~\ref{ex:bouncingball} with $u \equiv 0$~\cite[Example~3.15]{65}.
\end{example}

\begin{example}[Sample-and-Hold Control]
\label{ex:sample}
Consider a continuous-time control system with state~$z\in\reals^{n_z}$, control input~$\eta\in\reals^{m_z}$, and dynamics $\dot z = \tilde{f}(z,\eta)$, where $\tilde{f}:\reals^{n_z}\times\reals^{m_z}\to\reals^{n_z}$ is the right-hand side. When the input~$\eta$ of the plant is updated periodically using a zero-order hold mechanism (e.g., a digital-to-analog converter), the resulting control system can be written as in~\eqref{eqn:Hp} by treating~$\eta$ as a state component and by introducing a clock variable~$\tau_s\in\realsgeq$ that updates $\eta$ every $T_s$ seconds. 
Let~$x=(z,\eta,\tau_s) \in \reals^{n_z}\times\reals^{m_z}\times \realsgeq$ be the state 
and $u$ the input
of the system capturing the dynamics of the system being controlled as well as the sample-and-hold mechanism updating $\eta$. 
In contrast with the continuous-time approach in \cite{findeisen2006nonlinear},
we capture the dynamics by a hybrid plant 
$\HS$ as in \eqref{eqn:Hp} with the following data:
\begin{itemize}
\item the flow map is given by $f(x,u):=(\tilde{f}(z,\eta),0,1)$;
\item the flow set is given by $C:=\{(x,u) \in \reals^{n_z}\times\reals^{m_z}\times \realsgeq \times \realsgeq:\tau_s\in[0,T_s]\}$,
\item the jump map is given by $g(x,u):=(z,u,0)$;
\item the jump set is given by $D:=\{(x,u) \in \reals^{n_z}\times\reals^{m_z}\times \realsgeq\times \realsgeq:\tau_s=T_s\}$.
\end{itemize}
This data accomplishes the following. During flows, the definition of the flow map $f$ is such that the timer~$\tau_s$ evolves continuously with a constant rate of one, so as to reproduce ordinary time, $\eta$ remains constant, and~$z$ evolves according to the dynamics of the continuous-time system being controlled by the constant input $\eta$. When~$\tau_s$ reaches~$T_s$, which is captured by the definition of the jump set $D$, the definition of the jump map $g$ indicates that $\tau_s$ gets reset to zero and $\eta$ is updated to the value of the input~$u$ of the hybrid system at the time of the event.  At jumps, the state of the system being controlled remains constant, hence the first component of $g$ is equal to $z$.  Note that the input $u$ to the hybrid system is yet to be assigned.
In the forthcoming Section~\ref{sec:examples}, the input $u$ is assigned by the hybrid MPC algorithm presented in the next section.
If the input $u$ is given by a continuous-time signal defined for each $t \in \realsgeq$, complete solution pairs~$(x,u)$ of this system resulting from putting the continuous-time input and the state trajectory on the same hybrid time domain\footnote{See \cite{52,209} for more details.} have persistent flows and jumps, as the $j$-th jump occurs at ordinary time~$t_j=T_sj-\tau_s(0,0)$, where $\tau_s(0,0)$ is the initial condition for the timer $\tau_s$.
If the input $u$ is given only at jump times, the same construction is possible
by extending the definition of the input over the intervals of flow, for example, as a piecewise constant signal.
 \hfill $\bigtriangleup$
\end{example}

A hybrid system~$\HS$ has solution pairs with unique state trajectories~$(t,j)\mapsto x(t,j)$ if two inputs that are identical, namely, they are equal during jumps and equal almost everywhere (in time) during flows, lead to  the same state trajectory from the same initial condition. To ensure uniqueness for~$\HS$, we impose the following assumption in our forthcoming result establishing asymptotic stability (Theorem~\ref{thm:AS-MPC}).  
\begin{assumption}
\label{assmp:uniqueness}
		Given the hybrid plant $\HS=(C,f,D,g)$ in \eqref{eqn:Hp}, the {constrained} differential equation 
\begin{equation}\label{eqn:CT}
\dot{x}=f(x,u)\quad (x,u)\in C 
\end{equation}
has unique state trajectories $t\mapsto x(t)$; 
that is, for each $\xinitial$, and each pair of inputs $t \mapsto u_1(t)$ and $t \mapsto u_2(t)$ with $\dom u_1 = \dom u_2$ such that $u_1(t) = u_2(t)$ for almost all $t \in \dom u_1$, and $(\xinitial,u_1(0)) \in C$ and $(\xinitial,u_2(0)) \in C$, 
the respective solution pairs $(x_1,u_1)$ and $(x_2,u_2)$ have
identical state trajectories over their common domain of definition.
\end{assumption}

The next result can be obtained using similar steps as in the proof of \cite[Proposition 2.11]{65}.

\begin{proposition}
\label{prop:uniqueness}
The hybrid control system~$\HS$ has unique state trajectories if and only if Assumption~\ref{assmp:uniqueness} holds.
\end{proposition}
\begin{IEEEproof}
Necessity is obvious. For sufficiency, consider two solution pairs~$(x,u)$ and~$(x',u')$ with the same initial condition and domain, where~$u$ and $u'$ are equivalent. Since jumps of~$x$ and~$x'$ occur at the same ordinary time, uniqueness follows from uniqueness during flows (by Assumption~\ref{assmp:uniqueness}) and uniqueness during jumps.
\end{IEEEproof}

For  hybrid dynamical systems with single-valued flow and jump maps and no inputs (i.e., closed hybrid systems), uniqueness of solutions is guaranteed when flows are not possible on the intersection of the flow and jump sets, and the constrained differential equation
\eqref{eqn:CT}
has unique solutions~\cite[Proposition~2.11]{65} and \cite[Proposition~2.34]{220}. In contrast, Assumption~\ref{assmp:uniqueness} states that uniqueness during flows is necessary and sufficient for uniqueness of state trajectories. This property is due to the fact that given any two solution pairs~$(x,u)$ and~$(x',u)$, their jump times are determined by the domain of the input; hence, by Definition~\ref{defn:solpair}, $\dom(x,u)=\dom(x',u)=\dom u$.

\subsection{Hybrid Control Systems under Static State-Feedback}
\label{sec:feedback}

For analysis purposes, we also study the hybrid closed-loop system arising from 
assigning the input $u = (u_c,u_d)$ to static state-feedback laws $\kappa_C:\reals^n\to \reals^m$ and $\kappa_D:\reals^n\to \reals^m$ for the hybrid plant $\HS$ in \eqref{eqn:Hp}. With this input partition, the resulting hybrid closed-loop system, denoted~$\HS_{\kappa}$, is given by 
\begin{equation}
\HS_{\kappa} :
	\left\{
	\begin{aligned}
		\dot{x}		&= f_{\kappa}(x):=f(x,\kappa_C(x))		& x&\in C_{\kappa}\\
		x^+				&= g_{\kappa}(x):=g(x,\kappa_D(x))		& x&\in D_{\kappa}
	\end{aligned}
	\right.
	\label{eq:Hk}
\end{equation}
{where}
\begin{equation*}
\label{eq:CkDk}
	\begin{aligned}
		C_{\kappa}&:=\{x\in\reals^n:(x,\kappa_C(x))\in C\},\\
		D_{\kappa}&:=\{x\in\reals^n:(x,\kappa_D(x))\in D\}
	\end{aligned}
\end{equation*}
Similar to $\HS$ in \eqref{eqn:Hp}, we define solutions to~\eqref{eq:Hk} over hybrid time domains via Definition~\ref{defn:solpair}, as follows:
a hybrid arc-input pair $(x,u)$ is said to be {\em a solution pair to \eqref{eq:Hk} generated by the feedback~$\kappa$} if there exists a solution pair~$(x,u)$ to~$\HS$ that satisfies~\eqref{eq:flowdyn} with~$u(t,j)=\kappa_C(x(t,j))$ for all~$t\in \interior{I^j}$ and~\eqref{eq:jumpdyn} with~$u(t,j)=\kappa_D(x(t,j))$
for all $(t,j)\in \dom (x,u)$ such that $(t,j+1)\in \dom (x,u)$.
At times, we say that the $x$ component of such solution pair is a state trajectory of~\eqref{eq:Hk}.

\section{Hybrid MPC: Key Elements and Algorithm}
\label{sec:overview}

In this section, we introduce the main elements required to formulate
an MPC algorithm to asymptotically stabilize the hybrid plant $\HS$ in \eqref{eqn:Hp}.

\subsection{Overview}
\label{sec:HybridMPCOverview}

In simple words, the algorithm measures the current state of $\HS$ and finds an optimal solution pair to~$\HS$ over a finite horizon in hybrid time.
The optimization problem consists of minimizing a cost functional, subject to constraints on the terminal time and terminal state. 
The cost functional assesses the cost of the evolution of the state and of the possible inputs to be applied, both over intervals of flows and at time instances at which jumps occur.
The cost functional imposes a cost on the terminal state, relative to a desired {\em target set}. 
This optimization is performed over a {\em hybrid prediction horizon}, and the terminal hybrid time is allowed to vary within a set rather than being restricted to a point, as done, for instance, in discrete-time MPC where the optimization involves a fixed prediction horizon.
This optimization approach is reminiscent of free end-time optimal control.
It is instrumental for MPC for hybrid systems as the hybrid time domain of an optimal solution pair is not known a priori.
Each time the state is measured, the algorithm finds an optimal control hybrid input that is applied to~$\HS$ until the next measurement, leading to a receding horizon implementation. The initial optimization occurs at the initial time~$(t,j)=(0,0)$.  We refer to this algorithm as {\em hybrid MPC}.

There are three main differences between conventional continuous-time/discrete-time MPC, in their general formulations, and hybrid MPC:
\begin{itemize}
	\item Since hybrid control systems can have solution pairs from nearby initial conditions with drastically different hybrid time domains, the terminal time has to be allowed to vary within a set.  This set is the hybrid prediction horizon introduced above.
	\item To account for the differences in the hybrid time domains of optimal control hybrid inputs, the optimization times are not assumed to be periodic or known a priori. Instead, with the initial optimization being performed at time~$(0,0)$,  each subsequent optimization hybrid time is selected online and the re-optimization occurs 
before the current hybrid control input is exhausted.
	\item Since hybrid plants typically involve timer states, logic variables, and memory states, as Example~\ref{ex:sample} illustrates, the problem that hybrid MPC has to solve may be a set stabilization problem rather than a set-point stabilization problem. 
	In fact, the stabilization problem associated with  Example~\ref{ex:sample}  consists of asymptotic stabilizing the state $z$ to zero, $\eta$ to some compact set, and $\tau_s$ to the range $[0,T_s]$.
\end{itemize}

It should be noted that the results in this paper do not involve
discretization of the differential equation governing the flow of $\HS$, since the focus is on theoretical properties of MPC for hybrid systems.  \NotForJournal{A discussion about discretization is in the forthcoming Remark~\ref{remark:AboutDiscretization}.}

\NotForJournal{
\medskip
In the remainder of this section, we detail the formulation of the key elements involved in the formulation of the constrained optimization problem solved by hybrid MPC. }

\subsection{Hybrid Prediction Horizon}

To allow for the prediction of state trajectories that may flow or jump
without knowing the behavior a priori,
 we propose a prediction horizon $\T$ that is a subset of $\realsgeq\times\nats$.
This set-based generalization of the usual notion of prediction horizon accounts for hybrid time domains and maximizes the set of initial conditions such that the optimal control problem (OCP) associated to hybrid MPC is feasible. We refer to this construction as \textit{the hybrid prediction horizon}.
\NotForJournal{To motivate it, we revisit Example~\ref{ex:bouncingball}.

\begin{example}[Bouncing Ball Control (revisited)]
\label{ex:bouncingball-revisited2}
For the bouncing ball model in Example~\ref{ex:bouncingball} (see also Example~\ref{ex:bouncingball-revisited}), any open-loop hybrid input that steers the state to the origin necessarily leads to trajectories with the time between consecutive jumps converging to zero.\footnote{Given a complete solution pair~$(x,u)$ to~\eqref{eq:freefall}-\eqref{eq:bounce}, let~$t_{j+1}$ be the~$(j+1)$th jump time, i.e.,~$(t_{j+1},j),(t_{j+1},j+1)\in\dom(x,u)$ for all~$j\in\nats$. Then,~$t_{j+1}-t_j=\left(x_2(t_j,j)+\sqrt{(x_2(t_j,j))^2+2\gamma x_1(t_j,j)}\right)/\gamma$ for all~$j\in\nats$~\cite[Example 2.12]{65}, where~$t_0:=0$, as the flow map does not depend on the input. Consequently,~$\lim_{j\to\infty}t_{j+1}-t_j=0$ when~$\lim_{j\to\infty}|x(t_j,j)|=0$.} The implication of this property is that it would be impossible to asymptotically stabilize the origin of the bouncing ball by hybrid MPC when the associated OCP is restricted to a fixed prediction horizon of the form~$(T,J)$, no matter how~$T \in \realsgeq$ and~$J\in\nats$ are chosen.  In fact:
\begin{enumerate}
\item Since flows are not possible from the origin, given any solution pair~$(x,u)$ with compact domain and terminal time~$(T,J)$, since $T>0, $there exists~$(t,j)\in\dom(x,u)$ such that~$x(t,j)\neq(0,0)$. 
As a consequence, prediction horizons of the form $(T,J)$ with $T > 0$ are not suitable for asymptotic stabilization of the origin.
\item If the prediction horizon were to be fixed at~$(0,J)$ for some~$J\in\natspplus$ (that is, $T=0$) the OCP associated to hybrid MPC would only allow jumps and, therefore, be unsolvable from any initial condition with nonzero position or with positive velocity---namely, for every initial condition from the set~$C\backslash D$, where~$C$ and~$D$ are given in Example~\ref{ex:bouncingball}.  Again, this choice of prediction horizon would prohibit asymptotic stabilization of the origin.
\end{enumerate} 
\hfill $\triangle$
\end{example}

As highlighted in Example~\ref{ex:bouncingball-revisited}, the notion of a prediction horizon given by a singleton~$\{(T,J)\}$ can be overly restrictive and prevent a reasonable formulation of MPC for hybrid systems in the form~\eqref{eqn:Hp}.}\IfJournal{\newline}{} To maximize feasibility of the OCP, an appropriate selection of the prediction horizon~$\T$ should ensure that it intersects with unbounded hybrid time domains.
To this end, we define the hybrid prediction horizon as a set with  a specific geometry that overcomes this issue and maximizes feasibility.

\begin{definition}[Hybrid Prediction Horizon]
\label{def:Tstructure}
A set 
$\T \subset \realsgeq\times\nats$ is a hybrid prediction horizon if 
there exists {a finite nonincreasing sequence~$\{t_j\}_{j=0}^{J+1}$ such that $t_0>0$, $J\geq0$,~$t_{J+1}=0$}, and
\begin{equation*}
\label{eq:htdstructure}
	\T:=\bigcup_{j=0}^{J} \left([t_{j+1},t_{j}]\times\{j\}\right)
\end{equation*}
\end{definition}

This prediction structure guarantees that, if a solution pair $(x,u)$ ``lasts long enough,''\footnote{By the ordering induced by the structure of hybrid time domains, the length of elapsed hybrid time $(t,j)$ is succinctly captured by $t+j$.} then its hybrid time ``reaches''~$\T$---meaning that there exists $(t,j) \in \dom (x,u)$ such that $(t,j) \in \T$. This property is 
used to prove recursive feasibility in the forthcoming Section~\ref{sec:FIofX}.
\NotForJournal{
For instance, the choice $J= 1$ with
$t_0 = 2$, $t_1 = 1$, and $t_2 = 0$ leads to the hybrid prediction horizon
${\cal T} = ([0,1]\times\{1\}) \cup ([1,2]\times\{0\})$, which assures that predictions have either
at least $1$ second of flow or one jump.  This hybrid prediction horizon is depicted in Figure~\ref{fig:HybridPredictionHorizon}.

\begin{figure}[!ht]
	\centering
\hspace{-0.3in}	\includegraphics[scale=.7]{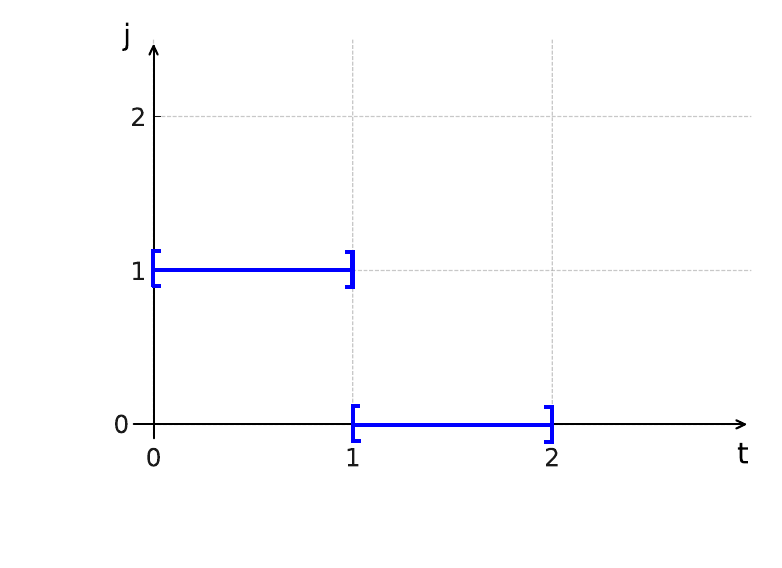}
	\vspace{-0.5in}
\caption{Hybrid prediction horizon for the case of~$J= 1$, $t_0 = 2$, $t_1 = 1$, and $t_2 = 0$.}
	\label{fig:HybridPredictionHorizon}
\end{figure}
}

\begin{remark}[On Particular Constructions of $\T$]
A natural construction of $\T$ is one that independently limits the amount of flow and the number of jumps.
Similar to discrete time MPC, the parameter $N \in \{1,2,\ldots\}$ 
denotes the maximum number of jumps allowed in the prediction.
To limit the amount of flow allowed, let $\delta > 0$ and 
define $\delta N$ as the flow time bound for prediction.
The construction of $\T$ given by
\begin{equation}
	\T:=\{(T,J)\in\realsgeq\times\nats:\max\{T/\delta,J\}=N\}
\label{eq:genericT}
\end{equation}
Note that the set $\T$ collects pairs $(T,J)$ in $\realsgeq\times\nats$ with $T$ no larger than $\delta N$ seconds and $J$ no larger than $N$ jumps, 
in this way, defining more than one possible hybrid time for prediction.
Another sample construction is given by
\begin{equation}
	\T:=\{(T,J)\in\realsgeq\times\nats:T+J\in[\mu,\mu+1]\}
\label{eq:muT}
\end{equation}
which, by properly choosing the parameter $\mu >0$, compactly captures prediction horizons allowing for arbitrary (but bounded) amount of flow or number of jumps.
\NotForJournal{In fact, given $\mu > 0$, the construction of $\cal T$ in \eqref{eq:muT} captures all pairs $(T,J)$ satisfying
        $$\mu \leq T + J \leq \mu +1$$
where $T \in \realsgeq$ and $J \in \nats$.
        Then, the set $\cal T$ in \eqref{eq:muT} can be interpreted as the region within $\realsgeq \times \nats$ satisfying the inequalities. For instance, for $J=0$, the set $\cal T$ collects the points $(T,0)$ with $T$ such that $\mu \leq T \leq \mu +1$, and so on.}
\end{remark}

\subsection{Terminal Constraint Set and Hybrid Cost Functional} 

The \textit{terminal constraint set}~$X$ specifies conditions that the state has to satisfy at the end of the hybrid prediction horizon. 
The set $X$ is assumed to be a subset of $\Pi(C \cup D)$, which 
captures the state values from which flow or jump might be possible.  Given a solution pair~$(x,u)$ of~$\HS$ with compact domain and terminal time~$(T,J)$, let~$\{t_j\}_{j=0}^{J+1}$ be the sequence such that $$\dom(x,u)=\bigcup_{j=0}^{J}([t_j,t_{j+1}]\times\{j\})$$ where~$t_{J+1}=T$. If~$x(T,J)\in X$, then the cost of the solution pair~$(x,u)$ is characterized by the
\textit{hybrid cost functional}~$\J$ defined as
\begin{equation}
\begin{aligned}
\label{eq:cost}
	\J(x,u)&:=\left(\sum_{j=0}^{J}\int_{t_j}^{t_{j+1}}L_{C}(x(t,j),u(t,j))\,dt\right) \\
	 &+\left(\sum_{j=0}^{J-1}L_{D}(x(t_{j+1},j),u(t_{j+1},j))\right)+ V(x(T,J))
\end{aligned}
\end{equation}
where $L_C:C\to\realsgeq$ is called the \textit{flow cost}, $L_D:D\to\realsgeq$ is called the \textit{jump cost}, and ${V:X\to\realsgeq}$ is called the \textit{terminal cost}.  Note that $\J$ is a functional as it depends on the solution pair $(x,u)$.

\begin{example}[Bouncing Ball Control (revisited)]
\label{ex:bouncingball-revisited-2}
Consider the hybrid model of the bouncing ball with a control input in
 Example~\ref{ex:bouncingball}, revisited in Example~\ref{ex:bouncingball-revisited}.
Since the input affects only the jumps, 
the flow cost $L_{C}$ can only depend on the state $x$. 
Suppose that the control goal is to asymptotically stabilize the bouncing ball to a
given energy level $c^* \geq 0$, where the total energy is given by $W$ in \eqref{eqn:BBenergy}. 
For this case, the flow cost and the terminal cost can be defined using the energy error captured by $W(x) - c^*$.
The jump cost $L_{D}$ can be chosen to depend on both the state and the
input. Explicit constructions of these functions are given in Example~\ref{ex:BBfinal} in Section~\ref{sec:examples}.
\end{example}

\subsection{Constrained OCP}

With the terminal constraint set~$X$ and the hybrid prediction horizon~$\T$ defined,  the minimization is performed over solution pairs of~$\HS$ with initial condition~$\xinitial$, terminal condition belonging to~$X$, and terminal hybrid time belonging to~$\T$.

\begin{problem}
\label{prob:optimum}
Given an initial condition~$\xinitial\in \reals^n$,
	\begin{equation}
	\label{eq:optimum}
		\begin{aligned}
			& \underset{(x,u)\in\widehat{\sol}_{\HS}(\xinitial)}{\text{minimize}}	& & \J(x,u)							\\
			& \text{subject to}																				& & (T,J)\in \T	\\
			& 																												& & x(T,J)\in X						\\
		\end{aligned}
	\end{equation}
\end{problem}

We say that a solution pair~$(x,u)$ is \textit{feasible} if it satisfies the constraints of~\eqref{eq:optimum} with~$x(0,0)=\xinitial$. A feasible solution may neither be maximal nor complete.  If, in addition, $(x,u)$ minimizes~$\J$, then it is said to be \textit{optimal}. The {feasible} set~$\X$ is the set of all~$\xinitial$'s with a feasible~$(x,u)\in\widehat{\sol}_{\HS}(\xinitial)$. The \textit{value function}~$\J^*:\X\to\realsgeq$ is defined as the infimum over all feasible solution pairs at a given initial condition~$\xinitial$, i.e.,
\begin{equation}
	\J^*(\xinitial):=\inf_{\substack{(x,u)\in\widehat{\sol}_{\HS}(\xinitial)\\ (T,J)\in \T\\ x(T,J)\in X}}\J(x,u) \quad \forall \xinitial\in \X
\label{eq:value}
\end{equation}
Recall that~$(T,J)$ is the terminal time of~$(x,u)$ and that $\widehat{\sol}_{\HS}(\xinitial)$ is the set of solution pairs to~$\HS$ in \eqref{eqn:Hp}  from $\xinitial$.

\begin{remark}[On Explicit Constraints]
Note that while the flow set~$C$ and jump set~$D$ implicitly define mixed state-input constraints, any additional explicit state and input constraints can be embedded in the problem by modifying $C$ and $D$ of the hybrid plant $\HS$. 
For example, the bouncing ball control system in Example~\ref{ex:bouncingball} with a maximum height constraint can be captured by redefining the flow set~$C$ therein, namely, 
$\{(x,u) \in \reals^2 \times \realsgeq : x_1\geq 0\}$,
as $C : = \{(x,u) \in \reals^2 \times \realsgeq : x_1\geq 0\} \cap \{(x,u):x_1 \leq h_{\max}\}$ for a given maximum height~$h_{\max}\geq 0$.
In fact, additional constraints to be satisfied by the closed-loop system solutions can be incorporated
by intersecting the flow set and jump set by sets capturing the new constraints.
Similarly, if state and input constraints on the continuous-time control system in Example~\ref{ex:sample} must to be satisfied, namely,
$z \in Z$ and $\eta \in U$
for some sets $Z \subset \reals^{n_p}$ and $U \subset \reals^{m_z}$,
then the flow and jump sets given in Example~\ref{ex:sample} can be simply intersected by $Z \times U$, and the resulting hybrid control system should be used in the optimization. In fact, solutions obtained with hybrid MPC, in order to exist, need to satisfy these constraints for all hybrid times---and if they are complete, these constraints are always satisfied and viable.
On the other hand, as already introduced, constraints on the terminal time and on the terminal state are specified explicitly, respectively, by a terminal constraint set and a hybrid prediction horizon set.
\end{remark}

\medskip
Summarizing, given the hybrid plant data $(C,f,D,g)$, Problem~\ref{prob:optimum}
depends on the following data:
\begin{itemize}
\item The hybrid prediction horizon $\T$;
\item The flow cost $L_{C}$, jump cost $L_{D}$, and terminal cost $V$;
\item The terminal constraint set $X$.
\end{itemize}
We refer to these as the data of the hybrid MPC problem
and denote it as $(\T,L_{C},L_{D},V,X)$.

\subsection{Receding Horizon Implementation}
\label{sec:Implementation}

We have defined all of the ingredients to formulate a receding horizon implementation of hybrid MPC.
At each iteration $i \in \nats$, the implementation performs the following operations: 
\begin{enumerate}[label={Step \arabic*)},leftmargin=*]
\item Measure the current state of the plant $\HS$;
\item Solve Problem~\ref{prob:optimum} to obtain an optimal solution pair $(x_i,u_i)$ from the current state;
\item Apply the optimal input $u_i$ to the plant $\HS$ to generate the state trajectory $x_i$ up to hybrid time $(T_{i+1},J_{i+1}) \in \dom (x_i,u_i)$;
\item Go to Step 1.
\end{enumerate}
The iteration starts from the initial condition $\xinitial$, 
generating the optimal solution pair $(x_0,u_0)$, which is applied up to $(T_1,J_1) \in \dom (x_0,u_0)$, 
and from where a new optimization problem is solved from $x_0(T_1,J_1)$, and so on.
Note that the hybrid time instances $(T_i,J_i)$ also indicate when the $i$-th optimization solving Problem~\ref{prob:optimum} is performed.
A {\em hybrid control horizon}, denoted $\T_c$, with the same structure as the hybrid prediction horizon $\T$ can be employed to regulate these time instances, in this way controlling the length of the input applied to the plant, within its domain of definition, which can be controlled by the hybrid prediction horizon.

This implementation generates a solution pair, extending the one in Definition~\ref{defn:solpair}.
\begin{definition}[Hybrid MPC Solution Pair]
\label{def:solgenMPC}
A solution pair~$(x,u)$ is said to be generated by the hybrid MPC algorithm if there exists a sequence~$\{(T_i,J_i)\}_{i=0}^{\infty}\in\dom(x,u)$ with~$(T_0,J_0)=(0,0)$ such that the following hold:
\begin{itemize}
	\item The sequence~$\{T_i+J_i\}_{i=0}^{\infty}$ is strictly increasing and unbounded.
	\item For every~$i\in\nats$, there exists an optimal solution pair~$(x_i,u_i)$, in the sense of Definition~\ref{defn:solpair}, such that for every~$(t,j)\in\dom(x,u)$ satisfying~$t+j\in[T_i+J_i,T_{i+1}+J_{i+1})$,
\IfJournal{
$$x(t,j)= x_i(t-T_i,j-J_i), \ u(t,j)= u_i(t-T_i,j-J_i)$$
	}{
		\begin{equation*}
		\begin{aligned}
			x(t,j)&= x_i(t-T_i,j-J_i)\\
			u(t,j)&= u_i(t-T_i,j-J_i)
		\end{aligned}
	\end{equation*}
}
\end{itemize}
\end{definition}

To illustrate the implementation above, consider the 
hybrid prediction horizon is given in \eqref{eq:genericT},
with parameters $N$ and $\delta$
and a hybrid control horizon with the same structure as the hybrid prediction horizon, parametrized by~$N_{c}\in\{1,2,\dots,N\}$ and~$\delta_c\in(0,\delta]$.
Algorithm~\ref{alg:HybridMPC} below 
describes, in concrete steps, the solution pair generated by hybrid MPC,
as formalized in Definition~\ref{def:solgenMPC}, for this particular case.\footnote{When the jump set is empty, Algorithm~\ref{alg:HybridMPC} simplifies to the standard periodic implementation of continuous-time MPC. Similar observations can be made when the flow set is empty, with~$N_{c}=1$ recovering the standard ``one-step ahead'' implementation in discrete-time MPC.}

\begin{algorithm}
\caption{Hybrid MPC Implementation}
	\begin{algorithmic}[1]
		\State Set $i = 0$.
		\State Set initial optimization time~$(T_0,J_0) =(0,0)$.
		\State Set the initial condition of the state of $\HS$ to $\xinitial$.
		\While{true}
			\State Solve Problem~\ref{prob:optimum} yielding optimal solution $(x_i,u_i)$.
			\While{$\max\left\{\frac{t-T_{i}}{\delta_c},j-J_{i}\right\}\leq N_c$}
				\State Apply $u_i$ to $\HS$ to generate the trajectory $x$.
			\EndWhile
			\State $i = i+1$.
			\State $(T_i,J_i) = (t,j)$, the end time of {\bf while} in line 6.
			\State $x_{\circ} = x(T_{i},J_{i})$.
		\EndWhile
	\end{algorithmic}\label{alg:HybridMPC}
\end{algorithm}
The receding horizon implementation of hybrid MPC is encoded by the inner while loop (Lines 6-8).
The initial optimization occurs at time~$(0,0)$ from the initial state $x_{\circ}$ (Lines 1 to 3 set the initial optimization time and the initial condition), and the initial optimal control input~$u_0$ is applied until~$\delta_c N_c$ units of ordinary time elapse or $N_c$ jumps occur,
 whichever happens first (Line 6). Upon any of these conditions, Problem~\ref{prob:optimum} is solved again to find the new input~$u_1$, which is again applied for~$\delta_c N_c$ units of ordinary time or until 
$N_c$ jumps occur.
 This process continues indefinitely, for each element in the input sequence.  For the $i$-th entry in this sequence, the portion of the state trajectory~$x$ from~$(T_i,J_i)$ to~$(T_{i+1},J_{i+1})$
 corresponds to the optimal state trajectory~$x_i$ computed at time~$(T_i,J_i)\in\dom(x,u)$ using input $u_i$. 
Note that due to the condition in Line~6,~$(T_{i+1}-T_i,J_{i+1}-J_i)$ is not necessarily constant as a function of $i$.

\NotForJournal{
A solution to Problem~\ref{prob:optimum} obtained using the implementation in Algorithm~\ref{alg:HybridMPC} is depicted in Figure~\ref{fig2}.
The parameters chosen are $N=4$,~$\delta=1$, $N_{c}=2$, and~$\delta_c=1$. 
The optimal control input~$u_0$ is applied until~$(T_1,J_1)=(2,1)$, which is two seconds after~$(T_0,J_0)=(0,0)$ since the hybrid time domain up to that time includes only one jump. 
The domain of $u_0$ is shown by the dash-dotted purple line,
which extends for four seconds of flow time since $N = 4$ and only two jumps occur up to that time.
At $(T_1,J_1)$, Problem~\ref{prob:optimum} is solved again to find the new optimal control input~$u_1$. The resulting optimal control input $u_1$ has a hybrid time domain that, after being shifted forward in hybrid time by~$(T_1,J_1)$, is indicated by solid blue. 
Since the hybrid time domain of $u_1$ exhibits two jumps without flow in between,
the next ``optimization event'' occurs at~$(T_2,J_2)=(2,3)$, which is two jumps after the prior optimization time~$(T_1,J_1)$. The resulting optimal control input~$u_2$ is then applied until~$(T_3,J_3)=(4,4)$, exactly two seconds after~$(T_2,J_2)$.
The domain of~$u_2$, shifted by~$(T_2,J_2)$, is depicted in red line in Figure~\ref{fig2}. Combining this sequence of optimal control inputs, the hybrid time domain of the resulting optimal state trajectory $(t,j) \mapsto x(t,j)$ is shown in Figure~\ref{fig3} from $(T_0,J_0) = (0,0)$ to~$(T_3,J_3)$.

\begin{figure}[!ht]
	\centering
	\includegraphics[scale=.8]{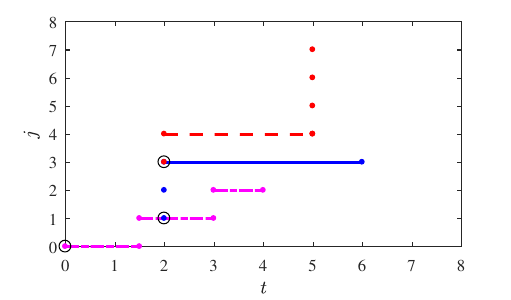}
\caption{Hybrid time domains associated with the hybrid MPC algorithm for the case of~$N=4$ and~$N_{c}=2$. 
The hybrid time instances $(T_1,J_1)$,~$(T_2,J_2)$ and~$(T_3,J_3)$ at which Problem~\ref{prob:optimum} is solved are circled.}
	\label{fig2}
\end{figure}

\begin{figure}[!ht]
	\centering
	\includegraphics[scale=.8]{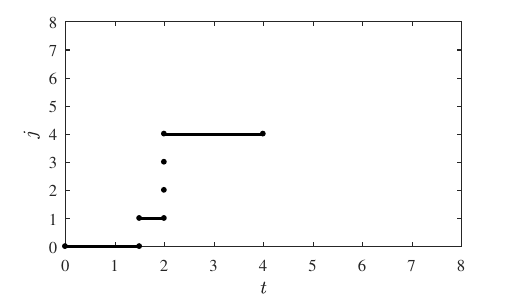}
	\caption{Hybrid time domain of the optimal trajectory associated to the hybrid time domains in Figure~\ref{fig2}.}
	\label{fig3}
\end{figure}

Under Definition~\ref{def:solgenMPC}, an appropriate notion of asymptotic stability is given in Section~\ref{sec:stability}. In general, asymptotic stability can be certified by using the value function~$\J^*$, without any assumptions on persistence of jumps or flows~\cite[Theorem~6.3]{193}, provided the cost functions defining~$\J$ have basic positive definiteness properties. With persistence of jumps or flows, these requirements can be relaxed, as detailed in Section~\ref{sec:stability}.
}

\begin{remark}[About Discretization]
\label{remark:AboutDiscretization}
Problem~\ref{prob:optimum} is formulated for the hybrid plant $\HS$ in \eqref{eqn:Hp}. Note that the flows of $\HS$ are not explicitly discretized in the implementation given in Algorithm~\ref{alg:HybridMPC}.
Such a discretization can be made explicit by choosing a discretization method for $\dot x = f(x,u)$ and defining solutions to the resulting system on discretized hybrid time domains---indeed, they are defined over hybrid time domains. Discretization for hybrid systems for the purposed of hybrid MPC are presented in \cite{205}, along with a method to solve  Problem~\ref{prob:optimum} with discretized flows; see also \cite{nodozi2025solving}. 
\end{remark}

\begin{remark}[Computing a Solution to Problem~\ref{prob:optimum}]
Computation of an optimal solution to Problem~\ref{prob:optimum} is a nontrivial task. In certain cases, Problem~\ref{prob:optimum} can be solved by converting it into a finite-dimensional nonlinear program. The development of general numerical methods to solve Problem~\ref{prob:optimum} is the current object of research, and can take the form of the maximum principle---see~\cite{pakniyat} and the references therein.
The approaches in \cite{205} and \cite{nodozi2025solving}
solve Problem~\ref{prob:optimum} with discretized flows. They algorithms therein use adaptations of
off-the-shelf optimization solvers for mixed-integer nonlinear programs, 
with potentially large computational times.
Note that \cite{286} provide conditions under
which the optimal cost can be continuously approximated numerically.
\end{remark}

\section{Main Assumptions}
\label{sec:assumptions}

We list the basic assumptions imposed on Problem~\ref{prob:optimum} to ensure feasibility and certify asymptotic stability of a given \textit{closed set}~$\A\subset X$ of interest. The required assumptions resemble those typically assumed in continuous-time and discrete-time MPC, when generalized to the set stabilization problem.\NotForJournal{\footnote{As mentioned in Section~\ref{sec:HybridMPCOverview}, hybrid systems typically have state variables that do not necessarily converge to equilibria (e.g., timers and logic variables). As such, stability theory for hybrid systems considers sets rather than singletons~\cite[Ch.~3]{65}.}} Combined with the hybrid prediction horizon defined in Definition~\ref{def:Tstructure}, the forthcoming assumptions guarantee most of the properties required to establish asymptotic stability, with the exception of positive definiteness of the value function, which is studied in depth in the forthcoming Sections~\ref{sec:posdef} and~\ref{sec:posdef-Sufficient}, as it can be achieved in multiple ways.

The following assumptions assert basic positive definiteness properties and bounds on the cost functions.
These properties are used in establishing positive definiteness of the 
value function and asymptotic stability of the closed-loop system resulting
from using hybrid MPC.

\begin{assumption}
\label{assmp:OCP}
Given a hybrid plant $\HS = (C,f,D,g)$, a closed set $\A \subset \reals^n$, flow cost $L_C$,
 jump cost $L_D$, terminal cost $V$, and terminal constraint set $X$,
	\begin{enumerate}[label={(W\arabic*)},leftmargin=*]
			\item \label{item:LC} There exists a class-$\mathcal{K}$
			function~${\alpha_C}$ such that $L_C(x,u)\geq {\alpha_C}(\Anorm{x})$ for each~$(x,u)\in C$.
		\item \label{item:LD} There exists a class-$\mathcal{K}$
		function~${\alpha_D}$ such that $L_D(x,u)\geq {\alpha_D}(\Anorm{x})$ for each~$(x,u)\in D$.
		\item \label{item:V} There exists a class-$\mathcal{K}$ function~${\alpha_V}$ such that $V(x)\geq {\alpha_V}(\Anorm{x})$ for each~$x \in X$.
	\end{enumerate}
\end{assumption}

\begin{assumption}
\label{assmp:OCP2}
Given a closet set $\A \subset \reals^n$, a terminal constraint set $X$, and a terminal cost function $V$, 
there exists~$\varepsilon>0$ and a class-$\mathcal{K}$ function~$\overline\alpha$ such that $V(x)\leq \overline\alpha(\Anorm{x})$ for each $x\in X{\cap(\A+\varepsilon\ball)}$.
\end{assumption}

The key stabilizing ingredient of our hybrid MPC algorithm is given by the following familiar CLF-like assumption
requiring the existence of a (static) state-feedback pair $\kappa = (\kappa_C,\kappa_D)$.
Below, forward pre-invariance of $X$ means that every maximal closed-loop solution with initial state in $X$ is complete and stays in~$X$ \cite[Definition 3.1]{185},
and forward completeness of $X$ means that for every point $X$ there exists a complete solution. 
A weaker notion of invariance is \textit{weak forward invariance}, which is used in the next section and simply requires that from each point in~$X$, there exists a maximal and complete solution that stays in~$X$.

\begin{assumption}
\label{assmp:CLF}
Given a hybrid plant $\HS = (C,f,D,g)$, a closet set $\A \subset \reals^n$, flow cost $L_C$,
jump cost $L_D$, terminal constraint set $X$, and terminal cost $V$,
there exists a  state-feedback law~$\kappa =(\kappa_C,\kappa_D)$ such that the terminal constraint set~$X$ is forward pre-invariant and forward complete for the hybrid system~$\HS_{\kappa}$ in~\eqref{eq:Hk}. Moreover, the terminal cost~$V$ is {differentiable} on an open set containing~${\closure(X\cap C_{\kappa})}$, and
		\begin{equation}
			\begin{aligned}
				\langle\nabla V(x),f_{\kappa}(x)\rangle	&\leq-L_C(x,\kappa_C(x)) &\forall x\in &X\cap C_{\kappa},\\
				V(g_{\kappa}(x))-V(x)										&\leq-L_D(x,\kappa_D(x)) &\forall x\in &X\cap D_{\kappa}
			\end{aligned}
		\label{eq:CLF}
		\end{equation}
\end{assumption}

The condition during flow  in \eqref{eq:CLF} resembles the one the continuous-time MPC literature. 
Therefore, it is possible to satisfy this condition, at least locally, using standard techniques relying on linearization in the special case when the flow cost and terminal cost are quadratic. The same is possible for the condition at jumps in \eqref{eq:CLF}. When~$L_C \circ \kappa_C$ (respectively,~$L_D\circ\kappa_D)$) is allowed to be zero, the terminal cost~$V$ can be viewed as a Lyapunov function with the feedback~$\kappa$, provided the trajectories of the hybrid closed-loop system~$\HS_{\kappa}$ have persistent jumps (respectively, flows). 

\begin{remark}[About Existence of Solutions]
\label{remark:Existence}
\NotForJournal{Any MPC-type algorithm would only be able to issue a control input when the optimization problem at the current state has a solution.}  Existence of a solution to Problem~\ref{prob:optimum} from a feasible initial condition $\xinitial$ 
follows directly from \cite{286} under standard regularity conditions,\footnote{Though the results in this paper assume that optimal (hence, feasible) solutions exist, for the integral 
in the hybrid cost functional~$\J$ in \eqref{eq:cost} to exist, the composition $L_C \circ (x,u)$ needs to be (Lebesgue) measurable.} or from \cite{goebel2019existence}.
In particular, the approach in \cite{286}  rewrites the problem into the Mayer form
	\begin{equation}
	\label{eq:optimumMayer}
		\begin{aligned}
			& \underset{z\in\widehat{\sol}_{\HS_M}}{\text{minimize}}	& & \J(z)							\\
			& \text{subject to}																				& &\!\!\!\! (z(0,0),(T,J),z(T,J))\!\in\! \{(\xinitial,0)\} \times \T \times (X \times \reals)
		\end{aligned}
	\end{equation}
where $z = (x,\ell)$ with $\ell$ representing the running cost and $\HS_M$ is an autonomous system
arising from the conversion to the Mayer form,
and studying its existence of solutions---see \cite[Section 6]{286} for examples constructing the Mayer form system $\HS_M$.
More precisely, \cite[Theorem 8]{286} ensures existence of solutions to Problem~\ref{prob:optimum} from the feasible set $\X$ when $\cal J$ is lower semicontinuous and  $X$ is closed.
\end{remark}

\section{Properties of the Optimal Control Problem}
\label{sec:ocp}

This section presents the properties pertinent to Problem~\ref{prob:optimum} that are used to show asymptotic stability of~$\A$. 

\subsection{Forward Invariance of the Feasible Set}
\label{sec:FIofX}

The initial results are in the same spirit as those in~\cite{193}, stated in more generality.

\begin{proposition}
\label{prop:aroundA}
If there exists a state-feedback law~$\kappa$ such that the terminal constraint set~$X$ is weakly forward invariant for the hybrid system~$\HS_{\kappa}$ in~\eqref{eq:Hk}, then~$X\subset \X$.
\end{proposition}
\begin{IEEEproof}
Since for each point $\xinitial$ in $X$ there exists a maximal and complete solution,
each trajectory to~$\HS_{\kappa}$ starting from~$X$ stays in~$X$.
The domain of each such trajectory intersects with the prediction horizon~$\T$, due to the structure of the hybrid prediction horizon in Definition~\ref{def:Tstructure}.
Hence, $\xinitial$ belongs to $\X$.
\end{IEEEproof}

\begin{proposition}
\label{prop:T1T2}
{Suppose that there exists a state-feedback law~$\kappa$ such that the terminal constraint set~$X$ is forward pre-invariant and forward complete for the hybrid system~$\HS_{\kappa}$ in~\eqref{eq:Hk}.} 
Then, each feasible solution pair $(x,u)$ to $\HS$ satisfies $x(t,j) \in \X$ for each~$(t,j)\in\dom(x,u)$.
\end{proposition}
\begin{IEEEproof}
The proof relies on the extension of the solution pair~$(x,u)$ from the terminal point by aid of the feedback~$\kappa$. That is, we take a solution pair~$(x',u')$ generated by the feedback~$\kappa$ starting from~$x(T,J)\in X$, where~$(T,J)$ is the terminal time of~$(x,u)$. That is, we let~$(z(s,i),v(s,i)):=(x(s,i),u(s,i))$ for all~$(s,i)$ with~$s+i<T+J$, and~$(z(s+T,i+J),v(s+T,i+J)):=(x'(s,i),u'(s,i))$, and observe that~$(z,v)$ is a solution pair to~$\HS$. Then, we consider the solution pair~$(z,v)$ from~$(t,j)\in\dom(z,v)$ onwards, which is itself a solution pair, say~$(z',v')$. We notice that from~$(T-t,J-j)\in\dom(z',v')$ onwards,~$z$ stays in the terminal constraint set~$X$. Finally, we note that there must exist~$(s,i)\in\dom(z',v')$ after time~$(T-t,J-j)\in\dom(z',v')$ that belongs to the prediction horizon~$\T$. Hence, the truncation of~$(z',v')$ at time~$(s,i)$ is a feasible solution pair to $\HS$, since~$(s,i)\in\T$ and~$z(s,i)\in X$.
\end{IEEEproof}

Proposition~\ref{prop:aroundA} and \ref{prop:T1T2} show that the terminal constraint set is contained in the feasible set, and recursive feasibility is maintained with the moving horizon implementation, respectively. The conditions in Proposition~\ref{prop:aroundA} and the inequalities in Assumption~\ref{assmp:CLF} ensure an upper bound on the value function~$\J^*$ over the terminal constraint set~$X$, in terms of the terminal cost~$V$. 

\subsection{Continuity of the Value Function}

Combining the property of $X$ 
established in Proposition~\ref{prop:aroundA}
with the bounds obtained in Assumption~\ref{assmp:CLF}
is the first step towards establishing~$\J^*$ as a Lyapunov function for the hybrid MPC algorithm.
 The terminal cost upper bounds not only the value function on the terminal constraint set, but also the cost of feasible solutions generated by $\kappa$.
To establish this result, we exploit similar techniques to those in the standard MPC literature, see, e.g.,
\cite{Mayne.ea.00.Automatica} and
\cite[Section 5.3]{GrunePannek2017}.

\begin{lemma}
\label{lem:valuecont}
Suppose that Assumption~\ref{assmp:CLF} holds.  Then,
\begin{enumerate}
\item \label{item:VboundOnValueFunction}
$\J^*$ and $V$ satisfy
\begin{equation}\label{eqn:VboundOnValueFunction}
\J^*(\xinitial)\leq V(\xinitial)\qquad 
\forall \xinitial\in X
\end{equation}
\item \label{item:JboundOnValueFunction}
For each $\xinitial\in X$ and each feasible solution pair $(x,u)$ starting from~$\xinitial$ generated by the state-feedback law~$\kappa$, $\J^*$ and $\J$ satisfy
\begin{equation}\label{eqn:VboundOnValueFunction2}
\J^*(\xinitial)\leq \J(x,u) \leq V(\xinitial)
\end{equation}
\end{enumerate}
\end{lemma}
\begin{IEEEproof}
Since Assumption~\ref{assmp:CLF} holds, Proposition~\ref{prop:aroundA} implies that $X \subset \X$.
Let $(x,u)$ be a feasible solution to Problem~\ref{prob:optimum} with $\xinitial \in X$
and, without loss of generality, let~$(T,J)$ be the terminal time of~$(x,u)$ and $\{s_j\}_{j=0}^{J}$ be the associated sequence of generalized jump times. Then, using \eqref{eq:cost},
\IfJournal{$\J(x,u)=\left(\sum_{j=0}^{J}\int_{s_j}^{s_{j+1}}L_{C}(x(t,j),\kappa_C(x(t,j)))\,dt\right) 
	+\left(\sum_{j=0}^{J-1}L_{D}(x(s_{j+1},j), \kappa_D(x(s_{j+1},j)))\right)+V(x(T,J))$.
Since~$X$ is forward pre-invariant for~$\HS_{\kappa}$, 
$x$ remains in $X$ and, by~\eqref{eq:CLF}, $\J(x,u)\leq-\left(\sum_{j=0}^{J}\int_{s_j}^{s_{j+1}}\frac{dV}{dt}(x(t,j))\,dt\right. 
	+\left.\sum_{j=0}^{J-1} \left(V(x(s_{j+1},j+1))-V(x(s_{j+1},j))\right)\right)+V(x(T,J))
	= -(V(x(T,J))-V(x(0,0)))+V(x(T,J)) 
	=V(x(0,0)) = V(\xinitial)$.}{
\begin{multline*}
	\J(x,u)=\left(\sum_{j=0}^{J}\int_{s_j}^{s_{j+1}}L_{C}(x(t,j),\kappa_C(x(t,j)))\,dt\right) \\
	+\left(\sum_{j=0}^{J-1}L_{D}(x(s_{j+1},j), \kappa_D(x(s_{j+1},j)))\right)+V(x(T,J))
\end{multline*}
Since~$X$ is forward pre-invariant for~$\HS_{\kappa}$, $x$ remains in $X$ and, by~\eqref{eq:CLF},
\begin{multline*}
	\J(x,u)\leq-\left(\sum_{j=0}^{J}\int_{s_j}^{s_{j+1}}\frac{dV}{dt}(x(t,j))\,dt\right. \\
	+\left.\sum_{j=0}^{J-1} \left(V(x(s_{j+1},j+1))-V(x(s_{j+1},j))\right)\right)+V(x(T,J))\\
	= -(V(x(T,J))-V(x(0,0)))+V(x(T,J)) \\
	=V(x(0,0)) = V(\xinitial)
\end{multline*}}
Item~2 follows from the definition of $\J^*$ in \eqref{eq:value} since the value of $\J^*$ 
is the result of minimizing the cost functional over all such feasible solution pairs.
Item~1 follows from the fact that the set $X$ is 
forward complete
 and that the cost
obtained by any optimal solution cannot be larger than the cost for the solution generated 
using $\kappa$. 
\end{IEEEproof}

The following continuity property of $\J^*$ is a direct consequence of item~\ref{item:VboundOnValueFunction} in Lemma~\ref{lem:valuecont}.
\begin{proposition}
\label{proposition:continuityValueFunction}
Under Assumption~\ref{assmp:CLF},
for each~$\epsilon>0$, there exists~$\delta>0$ such that~${\J^*(x)\leq\epsilon}$ for each~${x\in\X}$ satisfying~$\Anorm{x}\leq\delta$.
\end{proposition}

\subsection{Decreasing Properties of the Value Function}

Next, we show that~$\J^*$ is upper bounded by a nonincreasing function along optimal trajectories, which decreases during flows (respectively, at jumps) if $L_C$ (respectively, $L_D$) satisfies the lower bound in~\ref{item:LC} (respectively, in~\ref{item:LD}).  
The proof follows the steps used in the standard MPC literature,
see, e.g., 
\cite[Section 3.4]{Mayne.ea.00.Automatica} and
\cite[Lemma 5.4]{GrunePannek2017}.

\begin{lemma}
\label{lem:descent}
Suppose 
Assumption~\ref{assmp:CLF} holds.
Then, for each optimal solution~$(x,u)$ and each~$(t,j)\in\dom(x,u)$,
\begin{multline*}
\hspace{-0.15in}	\J^*(x(t,j))\leq\J^*(x(0,0)){-\sum_{i=0}^{j}\int_{s_i}^{s_{i+1}}L_C(x(s,i),u(s,i))\,ds}\\
									{-\sum_{i=0}^{j-1}L_D(x(s_{i+1},i),u(s_{i+1},i))}
\end{multline*}
where~$\{s_i\}_{i=0}^{j+1}$ is the sequence of generalized jump times of the truncation of~$(x,u)$ up to time~$(t,j)$.
\end{lemma}
\begin{IEEEproof}
Given the optimal solution $(x,u)$, we recall the feasible solution pair derived in the proof of Proposition~\ref{prop:T1T2} that starts from~$x(t,j)$. Call this solution pair~$(z,v)$, and note that after time~$(T-t,J-j)\in\dom(z,v)$,~$(z,v)$ corresponds to a solution generated by the feedback~$\kappa$, say~$(x',u')$, where~$(T,J)$ is the terminal time of~$(x,u)$. Then, observe that~$\J^*(x(t,j))\leq \J(z,v)$ as~$(z,v)$ is not optimal. 
The following holds:
\IfJournal{$\J^*(x(t,j))\leq\J(z,v)=\J(x,u)
-\sum_{i=0}^{j}\int_{s_i}^{s_{i+1}}L_C(x(s,i),u(s,i))\,ds
-\sum_{i=0}^{j-1}L_D(x(s_{i+1},i),u(s_{i+1},i))
-V(x(t,j))+\J(x',u')$.}{
\begin{multline*}
\J^*(x(t,j))\leq\J(z,v)=\J(x,u)\\
-\sum_{i=0}^{j}\int_{s_i}^{s_{i+1}}L_C(x(s,i),u(s,i))\,ds\\
-\sum_{i=0}^{j-1}L_D(x(s_{i+1},i),u(s_{i+1},i))\\
-V(x(t,j))+\J(x',u')
\end{multline*}}
The bound $\J(z,v)$ follows from the fact that $(z,v)$ starts from $x(t,j)$ and that $(z,v)$ is not optimal.
The equality captures the following fact: 
the cost of $(x,u)$, which is up to $(T,J)$, plus the cost of $(x',u')$ is equal to the cost of $(x,u)$ up to $(t,j)$ plus the cost of $(z,v)$.  By solving for the cost of $(z,v)$ from this relationship we obtain the equality above.
Finally, since~$x(t,j)=x'(0,0)$, the term~$-V(x(t,j))+\J(x',u')$ is nonpositive by Lemma~\ref{lem:valuecont}, leading to the desired result as~$\J(x,u)=\J^*(x(0,0))$ since $(x,u)$ is optimal.
\end{IEEEproof}

\subsection{Solution-dependent Conditions for Positive Definiteness of the Value Function}
\label{sec:posdef}

To ensure that the value function is positive definite and bounded from below by a positive definite function,
we start by assuming the existence 
of a positive definite function~$\alpha$ such that an appropriate combination of the following conditions holds for every optimal solution pair~$(x,u)$:
\begin{enumerate}[label={P\arabic*)},leftmargin=*]
	\item  \label{item:bestbound} 
there exist $(t,j)$, $(t',j')\in\dom(x,u)$ with~$t+j\leq t'+j'$ such that $t'-t+j'-j\geq \alpha(\Anorm{x(0,0)})$ and
\[
	\Anorm{x(s,i)}\geq\alpha(\Anorm{x(0,0)}) 
\]
for all $(s,i)\in\dom (x,u)$
satisfying $t+j\leq s+i\leq t'+j'$;
	\item \label{item:flobound}
there exist $(t,j), (t',j')\in\dom(x,u)$ with~$t+j\leq t'+j'$ such that $t'-t\geq \alpha(\Anorm{x(0,0)})$ and
\[
	\Anorm{x(s,i)}\geq\alpha(\Anorm{x(0,0)}) 
\]
for all $(s,i)\in\dom (x,u)$
satisfying $t+j\leq s+i\leq t'+j'$;
\item \label{item:jumbound} 
$\Anorm{x(t,j)}\geq\alpha(\Anorm{x(0,0)})$ for some $(t,j)\in\dom(x,u)$ such that~$(t,j+1)\in\dom(x,u)$;
\item \label{item:termbound} 
$\Anorm{x(T,J)}\geq {\alpha}(\Anorm{x(0,0)})$, where~$(T,J)$ is the terminal time of~$(x,u)$;
\item \label{item:flowjumpbound}
There exist an open neighborhood~$\U$ of $\A$ and 
a continuous function~$\sigma:\realsgeq\to\realsgeq$ such that
$|f(x,u)|\leq \sigma\left(\Anorm{x}\right)$ for all $(x,u)\in C$ such that $x\in \U$.
\end{enumerate}
Section~\ref{sec:posdef-Sufficient} provides sufficient conditions for these properties.
Note that \ref{item:flobound} implies \ref{item:bestbound}, and that \ref{item:jumbound} implies \ref{item:bestbound}.  

\begin{theorem}(Positive Definiteness of the Value Function)
\label{thm:PDofJ}
Given a hybrid plant $\HS = (C,f,D,g)$ as in \eqref{eqn:Hp},
the data $(\T,L_{C},L_{D},V,X)$ defining Problem~\ref{prob:optimum},
and a closed set $\A \subset \reals^n$,
if there exists a 
positive definite function $\alpha:\realsgeq\to\realsgeq$ such that 
one of the following conditions holds:\footnote{These conditions allow for zero stage cost $L_C$ or $L_D$, in particular, when the flows or jumps are persistent. They also prevent fast convergence to $\mathcal{A}$ during flow, which may prevent guaranteeing positive definiteness of the value function. In fact, while, in discrete time, positive definiteness of the value function follows from positive definiteness of the stage cost, when flows are not discretized like in our setting, optimal solutions might reach the target set $\mathcal{A}$ in finite time, possibly making the integral cost zero almost everywhere. For example, for the scalar continuous-time system $\dot x = u$ with $L_C(x,u) = |x|$ and $\mathcal{A} = \{0\}$, the control $u(t) = -\frac{x_{\circ}}{r}$ for each $t \in [0,r)$ and zero otherwise, where $x_{\circ} \not \in \mathcal{A}$ is the initial condition and $r >0$ is arbitrary, leads to a cost equal to $\frac{x_{\circ}r}{2}$, which vanishes to zero as $r$ approaches zero, making the value function to be zero at $x_{\circ}$. }
\begin{enumerate}
    \item\label{item1-pd} 
    Every optimal solution pair~$(x,u)$ satisfies \ref{item:bestbound}, and 
    $L_C:C\to\realsgeq$ and $L_D:D\to\realsgeq$ 
     satisfy Assumption~\ref{assmp:OCP};
    \item\label{item2-pd}  
    Every optimal solution pair~$(x,u)$ satisfies \ref{item:flobound} and 
    $L_C:C\to\realsgeq$  satisfies  Assumption~\ref{assmp:OCP};
    \item\label{item3-pd}  
    Every optimal solution pair~$(x,u)$ satisfies \ref{item:jumbound} and $L_D:D\to\realsgeq$ satisfies Assumption~\ref{assmp:OCP};
    \item\label{item4-pd}  
    Every optimal solution pair~$(x,u)$ satisfies \ref{item:termbound} and 
    $V:X\to\realsgeq$ satisfies  Assumption~\ref{assmp:OCP};
    \item\label{item5-pd}  
    Every optimal solution pair~$(x,u)$ satisfies \ref{item:flobound}
    or \ref{item:termbound}, and 
    $L_C:C\to\realsgeq$ and $V:X\to\realsgeq$ satisfy  Assumption~\ref{assmp:OCP};
    \item\label{item6-pd}  
    Every optimal solution pair~$(x,u)$ satisfies \ref{item:jumbound} or \ref{item:termbound}, and 
    $L_D:D\to\realsgeq$ and $V:X\to\realsgeq$ satisfy Assumption~\ref{assmp:OCP};
    \item\label{item7-pd} 
        Every optimal solution pair~$(x,u)$ satisfies one among \ref{item:bestbound}-\ref{item:flowjumpbound},
and $L_C:C\to\realsgeq$, $L_D:D\to\realsgeq$, and $V:X\to\realsgeq$ satisfy Assumption~\ref{assmp:OCP};
\end{enumerate}
then there exists a positive definite function 
$\alpha^*:\realsgeq\to\realsgeq$ such that 
the value function satisfies
\begin{equation}\label{eqn:ValueFunctionLowerBound}
\J^*(\xinitial)\geq \alpha^*(\Anorm{\xinitial}) 
\end{equation}
for each $\xinitial\in \X$ such that 
an optimal solution pair~$(x,u)\in\widehat{\sol}_{\HS}(\xinitial)$ exists.
Furthermore, the function
 $\alpha^*$ is class-$\mathcal{K}_{\infty}$
in items~\ref{item1-pd}-\ref{item6-pd} if the bounds in Assumption~\ref{assmp:OCP} hold with class-$\mathcal{K}_\infty$ functions
and $\alpha$ is class-$\mathcal{K}_\infty$.
\end{theorem}
\begin{IEEEproof}
Pick $\xinitial \in \X$ and suppose that there exists an optimal solution pair $(x,u)$ starting from $\xinitial$.
Let $\{s_j\}_{j=0}^{J+1}$ be the sequence of generalized jump times of $(x,u)$, namely, it satisfies $\dom (x,u) = \cup_{j=0}^J ([s_j,s_{j+1}]\times\{j\})$.
Therefore,  we have
\IfJournal{$
\J^*(\xinitial) = \J(x,u) = \left(\sum_{j=0}^{J}\int_{s_j}^{s_{j+1}}L_{C}(x(t,j),u(t,j))\,dt\right) 
	+\left(\sum_{j=0}^{J-1}L_{D}(x(s_{j+1},j), u(s_{j+1},j))\right)+V(x(T,J))$.}{
\begin{multline*}
\J^*(\xinitial) = \J(x,u) = \left(\sum_{j=0}^{J}\int_{s_j}^{s_{j+1}}L_{C}(x(t,j),u(t,j))\,dt\right) \\
	+\left(\sum_{j=0}^{J-1}L_{D}(x(s_{j+1},j), u(s_{j+1},j))\right)+V(x(T,J))
\end{multline*}}
Next, we consider the cases in items~\ref{item1-pd}-\ref{item7-pd}.

Now, suppose that item~\ref{item1-pd} holds.
Using the fact that $L_C$ and $L_D$ are positive definite with respect to $\A$, 
$V$ takes on nonnegative values,
the properties of the solution $(x,u)$ for hybrid times between
$(t,j)$ and  $(t',j')$ coming from item~\ref{item1-pd} with positive definite function $\alpha$ obtained from \ref{item:bestbound}, 
there exists a positive definite function
$\widetilde{\alpha}:\realsgeq \to \realsgeq$ such that\footnote{We use $s$ instead of $t$ as integration variable to simplify notation.}
\begin{align*}
\J^*(\xinitial) & \geq \int_{t}^{t'}{\widetilde{\alpha}}(\Anorm{x(0,0)})\,d s  +\sum_{j}^{j'}{\widetilde{\alpha}}(\Anorm{x(0,0)}) \\
&\geq  {\alpha}(\Anorm{\xinitial})  \widetilde{\alpha}(\Anorm{\xinitial}) =: \alpha^*(\Anorm{\xinitial}) 
\end{align*}
When item~\ref{item2-pd} holds,
similarly,
using the fact that $L_C$ is positive definite with respect to $\A$,
the properties in item~\ref{item2-pd},
and the fact that $V$ is nonnegative,
it follows that there exists a positive definite function
$\widetilde{\alpha}:\realsgeq \to \realsgeq$ such that
\IfJournal{$\J^*(\xinitial)  \geq \int_{t}^{t'}{\widetilde{\alpha}}(\Anorm{x(0,0)})\,ds 
\geq {\alpha}(\Anorm{\xinitial}) \widetilde{\alpha}(\Anorm{\xinitial}) =: \alpha^*(\Anorm{\xinitial})$,}{
\begin{align*}
\J^*(\xinitial) & \geq\int_{t}^{t'}{\widetilde{\alpha}}(\Anorm{x(0,0)})\,d s \\
&\geq (t'-t) \underline{\alpha}(\Anorm{\xinitial})^{\frac{1}{2}} \\
&\geq {\alpha}(\Anorm{\xinitial}) \widetilde{\alpha}(\Anorm{\xinitial}) =: \alpha^*(\Anorm{\xinitial})
\end{align*}}
The case when item~\ref{item3-pd} holds follows similarly.
The case when item~\ref{item4-pd} holds is immediate due to $V$ being positive definite, from where there exists 
a positive definite function
${\alpha}^*:\realsgeq \to \realsgeq$
satisfying $V(x(T,J)) \geq {\alpha}^*(\Anorm{x(T,J)}) \geq {\alpha}^*(\Anorm{\xinitial})$, which is a lower bound for $\J^*(\xinitial)$.
The cases in items~\ref{item5-pd}, \ref{item6-pd}, and \ref{item7-pd} under conditions  
\ref{item:bestbound}-\ref{item:termbound}
follow directly from combinations of the steps above.
The case in item~\ref{item7-pd} under condition \ref{item:flowjumpbound} is shown next.

Let~$T:=\min\{t\in\realsgeq: (t,0)\in \T\}$, which is positive from the construction in Definition~\ref{def:Tstructure}. With $\U$ and $\sigma$~(assumed to be p.d. without loss of generality) 
coming from \ref{item:flowjumpbound},
fix~$\epsilon>0$ such that~$\Anorm{x}\leq\epsilon$ implies~$x\in \U$. Define the function~$\varphi:[0,\epsilon]\to\realsgeq$ as $\varphi(r):=\max_{r'\in[r/2,r]}\sigma(r')$ for all  $r \in[0,\epsilon]$.
Note that~$\varphi$ is continuous. \NotForJournal{This property can be shown via Berge's maximum theorem~\cite[Theorem~1.4.16]{aubin}, similar to the proof of Lemma~\ref{lem:cont}, by showing that the set-valued mapping associating the interval~$[r/2,r]$ with each~$r \in[0,\epsilon]$ is upper and lower semicontinuous.\footnote{A set-valued mapping~$F:\reals^n\rightrightarrows\reals^m$ is upper semicontinuous at a point~$x\in\dom F$ with~$F(x)$ compact if for any~$\epsilon>0$, there exists~$\delta>0$ such that~$F(x+\delta\ball)\subset F(x)+\epsilon\ball$~\cite[Definition~1.4.1]{aubin}. It is said to be lower semicontinuous at a point~$x\in\dom F$ if for any~$y\in F(x)$ and for any sequence~$\{x_i\}_{i\in\nats}\in\dom F$ converging to~$x$, there exists a sequence $\{y_i\}_{i\in\nats}$ with~$y_i\in F(x_i)$ for all~$i\in\nats$, which converges to~$y$~\cite[Definition~1.4.1]{aubin}. These notions should not be confused with upper and lower semicontinuity of (extended) real-valued functions.} 
}
Take any~$r\in(0,\epsilon]$. Let~$(x,u)\in\sol_{\HS}$ be a feasible solution pair from $\xinitial$ such that~$\Anorm{\xinitial}=r$. Define the scalar~$t'(r):= \min\{r/(2\varphi(r)),T\}>0$. Observe that for all~$t< t'(r)$ satisfying~$(t,0)\in\dom(x,u)$, we have~$\Anorm{x(t,0)}>r/2$; the converse would imply~$t'(r)\varphi(r)>r/2$. Therefore, 
from the p.d. property of $L_C$, there exists a positive definite function $\widetilde{\alpha}_C$ such that
\begin{multline}
	\int_{0}^{t}L_C(x(s,0),u(s,0))\,ds \NotForJournal{\geq\int_{0}^{t}\widetilde{\alpha}_C(\Anorm{x(s,0)})\,ds} \\
	\qquad \qquad \geq t\widetilde{\alpha}_C(r/2)\\ \quad \forall (t,0)\in\dom(x,u)\cap[0,t'(r)]\times\{0\},
\label{eq:costf}
\end{multline}
and, if there exists~$t\leq t'(r)$ so that~$(t,1)\in\dom(x,u)$, then
\begin{equation}
	L_D(x(t,0),u(t,0))\geq\widetilde{\alpha}_D(\Anorm{x(t,0)})\geq\widetilde{\alpha}_D(r/2)
\label{eq:costj}
\end{equation}
for some p.d. function $\widetilde{\alpha}_D$,
from the p.d. property of $L_D$.

Suppose there exists no~$t\leq t'(r)$ that satisfies~$(t,1)\in\dom(x,u)$. Then, because~$t'(r)\leq T$, it must be that~$(t'(r),0)\in\dom(x,u)$, and hence by~\eqref{eq:costf}, we have
from \eqref{eq:cost}
 that~$\J(x,u)\geq t'(r)\widetilde{\alpha}_C(r/2)$. On the other hand, if~$t\leq t'(r)$ exists so~$(t,1)\in\dom(x,u)$, then~\eqref{eq:costj} implies~$\J(x,u)\geq \widetilde{\alpha}_D(r/2)$. Thus, it follows that
\begin{equation}
	\J(x,u)\geq 
	\min\{t'(r)\widetilde{\alpha}_C(r/2),\widetilde{\alpha}_D(r/2)\} =:\widetilde{\alpha}(r) >0
\label{eq:lobound}
\end{equation}
for all feasible~$(x,u)\in\sol_{\HS}$ such that~$\Anorm{\xinitial}=r$. Note that $\widetilde{\alpha}(r)$ depends continuously on~$r$ on~$(0,\delta]$ since~$t'(r)$ depends continuously on~$r$ on~$(0,\delta]$. Moreover, $\lim_{r\to 0}\widetilde{\alpha}(r)=0$.
Now, let~$(x,u)\in\sol_{\HS}$ be a feasible solution pair from $\xinitial$ such that~$\Anorm{\xinitial}\geq\epsilon$. Since~$x$ can reach~$\A+(\epsilon/2)\ball$ no sooner than~$(t'(\epsilon),0)$ by flowing, a similar analysis shows that \eqref{eq:lobound} holds with $\epsilon$ instead of $r$.
Combining these bounds
with the fact that~$\widetilde{\alpha}(r)$ depends continuously on~$r$ on~$[0,\epsilon]$ and approaches zero as~$r$ tends to zero, it follows that~$\J(x,u)\geq \alpha^*(\Anorm{x(0,0)})$ for every feasible solution pair~$(x,u)\in\sol_{\HS}$, where the function~$\alpha^*:\realsgeq\to\realsgeq$ is defined as $\alpha^*(\epsilon) := 0$ if $\epsilon =0$ and $\alpha^*(\epsilon) := \widetilde{\alpha}(\epsilon)	$ if $\epsilon > 0$.
Consequently,~$\J^*(\xinitial)\geq \alpha^*(\Anorm{\xinitial})$ for all~$\xinitial\in \X$.
\end{IEEEproof}

\begin{remark}
The solution-dependent conditions in Theorem~\ref{thm:PDofJ} guarantees required properties to
guarantee positive definiteness of the value function.
Conditions~\ref{item:bestbound}-\ref{item:termbound} therein relax positive definiteness properties of the stage costs and terminal cost, in particular, to allow for the stage cost for flows or jumps to be zero, which permits the user to not associate a cost to flows or jumps, respectively.  
For instance, for solution pairs with persistent jumps, item~\ref{item3-pd} relaxes positive definiteness of the flow cost $L_C$ and terminal cost $V$ by assuming a property on optimal solution pairs $(x, u)$ originating away from $\A$ -- see condition \ref{item:jumbound}. This property guarantees that $x$ is away from $\A$ at a jump time, which ensures that the cost of $(x, u)$ is nonzero.
A similar relaxation is possible for the jump cost $L_D$, via item~\ref{item2-pd}, through the use of \ref{item:flobound}. 
Note that the conditions in item~\ref{item2-pd} relax those in item~\ref{item1-pd}
by requiring the flow cost to be positive definite when flows are persistent, as captured by \ref{item:flobound}.
A similar comment applies to item~\ref{item3-pd}. Item~\ref{item7-pd} exploits condition \ref{item:flowjumpbound} to establish the lower bound, similar to results in the standard MPC literature, at the price of requiring positive definiteness of $L_C$, $L_D$, and $V$.
Note that \ref{item:flowjumpbound} implies the existence of a uniform upper bound on the magnitude of the velocity of the state trajectories from nearby $\A$, and that is satisfied when $f$ is continuous, $\A$ is compact, and $C = C' \times U$ for some closed set $C' \subset \reals^n$ and some compact set $U \subset \reals^m$,  which corresponds to the typical assumptions on state and input constraints in the MPC literature.
\end{remark}

Before stating solution-independent sufficient conditions for \ref{item:bestbound}-\ref{item:termbound}, we revisit the sample-and-hold control example.

\begin{example}[Sample-and-Hold Control (revisited)]
\label{ex:samplejump} 
Consider the digital control system in Example~\ref{ex:sample} and suppose that the function~$\widetilde{f}$ is linear, i.e.,~$\widetilde{f}(z,\eta)=Az+B\eta$ for some matrices~$A$ and~$B$. Let~$\A=\{0\}\times\{0\}\times[0,T_s]$. Along the first period of flow, a solution pair~$(x,u)$ is given by~$\tau_s(t,0)=t+\tau_s(0,0)$ and
\[
	\begin{bmatrix}
	z(t,0)\\ \eta(t,0)
	\end{bmatrix}
=\exp(\widetilde{A}t)
	\begin{bmatrix}
	z(0,0)\\ \eta(0,0)
	\end{bmatrix}, \quad \widetilde{A}:=	\begin{bmatrix}
	A& B\\ 0 & 0
	\end{bmatrix}
\]
Note that the function~$t\mapsto\exp(\widetilde{A}t)$ is continuous and $\exp(\widetilde{A}t)$ is nonsingular
for each~$t\geq 0$. Moreover, regardless of the choice of the horizon~$\T$, due to the timer dynamics, 
the sequence~$\{t_j\}_{j=0}^{J+1}$ defining~$\dom(x,u)$ is such that~$t_1\leq T_s-\tau_s(0,0)\leq T_s$. Since singular values vary continuously in the matrix entries, there exists~$c>0$ such that
\begin{multline*}
	\Anorm{x(t,0)}=|(z(t,0),\eta(t,0))|\\
\qquad	\geq c|(z(0,0),\eta(0,0))|=\Anorm{x(0,0)}
\end{multline*}
for all $(t,0)\in\dom(z,u)$.
Hence,~$\Anorm{x(t_1,0)}\geq c\Anorm{x(0,0)}$, and \ref{item:jumbound} holds with $\alpha(s) := c s$ for each $s \geq 0$.
\end{example}

\subsection{Sufficient Conditions for P1-P4}
\label{sec:posdef-Sufficient}

P1-P4 can be verified via infinitesimal conditions if the flows of $\HS$ satisfy a Lyapunov-like inequality that limit the rate of convergence to~$\A$ and prohibit finite-time convergence, or if appropriate local boundedness conditions hold.

\begin{proposition}
\label{prop:mess}
Given a hybrid plant $\HS = (C,f,D,g)$ as in \eqref{eqn:Hp},
the data $(\T,L_{C},L_{D},V,X)$ defining Problem~\ref{prob:optimum},
and a closed set $\A \subset \reals^n$, 
suppose that one of the following conditions holds:
\begin{enumerate}
\item\label{item:mess1} Given a function $\widetilde{V}:\reals^n\to\realsgeq$ that is continuously differentiable\footnote{This condition could as well be formulated for a locally Lipschitz function $\widetilde{V}$ through the use of the generalized Clarke derivative.} on an open set containing~$\closure(\Pi(C))$, suppose that there exist 
class-${\cal K}$
 functions~$\widetilde{\alpha}_1$ and $\widetilde{\alpha}_2$, and constants $\lambda\in\reals$ and~$\varepsilon>0$ such that
\begin{equation}
	{\widetilde{\alpha}_1}(\Anorm{x})\leq \widetilde{V}(x)\leq {\widetilde{\alpha}_2}(\Anorm{x}) \quad \forall x\in \Pi(C): \Anorm{x}\leq\varepsilon,
\label{eq:candidate}
\end{equation}
\begin{equation}
	\langle\nabla \widetilde{V}(x),f(x,u)\rangle	\geq \lambda \widetilde{V}(x) \quad\forall (x,u)\in C: \Anorm{x}\leq\varepsilon
\label{eq:lodescent}
\end{equation}
\item\label{item:mess2} Suppose that there exist a continuous function~$\sigma:(0,\infty)\to[0,\infty)$ and~$\varepsilon>0$ satisfying
\begin{equation}
		|f(x,u)|\leq \sigma\left(\Anorm{x}\right) \quad \forall (x,u)\in C: 0<\Anorm{x}\leq\varepsilon
\label{eq:finiteact}
\end{equation}
\end{enumerate}
Then, there exists a class-$\classK$ function~$\alpha$ such that every feasible solution pair~$(x,u)$ to $\HS$ satisfies \ref{item:flobound}, \ref{item:jumbound}, or \ref{item:termbound} 
with $J=0$. 
Furthermore, 
if, in addition, 
\begin{enumerate}[label={\alph*)},leftmargin=*]
\item 
\label{item:P1suff} $\T$ is such that
\begin{enumerate}[label={\roman*)},leftmargin=*]
\item
\label{item:P1suff-1}
 there exists $c>0$ s.t. $T+J\geq c$ for all $(T,J)\in{\cal T}$,
\end{enumerate}
then there exists a class-$\classK$ function~$\alpha$ such that every feasible solution pair to $\HS$ satisfies~\ref{item:bestbound};
\item 
\label{item:P3suff}
$\T$ is such that item~\ref{item:P1suff}.\ref{item:P3suff-1} holds,
\begin{enumerate}[label={\roman*)},leftmargin=*]
\item
\label{item:P3suff-1}
there exists a positive definite function $\widetilde{\alpha}_D$ such that
\begin{multline}\label{eqn:P3lowerBound}
	\Anorm{g(x,u)}\geq \widetilde{\alpha}_D(\Anorm{x}) \quad \forall (x,u)\in D
\end{multline}
and
\item
\label{item:P3suff-2}
there exists $\epsilon > 0$ such that each feasible pair flows for at least $\epsilon$ seconds,
\end{enumerate}
then \ref{item:flobound} or \ref{item:termbound} holds for each feasible 
solution pair to $\HS$ if 
$\widetilde{\alpha}_D(r) \geq r$ for each $r \geq 0$,
or there exists $\overline{J} \in \nats$
such that $J \leq \overline{J}$ for each $(T,J) \in {\cal T}$.
\end{enumerate}

\end{proposition}

The proof of Proposition~\ref{prop:mess} uses the following lemmas.  
These lemmas involve the data $(C,f,D,g)$ of $\HS$ and the closed set $\A$. Their proofs are in \IfJournal{\cite[Appendix A]{HybridMPCTR}.}{Appendix~\ref{sec:proofs}.}

\begin{lemma}
\label{prop:Lyaplobound}
Suppose that item~\ref{item:mess1} 
in Proposition~\ref{prop:mess} 
holds.
Then, there exists a 
class-$\classK$ 
function~$\alpha$ 
such that every solution pair~$(x,u)$ to $\HS$ is such that,
for each $T$ satisfying $(T,0) \in \dom (x,u)$, $\Anorm{x(t,0)}\geq\alpha(\Anorm{x(0,0)})$  for all~ $(t,0)\in\dom (x,u) \cap ([0,T]\times\{0\})$.
\end{lemma}

\begin{lemma}
\label{prop:nice}
Suppose that item~\ref{item:mess2} in Proposition~\ref{prop:mess} holds.
Then, there exists a class-$\classK$ function~$\alpha$ such that every solution pair~$(x,u)$ to $\HS$ satisfies $\Anorm{x(t,0)}\geq\alpha(\Anorm{x(0,0)})$ 
for all~$t \in [0, \alpha(\Anorm{x(0,0)})]$ such that $(t,0)\in\dom (x,u)$.
\end{lemma}

We are now ready to present the proof of Proposition~\ref{prop:mess}.

\medskip
\begin{IEEEproof}
If either item~\ref{item:mess1} or item~\ref{item:mess2} holds, an application of  Lemma~\ref{prop:Lyaplobound} or Lemma~\ref{prop:nice} implies that there exists a class-$\classK$ function~$\alpha'$ such that every solution pair~$(x,u)$ to $\HS$ satisfies the following: for each $T$ such that $(T,0)\in\dom (x,u)$, we have $\Anorm{x(t,0)}\geq\alpha'(\Anorm{x(0,0)})$ for all~$t\in [0, \min\{T,\alpha'(\Anorm{x(0,0)})\}]$ such that~$(t,0)\in\dom (x,u)$. Take a feasible solution pair $(x,u)$. 
We have the following cases:
\begin{enumerate}[label=\Roman*),leftmargin=*]
\item
\label{item:prop:messProof-I}
If the solution jumps at $(t_1,0)$ for some~$t_1\leq\alpha'(\Anorm{x(0,0)})$, then~$(t_1,1)\in \dom (x,u)$ and \ref{item:jumbound} holds with~$\alpha=\alpha'$. 
\item
\label{item:prop:messProof-II}
Otherwise, either~$(\alpha'(\Anorm{x(0,0)}),0)\in\dom(x,u)$, in which case the solution evolves continuously beyond $x(t_1,0)$ and satisfies~\ref{item:flobound} with $j=j'=0$, or 
\item
\label{item:prop:messProof-III}
it terminates at time~$(T,0)$ with~$T\leq \alpha'(\Anorm{x(0,0)})$ and, hence, satisfies~\ref{item:termbound} with $J=0$, both with $\alpha=\alpha'$. 
\end{enumerate}

In the case of the time horizon~${\cal T}$ being bounded away from the origin considered in item \ref{item:P1suff}.\ref{item:P1suff-1}, \ref{item:bestbound} holds with $\alpha=\alpha'$ since each feasible solution pair $(x,u)$ is such that there exists $(t',j') \in \dom (x,u)$ such that $t'+j' \geq c > 0$, which, in turn, implies that there exists $(t,j) \in \dom (x,u)$ such that $t+j \leq t' + j'$ and $t' -t + j - j' \geq \alpha'(\Anorm{x(0,0)})$.

Finally, \ref{item:flobound} or \ref{item:termbound} hold under the conditions in
\ref{item:P3suff}. 
Without loss of generality, suppose that $\sigma$ in \eqref{eq:finiteact} is such that there exists~$c>0$ such that~$\sigma(r)\geq c$ for all~$r>0$. With $\varepsilon$ coming from \eqref{eq:finiteact}, let
\begin{equation}\label{eqn:varphi}
	\varphi(r):=\max_{r'\in[r/2,r]}\sigma(r') \quad \forall r \in(0,\varepsilon]
\end{equation}
which satisfies $\varphi(r)\geq c$ for all $r \in(0,\varepsilon]$.
Continuity of $\varphi$ follows via Berge's maximum theorem~\cite[Theorem~1.4.16]{aubin} by showing 
upper and lower semicontinuity of the set-valued map mapping each $r \in[0,\varepsilon]$ to
the compact interval~$[r/2,r]$.
Indeed, lower semicontinuity of the mapping can be verified directly, while upper semicontinuity follows from the fact that the mapping is outer semicontinuous (since its graph is closed relative to~$(0,\varepsilon]\times\reals$~\cite[Lemma~5.10]{65}) and locally bounded, due to~\cite[Lemma~5.15]{65}. Let~$\tau(0):=0$ and
\IfJournal{$\tau(r):= \frac{r}{2\varphi(r)} \text{if }r\in(0,\varepsilon]$, $\tau(r):= \frac{\varepsilon}{2\varphi(\varepsilon)} \text{if }r\in(\varepsilon,\infty)$}
{\[
	\tau(r):=
	\begin{cases}
	\frac{r}{2\varphi(r)} 										& \text{if }r\in(0,\varepsilon]\\
	\frac{\varepsilon}{2\varphi(\varepsilon)} & \text{if }r\in(\varepsilon,\infty)
	\end{cases}
\]}
and similarly, for every~$r\geq 0$, let~$\rho(r):=\min\{r,\varepsilon\}$. Note that~$\tau$ is continuous,
 positive definite,  upper bounded by $\frac{\varepsilon}{2\varphi(\varepsilon)}$, and lower bounded by zero. Hence, there exists a continuous function $\widetilde{\tau}:\realsgeq\to\realsgeq$ that is positive definite, increasing, and such that~$\widetilde{\tau}(r)\leq\tau(r)$ for all~$r\geq 0$.  It follows that, for any continuous solution pair~$(x,u)$ to $\HS$ with~$\Anorm{x(0,0)}> 0$, $t< \widetilde{\tau}(\rho(\Anorm{x(0,0)}))$ satisfying~$(t,0)\in\dom(x,u)$ implies~$\Anorm{x(t,0)}\geq\frac{\rho(\Anorm{x(0,0)})}{2}$.
To show this implication, first note that, by definition of $\widetilde{\tau}$, 
$\widetilde{\tau}(\rho(\Anorm{x(0,0)}))$ is upper bounded by $\tau(\rho(\Anorm{x(0,0)}))$ and that, since $\rho$ is upper bounded by $\varepsilon$, $\widetilde{\tau}(\rho(\Anorm{x(0,0)}))\varphi(\rho(\Anorm{x(0,0)})) \leq \frac{\rho(\Anorm{x(0,0)})}{2}$.
Using \eqref{eq:finiteact} and the construction of $\varphi$ above, we have $\Anorm{x(t,0)} \geq \Anorm{x(0,0)} - \varphi(\Anorm{x(0,0)})t$ for all $(t,0) \in \dom(x,u)$ such that $0<\Anorm{x(t,0)} \leq \varepsilon$.
Then, the lower bound on $\Anorm{x(t,0)}$ can be obtained by evaluating the latter inequality at $t = \widetilde{\tau}(\rho(\Anorm{x(0,0)}))$, leading to $\Anorm{x(t,0)}\geq\frac{\rho(\Anorm{x(0,0)})}{2}$.
In other words,~$\widetilde{\tau}\circ\rho(r)$ is a lower bound on the ordinary time to reach the~$(r/2)$-disc around~$\A$ from the~$r$-disc around~$\A$ via flows, for each $r \geq 0$. Finally, for every~$r\geq 0$, let~$\hat{\alpha}_D(r):=\widetilde{\alpha}_D(\rho(r)/2)$, and note that~$\hat{\alpha}_D$ and~$\rho$ are continuous, nondecreasing, and positive definite.

Let~$(x,u)$ be an optimal solution pair and~$\{t_j\}_{j=0}^{J+1}$ be a sequence such that~$\dom(x,u)=\cup_{j=0}^{J}([t_j,t_{j+1}]\times\{j\})$, with~$\Anorm{x(0,0)}>0$. First, suppose~$\widetilde{\alpha}_D(r)\geq r$ for all~$r\geq 0$. Then, following the same arguments as before, one can deduce that for any~$t< \widetilde{\tau}(\rho(\Anorm{x(0,0)}))$ and~$j\in\nats$ satisfying~$(t,j)\in\dom(x,u)$, we have~$\Anorm{x(t,j)}\geq \frac{\rho(\Anorm{x(0,0)})}{2}$, due to the fact that the distance of~$x$ to~$\A$ does not decrease during jumps, in light of \eqref{eqn:P3lowerBound}. Therefore, either there exists~$(s,i)\in\dom(x,u)$ such that~$s=\widetilde{\tau}(\rho(\Anorm{x(0,0)}))$ and every~$(t,j)\in\dom x$ with $t+j\leq s+i$ satisfies~$\Anorm{x(t,j)}\geq\frac{\rho(\Anorm{x(0,0)})}{2}$, or letting~$T:=t_{J+1}$,~$x(T,J)\geq \frac{\rho(\Anorm{x(0,0)})}{2}$. In other words, \ref{item:flobound} or \ref{item:termbound} holds~$\alpha=\alpha_0$, where
\begin{equation}
	\alpha_0(r):=\min\left\{\widetilde{\tau}(\rho(r)),\frac{\rho(r)}{2}\right\} \quad\forall r\geq 0
\label{eq:alpha0}
\end{equation}

Next, suppose that there exists~$J_{\max}\in\nats$ such that for every~$(T,J)\in\T$,~$J\leq J_{\max}$, but $\widetilde{\alpha}_D(r)\geq r$ does not necessarily hold for all~$r\geq 0$. Let~$T:=t_{J+1}$. If there exists~$j\in\{0,1,\dots, J\}$ such that~$t_{j+1}-t_j\geq \widetilde{\tau}(\rho(\Anorm{x(t_{j},j)}))$, let~$i$ be the smallest such~$j$, otherwise, let~$i:=J+1$. Then,
\IfJournal{$\Anorm{x(t_j,j)} \geq \widetilde{\alpha}_D(\Anorm{x(t_j,j-1)})
\geq \widetilde{\alpha}_D\left(\rho\left(\frac{\Anorm{x(t_{j-1},j-1)}}{2}\right)\right)=\hat{\alpha}_D(\Anorm{x(t_{j-1},j-1)})$}{
\begin{eqnarray*}
	\Anorm{x(t_j,j)}& &\hspace{-0.2in}\geq \widetilde{\alpha}_D(\Anorm{x(t_j,j-1)})\\
		& &\hspace{-0.7in} \geq \widetilde{\alpha}_D\left(\rho\left(\frac{\Anorm{x(t_{j-1},j-1)}}{2}\right)\right)=\hat{\alpha}_D(\Anorm{x(t_{j-1},j-1)})
\end{eqnarray*}}
for all~$j\in\{1,2,\dots,\min\{i,J\}\}$, and consequently,
\begin{equation}
\Anorm{x(t_j,j)}\geq \hat{\alpha}_D^{j}(\Anorm{x(0,0)})
 \quad \forall j\in\{0,1,\dots,\min\{i,J\}\}
\label{eq:badflow}
\end{equation}
where~$\hat{\alpha}_D^{j}$ denotes the~$j$-times composition of~$\hat{\alpha}_D$, e.g.,~$\hat{\alpha}_D^{0}$ is the identity function and~$\hat{\alpha}_D^{2}=\hat{\alpha}_D\circ\hat{\alpha}_D$. In addition, if~$i\leq J$, since~$\rho$ is nondecreasing, by~\eqref{eq:badflow},
\IfJournal{$\Anorm{x(t,i)}\geq \frac{\rho(\Anorm{x(t_i,i)})}{2}\geq \frac{\rho(\hat{\alpha}_D^{i}(\Anorm{x(0,0)})}{2}$}{
\[
	\Anorm{x(t,i)}\geq \frac{\rho(\Anorm{x(t_i,i)})}{2}\geq \frac{\rho(\hat{\alpha}_D^{i}(\Anorm{x(0,0)})}{2}
\]}
for all $t\in[t_i,t_i+\widetilde{\tau}(\rho(\Anorm{x(t_{i},i)}))]$,
which, as~$\widetilde{\tau}$ is increasing, implies that, 
for all  $t\in[t_i,t_i+\widetilde{\tau}(\rho(\hat{\alpha}_D^{i}(\Anorm{x(0,0)})))]$,
\begin{equation}
	\Anorm{x(t,i)}\geq \frac{\rho(\hat{\alpha}_D^{i}(\Anorm{x(0,0)})}{2}
\label{eq:niceflow}
\end{equation}
On the other hand, if~$i=J+1$, then again by~\eqref{eq:badflow},
\begin{eqnarray}\nonumber
	\Anorm{x(T,J)}=\Anorm{x(t_{J+1},J)}&\geq& \frac{\rho(\Anorm{x(t_J,J)})}{2}\\
	&\geq& \frac{\rho({\hat{\alpha}_D^{J}(\Anorm{x(0,0)})}}{2}
\label{eq:endpoint}
\end{eqnarray} 
\IfJournal{Given~$r\geq 0$, let
$\alpha_1(r):=\min_{j\in\{0,1,\dots,J_{\max}\}}\frac{\rho(\hat{\alpha}_D^{j}(r))}{2}$,
$\alpha_2(r):=\min_{j\in\{0,1,\dots,J_{\max}\}}\widetilde{\tau}(\rho(\hat{\alpha}_D^{j}(r)))$.
}{Given~$r\geq 0$, let
\begin{equation*}
	\begin{aligned}
	\alpha_1(r)		&:=\min_{j\in\{0,1,\dots,J_{\max}\}}\frac{\rho(\hat{\alpha}_D^{j}(r))}{2},\\
	\alpha_2(r)		&:=\min_{j\in\{0,1,\dots,J_{\max}\}}\widetilde{\tau}(\rho(\hat{\alpha}_D^{j}(r)))
	\end{aligned}
\end{equation*}
}
Then, if~$i\leq J$, by~\eqref{eq:niceflow},~$\Anorm{x(t,i)}\geq \alpha_1(\Anorm{x(0,0)})$ for every~$t\in[t_i,t_i+\alpha_2(\Anorm{x(0,0)})]$ satisfying~$(t,i)\in\dom(x,u)$. Otherwise, by~\eqref{eq:endpoint},~$\Anorm{x(T,J)}\geq \alpha_2(\Anorm{x(0,0)})$.  Combining this bound with the one in the previous case when~$\widetilde{\alpha}_D(r)\geq r$ for all~$r\geq 0$, it follows that~\ref{item:flobound} or \ref{item:termbound}  holds with~$\alpha(r)=\min_{k\in\{0,1,2\}}\alpha_k(r)$ for all~$r\geq 0$, where~$\alpha_0$ is defined in~\eqref{eq:alpha0}.
\end{IEEEproof}

\begin{remark}
Condition \eqref{eqn:P3lowerBound} is to be checked for all points
$(x,u)$ in $D$. 
When the interest is in 
\ref{item:flobound} holding for feasible 
and complete solution pairs, 
\eqref{eqn:P3lowerBound} can be relaxed and only be checked at points $(x,u) \in D$
such that $g(x,u)\in\closure(\Pi(C))\cup \Pi(D)$.
\end{remark}

\begin{example}[Bouncing Ball Control (revisited)]
\label{ex:bouncingballflows}
Consider the data of the bouncing ball system in Example~\ref{ex:bouncingball} and the total energy function~$W$ defined in \eqref{eqn:BBenergy}.  Following the control goal stated in Example~\ref{ex:bouncingball-revisited-2}, let
\begin{equation}\label{eqn:A-BB}
\A = \{x \in \reals^2:x_1\geq 0, W(x) = c^*\}
\end{equation}
with $c^* = \gamma h$ for some~$h\geq 0$. This definition of the set $\A$ corresponds to the limit cycle of the autonomous bouncing ball starting from~$x(0,0) = (h,0)$ when~$\lambda=1$. Since~$\A$ is compact and~$W$ is positive definite (on the domain~$x_1\geq 0$), the function~$x\mapsto (W(x)-\gamma h)^2=:\widetilde{V}(x)$ satisfies~\eqref{eq:candidate} for some~$\varepsilon>0$ and class-${\cal K}_{\infty}$ functions~$\widetilde{\alpha}_1$ and~$\widetilde{\alpha}_2$, as~$W$ is continuous. In fact, these functions can be chosen independently of~$\varepsilon>0$, due to \textit{radial unboundedness} of~$W$. Moreover, $\langle\nabla \widetilde{V}(x),f(x,u)\rangle=0$ for all~$(x,u)\in C$, so~\eqref{eq:lodescent} holds with arbitrarily large~$\varepsilon>0$.
\end{example}

Unlike item~\ref{item:mess1} in Proposition~\ref{prop:mess},
\eqref{eq:finiteact} does allow for finite-time convergence to~$\A$ during flows. Existence of a continuous function~$\sigma$ and~$\varepsilon>0$ satisfying~\eqref{eq:finiteact} is guaranteed when~$C = C'\times U$ for a closed set~$C'\subset \reals^n$ and compact set~$U\subset \reals^m$, provided~$\A$ is compact and~$f$ is continuous.

\begin{example}[Sample-And-Hold Control (revisited)]
\label{ex:sampleflow}
Given a compact set $U\subset\reals^{m_z}$, we revisit the sample-and-hold control system in Example~\ref{ex:sample} with flow and jump sets
\IfJournal{
$C=\{(x,u):\eta\in U,\tau_s\in[0,T_s]\}$ and $D=\{(x,u):\eta\in U,\tau_s=T_s,u\in U\}$}{\begin{equation*}
	\begin{aligned}
		C&=\{(x,u):\eta\in U,\tau_s\in[0,T_s]\}\\
		D&=\{(x,u):\eta\in U,\tau_s=T_s,u\in U\}
	\end{aligned}
\end{equation*}}
Note that the only difference between these sets and those in Example~\ref{ex:sample} is that the state component $\eta$ and the input $u$ are restricted to the compact set $U$.
Using these sets, we consider the compact set~$\A=\{0\}\times U\times[0,T_s]$, requiring convergence of $z$ and $\eta$ to zero. Item~\ref{item:mess2} in Proposition~\ref{prop:mess} holds as the flow map~$f$ is affine and does not depend on~$u$. 
We also note that~$\Anorm{g(x,u)}=|z|=\Anorm{x}$ for every~$(x,u)\in D$, so item~\ref{item:P3suff}.\ref{item:P3suff-1} holds with~$\widetilde{\alpha}_D$ as the identity. Item~\ref{item:P3suff}.\ref{item:P3suff-2} holds since each solution flows for some nonzero amount of ordinary time.  Then, by Proposition~\ref{prop:mess}, either \ref{item:flobound} or \ref{item:termbound} hold.
\end{example}

\NotForJournal{
\begin{example}[Bouncing Ball Control (revisited)]
\label{ex:bouncingballflow}
Given a compact set $U\subset\reals$, consider the data
\IfJournal{$C			=\{(x,u)\in\reals^2\times\reals:x_1\geq 0,u\in U\}$,
		$f(x,u)=(x_2,-\gamma+u)$,
		$D			=\{(x,u)\in\reals^2\times\reals:x_1=0 ,x_2\leq 0\}$,
		$g(x,u)=(0,-\lambda x_2)$.}{
\begin{equation}
	\begin{aligned}
		C			&=\{(x,u)\in\reals^2\times\reals:x_1\geq 0,u\in U\}\\
		f(x,u)&=(x_2,-\gamma+u)\\
		D			&=\{(x,u)\in\reals^2\times\reals:x_1=0 ,x_2\leq 0\}\\
		g(x,u)&=(0,-\lambda x_2)
	\end{aligned}
\end{equation}}
This data represents a modification of the bouncing ball system in Examples~\ref{ex:bouncingball} and~\ref{ex:bouncingballflows}, where the input~$u$ affects flows instead of jumps. Suppose~$\lambda>0$ and, similar to the definition of $\A$ in \eqref{eqn:A-BB} for  $c^* = \gamma h$, let
\IfJournal{$\A=\{x:x_1\geq 0, W(x)=c^*\}\cup\left\{(0,-\sqrt{2\gamma h}/\lambda)\right\}$.}{
\[
	\A=\{x:x_1\geq 0, W(x)=c^*\}\cup\left\{(0,-\sqrt{2\gamma h}/\lambda)\right\}.
\]}
Item~\ref{item:mess2} in 
Proposition~\ref{prop:mess}
 holds by continuity of~$f$ and compactness of~$U$, since~$\A$ is compact. 
To show that \ref{item:flobound} holds, it suffices to note that the jump map~$g$ is continuous, and~$\Anorm{g(x,u)}$ tends to infinity in~$D$ as~$\Anorm{x}$ tends to infinity, and importantly,~$\Anorm{g(x,u)}=0$ if and only if~$x=(0,-\sqrt{2\gamma h})\in D$. Note that if the point~$z=(0,-\sqrt{2\gamma h}/\lambda)\in D$ were to be excluded from~$\A$, this condition would not hold since~$W(g(z,u))=c^*$. Hence, item~\ref{item:P3suff}.\ref{item:P3suff-1} holds
and, by Proposition~\ref{prop:mess}, either \ref{item:flobound} or \ref{item:termbound} hold.
\end{example}
}

\section{Asymptotic Stability}
\label{sec:stability}

In this section, we establish asymptotic stability of the set $\A$ for the closed-loop system resulting from controlling the hybrid plant \eqref{eqn:Hp} with our hybrid MPC algorithm.  Since such a closed-loop system cannot be written as an autonomous system, we introduce the following asymptotic stability notion (cf. \cite{65,220}).

\begin{definition}
\label{def:stab}
The hybrid MPC algorithm is said to render the set~$\A$ asymptotically stable for~$\HS$ if~$\HS$ has unique state trajectories and the following hold:
\begin{enumerate}
	\item[(E)] There exists~$\delta>0$ such that for each~$\xinitial\in\Pi(C\cup D)$ satisfying~$\Anorm{\xinitial}\leq\delta$, there exists a solution pair~$(x,u)$ generated by the hybrid MPC algorithm starting from~$\xinitial$.
	\item[(S)] For each~$\varepsilon>0$, there exists~$\delta>0$ such that for each solution pair~$(x,u)$ generated by the hybrid MPC algorithm,~$\Anorm{x(0,0)}\leq\delta$ implies~$\Anorm{x(t,j)}\leq\varepsilon$ for all~$(t,j)\in\dom(x,u)$.
	\item[(A)] Each maximal solution pair~$(x,u)$ generated by the hybrid MPC algorithm is complete and satisfies $\lim_{t+j\to\infty, (t,j) \in \dom (x,u)}\Anorm{x(t,j)}=0$.
\end{enumerate}
\end{definition}

Item~(E) in Definition~\ref{def:stab} implies that,
from a neighborhood of $\A$,
 either no solution pair exists or, if $\xinitial\in\Pi(C\cup D)$, a solution pair~$(x,u)$ can be generated by the hybrid MPC algorithm. For this property to hold, it is clear that a solution to Problem~\ref{prob:optimum} must exist from those points---see Remark~\ref{remark:Existence}.  While not emphasized in Definition~\ref{def:stab}, note that in general, from any initial condition~$\xinitial\in\X$, there might be infinitely many solution pairs generated by the hybrid MPC algorithm. This is for two reasons. First, the minimizers of the cost functional~$\J$ might be nonunique; that is, given~$\xinitial\in\X$, there may be multiple solution pairs achieving the infimum in~\eqref{eq:value}. Second, assuming that the recursive feasibility property stated in Proposition~\ref{prop:T1T2} holds, the optimal control input~$u_i$ computed at optimization time $(T_i,J_i)$---see Section~\ref{sec:Implementation}---
 can be updated at \textit{any (hybrid) time} before the domain of $u_i$ ends, that is, at any~$(T_i+t,J_i+j)$, as long as~$(t,j)\in\dom u_i$.

\begin{theorem}
\label{thm:AS-MPC}
Given a hybrid plant $\HS = (C,f,D,g)$ as in \eqref{eqn:Hp},
the data $(\T,L_{C},L_{D},V,X)$ defining Problem~\ref{prob:optimum} with
$\T$ having
$t_0>0$ and~$J\geq 1$,
and a closed set $\A \subset \reals^n$, 
suppose that Assumptions~\ref{assmp:uniqueness}, \sout{\ref{assmp:OCP}}, \ref{assmp:OCP2}, and \ref{assmp:CLF} hold. 
Furthermore, suppose that there exists a class-$\classK$ function
$\alpha^*:\realsgeq\to\realsgeq$ such that 
the value function $\J^*$ in \eqref{eq:value} satisfies\footnote{Sufficient conditions for \eqref{eqn:ValueFunctionLowerBound} to hold are given in Section~\ref{sec:posdef-Sufficient}.}
\eqref{eqn:ValueFunctionLowerBound}
for each $\xinitial\in \X$ and that 
an optimal solution pair~$(x,u)\in\widehat{\sol}_{\HS}(\xinitial)$ to Problem~\ref{prob:optimum} exists.\footnote{Sufficient conditions for such existence are given in \cite{286}; see Remark~\ref{remark:Existence}.}
Then, the hybrid MPC algorithm renders the set~$\A$ asymptotically stable for the hybrid system~$\HS$  if any of the following statements are true: 
	\begin{enumerate}[label={(AS\arabic*)},leftmargin=*]
		\item\label{item:AS}
		$L_C$ and $L_D$ satisfy \ref{item:LC} and \ref{item:LD}, respectively;
		\item\label{item:AS-RelaxJump} Each solution pair generated by the hybrid MPC algorithm has persistent jumps and $L_D$ satisfies \ref{item:LD} with $\alpha_D$ being $\classKinfty$;
		\item\label{item:AS-RelaxFlow} Each solution pair generated by the hybrid MPC algorithm has persistent flows and $L_C$ satisfies \ref{item:LC} with $\alpha_C$ being $\classKinfty$.
	\end{enumerate}
\end{theorem}
\begin{IEEEproof}
The proof proceeds as follows.  First, a complete optimal solution with initial condition in a (small enough) neighborhood of the set $\A$ is constructed by induction.  Using properties of $\J$ established earlier, we show that $\J$ is a Lyapunov function certifying Lyapunov stability of $\A$.
Finally, attractivity of $\A$ is shown using the assumptions in the different
items in the statement.

Using Assumption~\ref{assmp:uniqueness}, 
Proposition~\ref{prop:uniqueness} implies that, for any given input, $\HS$ has unique state trajectories. 
From Proposition~\ref{prop:aroundA}, we have that $X\subset \X$.
Then, pick $\delta' > 0$  and note that 
existence of optimal solutions from $\X$
implies that, for every~$\xi\in\Pi(C\cup D) \cap X$ satisfying~$\Anorm{\xi}\leq\delta'$, there exists an optimal solution~$(x_0,u_0)$ starting from~$\xi$.
Now, 
the structure of the hybrid prediction horizon in Definition~\ref{def:Tstructure}
implies that there exists~$c>0$ such that~$\widehat{T}+\widehat{J}\geq c$ for any~$(\widehat{T},\widehat{J})\in\T$.  Hence, the terminal time~$(T_0,J_0)$ of~$(x_0,u_0)$ satisfies~$T_0+J_0\geq c$. In addition, due to the recursive feasibility property stated in Proposition~\ref{prop:T1T2} and existence of optimal solutions,
there exists an optimal solution pair~$(x_1,u_1)$ starting from~$x_0(T_0,J_0)$, whose terminal time~$(T_1,J_1)$ satisfies~$T_1+J_1\geq c$. Then, following Definition~\ref{def:solgenMPC},
the concatenation of $(x_0,u_0)$ and~$(x_1,u_1)$ form a solution pair. Specifically,~$(t,j)\mapsto (x'(t,j),u'(t,j))$ defined as follows is a solution pair to~$\HS$: for every $(t,j)\in\dom(x_0,u_0)$ such that~$t+j<T+J$,
\[
	(x'(t,j),u'(t,j))=(x_0(t,j),u_0(t,j))
\]
and for every $(t,j)\in\realsgeq\times\nats$ such that~$(t-T_0,j-J_0)\in\dom(x_1,u_1)$ and~$(t-T_0)+(j-J_0)<T_1+J_1$,
\[
	(x'(t,j),u'(t,j))=(x_1(t-T_0,j-J_0),u_1(t-T_0,j-J_0)).
\]
Setting~$(x'(T_0+T_1,J_0+J_1),u'(T_0+T_1,J_0+J_1))=(x_1(T_1,J_1),u_1(T_1,J_1))$, the terminal time~$(T',J')$ of~$(x',u')$ is such that~$T'+J'\geq 2c$. Similarly,~$(x',u')$ can be concatenated with an optimal pair $(x_2,u_2)$ from~$x_1(T_1,J_1)$ with terminal $(T_2,J_2)$ satisfying $T_2 + J_2 \geq c$. By induction, the resulting  solution pair starts from $\xi$ and constitutes a complete optimal solution generated by the hybrid MPC algorithm.

Next, we establish that the value function~$\J^*$ is a Lyapunov function. 
Since~$X\subset \X$ from Proposition~\ref{prop:aroundA}, by 
Assumption~\ref{assmp:OCP2}, there exists a class-$\mathcal{K}$ function, denoted~$\alpha_2$, such that~$\J^*(\tilde{x})\leq\alpha_2(\Anorm{\tilde{x}})$ for each~$\tilde{x}\in\X$ satisfying~$\Anorm{\tilde{x}}\leq\eta$, for some~$\eta>0$. 
By assumption, there exists a class-$\classK$ 
function $\alpha^*$ such that  
$\J^*(\tilde{x})\geq \alpha^*(\Anorm{\tilde{x}})$
for each $\tilde{x}\in \X$ such that 
an optimal solution pair~$(x,u)\in\widehat{\sol}_{\HS}(\tilde{x})$ exists. 
Hence, $\J^*$ is positive definite with respect to $\A$.
Moreover, by Lemma~\ref{lem:descent}, for each solution pair~$(x,u)$ generated by hybrid MPC, $\J^*(x(t,j))\leq\J^*(x(0,0))$ for each~$(t,j)\in\dom(x,u)$. 
Then, given $\varepsilon > 0$ there exists $\delta \in (0,\min\{\delta',\eta,{\alpha^*}^{-1}\circ\alpha_2(\varepsilon)\}]$
such that, for each~$\xi\in\Pi(C\cup D)$ satisfying~$\Anorm{\xi}\leq\delta$,
each (complete) optimal solution pair $(x,u)$ constructed as above
satisfies 
$$\alpha^*(|x(t,j)|_\A) \leq \J^*(x(t,j))\leq\J^*(\xi) \leq \alpha_2(|\xi|_\A)$$
for each~$(t,j)\in\dom(x,u)$, from where $|x(t,j)|_\A \leq \varepsilon$. 
Hence, $\A$ is stable in the sense of item~(S) in Definition~\ref{def:stab}.

To prove attractivity, using the complete solution pair~$(x,u)$ generated by hybrid MPC constructed at the beginning of this proof, we show that~$(t,j) \mapsto \J^*(x(t,j))$ approaches zero as $t + j$ approaches infinity.
First, let $\widetilde{\alpha}_C$ and $\widetilde{\alpha}_D$ be nonnegative 
and such that 
$L_C(x,u)\geq {\widetilde{\alpha}_C}(\Anorm{x})$ for each~$(x,u)\in C$ and
$L_D(x,u)\geq {\widetilde{\alpha}_D}(\Anorm{x})$ for each~$(x,u)\in D$, respectively.
From the inequality in Lemma~\ref{lem:descent}, we obtain,
for each~$(t,j)\in\dom(x,u)$,
\IfJournal{
\begin{multline}
	\J^*(x(t,j))\leq\J^*(x(0,0))
{-\sum_{i=0}^{j}\int_{s_i}^{s_{i+1}}\widetilde{\alpha}_C(\Anorm{x(s,i)}))\,ds} \\
									{-\sum_{i=0}^{j-1}}\widetilde{\alpha}_D(\Anorm{x(s_{i+1},i)}) \label{eqn:UpperBoundValueFunction}
\end{multline}
}
{
\begin{multline}
	\J^*(x(t,j))\leq\J^*(x(0,0)) {-\sum_{i=0}^{j}\int_{s_i}^{s_{i+1}}L_C(x(s,i),u(s,i))\,ds} \\
									{-\sum_{i=0}^{j-1}L_D(x(s_{i+1},i),u(s_{i+1},i))} \\
\leq\J^*(x(0,0))
{-\sum_{i=0}^{j}\int_{s_i}^{s_{i+1}}\widetilde{\alpha}_C(\Anorm{x(s,i)}))\,ds} \\
									{-\sum_{i=0}^{j-1}}\widetilde{\alpha}_D(\Anorm{x(s_{i+1},i)}) \label{eqn:UpperBoundValueFunction}
\end{multline}
}
where~$\{s_i\}_{i=0}^{j+1}$ is the sequence of generalized  jump times of the truncation of~$(x,u)$ up to time~$(t,j)$---see \eqref{eqn:GenJumpTimes}. 
Next, we show attractivity when items~\ref{item:AS}-\ref{item:AS-RelaxFlow} hold, one at a time.

First, suppose that item~\ref{item:AS} holds.
Then, \eqref{eqn:UpperBoundValueFunction} holds with $\widetilde{\alpha}_C = \alpha_C$ and $\widetilde{\alpha}_D = \alpha_D$ for the positive definite
functions $\alpha_C$ and $\alpha_D$ given in item~\ref{item:AS},
via 
\ref{item:LC}
and
\ref{item:LD}.
Let 
\begin{equation}\label{eqn:c-ASproof}
c:=\lim_{t+j\to\infty, (t,j) \in \dom (x,u)}\J^*(x(t,j))
\end{equation}
As~$\J^*(\xinitial)\geq\alpha^*(\Anorm{\xinitial})$ for all~$\xinitial\in \X$,
since $\alpha^*$ is a class-$\classK$
function,
$c=0$ implies~${\lim_{t+j\to\infty, (t,j)\in \dom(x,u)}\Anorm{x(t,j)}=0}$. 
Assuming the opposite, we have~$\J^*(x(t,j))\geq c>0$ for all~$(t,j)\in \dom(x,u)$. By continuity of~$\J^*$ 
obtained from Proposition~\ref{proposition:continuityValueFunction},
and the fact that~$\J^*(\A\cap\X)=0$ since $\J^*$ is positive definite relative to $\A$ by assumption, using $\alpha_2$, this means that there exists~$\ubar{\epsilon}>0$ so~$\Anorm{x(t,j)}\geq\ubar{\epsilon}$ for all~$(t,j)\in\dom(x,u)$. Let~$\{(T_k,J_k)\}_{k=0}^{\infty}$ be the sequence of optimization times of~$(x,u)$. By Lemma~\ref{lem:descent}, for each $k\in\nats$,
\begin{multline*}
	\J^*(x(T_{k+1},J_{k+1}))- \J^*(x(T_{k},J_{k}))\\
	\leq-\left((T_{k+1}-T_k)\alpha_C(\ubar{\epsilon})+(J_{k+1}-J_k)\alpha_D(\ubar{\epsilon})\right)
\end{multline*}
Summing over $k$'s in $\nats$ and taking limit on each side
\begin{multline*}
	c=\lim_{k\to\infty}\J^*(x(T_k,J_k))\\
	\leq \J^*(x(0,0))-\min\{\alpha_C(\ubar{\epsilon}),\alpha_D(\ubar{\epsilon})\}\lim_{k\to\infty}(T_k+J_k)
\end{multline*}
which contradicts~$c>0$ as~$\lim_{k\to\infty}(T_k+J_k)=\infty$ and~$\J^*(x(0,0))$ is finite. Therefore, it follows that~$c=0$, and consequently~$\lim_{t+j\to\infty, (t,j) \in \dom (x,u)}\Anorm{x(t,j)}=0$, which completes the proof.

Now, suppose item~\ref{item:AS-RelaxJump} is the one that holds.
It follows that \eqref{eqn:UpperBoundValueFunction} holds with $\widetilde{\alpha}_C = 0$ and $\widetilde{\alpha}_D = \alpha_D$ for the class-$\mathcal{K}_{\infty}$ function $\alpha_D$ given in item~\ref{item:AS-RelaxJump} via \ref{item:LD}.
With $c$ as in \eqref{eqn:c-ASproof}, we follow the steps of the proof for
the case when item~\ref{item:AS} holds, but in this case,
through recursive use of \eqref{eqn:UpperBoundValueFunction}, we obtain
\[
	\limsup_{t+j\to\infty}\J^*(x(t,j))\leq\lim_{j\to\infty}\J^*(x(0,0))
	-\sum_{j=0}^{\infty}\alpha_D(\Anorm{x(t_{j+1},j)})
\]
where~$\{t_j\}_{j=1}^{\infty}$ are the jump times of~$(x,u)$. Proceeding as above and using that $\alpha_D$ is a class-$\classKinfty$ function, assuming $c>0$ leads to a contradiction showing that~$c=0$. 

The case when item~\ref{item:AS-RelaxFlow} holds follows similarly.
In this case, the second term in \eqref{eqn:UpperBoundValueFunction} takes the form of a sum of integrals (for each period of flow) with the integrand~$\alpha_C(\Anorm{x(t,j)})$, where the union of the domains of integration is~$[0,\infty)$, and similar arguments apply, for which the fact that $\alpha_C$ is a class-$\classKinfty$ function is used.
\end{IEEEproof}

\NotForJournal{
\begin{remark}
The conditions in Theorem~\ref{thm:AS-MPC} involve
conditions on the data of the hybrid plant (Assumption~\ref{assmp:uniqueness}),
on the data defining Problem~\ref{prob:optimum}
(Assumptions~\ref{assmp:OCP},
\ref{assmp:OCP2}, and \ref{assmp:CLF}),
and on the existence of optimal pairs, which couples both data.
The required class-$\classK$ function $\alpha^*$ can be obtained from the results in 
Sections~\ref{sec:posdef} and~\ref{sec:posdef-Sufficient}.
Note that when the hybrid MPC algorithm generates solution pairs with persistent jumps, we do not insist on the flow cost to be positive definite with respect to the distance of the state~$x$ to the set~$\A$---similarly when the solution pairs have persistent flows.
Persistence of jumps or flows are usually inherited from the open-loop system~$\HS$. For example, for the bouncing ball in Examples~\ref{ex:bouncingball} and~\ref{ex:bouncingballflows}, every complete solution pair has persistent jumps (see Section~\ref{sec:intro}). Consequently, every solution pair generated by the hybrid MPC algorithm has persistent jumps, regardless of the OCP formulation.
\end{remark}
}

\section{Examples}
\label{sec:examples}

In this section, we implement the proposed hybrid MPC algorithm 
in the two examples introduced and revisited throughout the paper.
The files used in simulations are available at \url{github.com/HybridSystemsLab/HybridMPC}.

\begin{example}
\label{ex:BBfinal}
We illustrate the implementation of the hybrid MPC algorithm in Algorithm~\ref{alg:HybridMPC}
in the bouncing ball control system in Example~\ref{ex:bouncingball} with 
$\lambda \in (0,1]$ and 
constraints in $u$ defined by a compact set $U \subset \reals$.
Using the definition of the total energy $W$ in \eqref{eqn:BBenergy},
define the set to stabilize 
as the constant level of energy $c^* = \gamma h$ for some $h\geq 0$, namely
$$\A := \{x\in \Pi(C):W(x)=c^*\}$$
where $\Pi(C) = \{x  \in \reals^2 : x_1\geq 0\}$; see Examples~\ref{ex:bouncingball-revisited-2} and \ref{ex:bouncingballflows}. 
Next, we define the data $(\T,L_{C},L_{D},V,X)$ associated with hybrid MPC  (see Section~\ref{sec:overview}) so as to satisfy the assumptions in Theorem~\ref{thm:AS-MPC}.

The hybrid prediction horizon $\T$ is chosen as in \eqref{eq:genericT}, with $\delta > 0$ and $N \in \{1,2,\ldots\}$.  This set is such that 
there exists $c > 0$ such that $T+J \geq c$ for each $(T,J) \in \T$,
from where it follows that item a.i in Proposition~\ref{prop:mess} holds.

The terminal set $X$ is chosen as~$X :=\Pi(C)$.
Pick~$\theta\in(0,(2/\pi)(1-\lambda^4)/(1+\lambda^4))$.
We define the terminal cost as
\[
	V(x) := (1+\theta\arctan x_2)\widetilde{V}(x)
	\quad \forall x\in X
\]
where 
$\widetilde{V}$ is
defined in Example~\ref{ex:bouncingballflows}.
By the choice of~$\theta$,
$V$ is positive definite with respect to~$\A$
and the fact that each $(x,u) \in D$ is such that $x_1=0$.
Moreover, Assumption~\ref{assmp:OCP2} holds with $\alpha$ radially unbounded and any $\varepsilon > 0$.
 From Example~\ref{ex:bouncingballflows} and the properties of $\T$,
exploiting item~1 and item~a in Proposition~\ref{prop:mess},
we have that
\ref{item:bestbound}, \ref{item:flobound}, and \ref{item:jumbound} 
hold for each solution to $\HS$. 
\sout{We denote the  function satisfying \ref{item:bestbound} as $\tilde{\alpha}$.}

To define the hybrid closed-loop system~$\HS_{\kappa}$ in~\eqref{eq:Hk}, we proceed as follows.
Since the bouncing ball control system does not have inputs that affect the flow map,
we choose the state-feedback law~$\kappa_C$ as an arbitrary function with its range in~$U_C$.
The state-feedback law $\kappa_D$ is defined as
\[
	\kappa_D(x):=\max\left\{\lambda x_2+\sqrt{2\gamma h},0\right\} \quad \forall x\in \reals^2
\]
The resulting maps $f_{\kappa}$, $g_{\kappa}$ and the sets
$D_{\kappa}$, $C_{\kappa}$ defining the data of \eqref{eq:Hk} are given by
${C_{\kappa} := \Pi(C)}$, ${f_{\kappa}(x):=(x_2,-\gamma)}$ for each $x \in C_{\kappa}$, $D_{\kappa}:=\{x  \in \reals^2 : x_1= 0, x_2\leq 0\}$,
and
\[
	g_{\kappa}(x) :=
	\begin{cases}
	(0,-\lambda x_2) & \textrm{if }x_2\leq-\sqrt{2\gamma h}/\lambda\\
	(0,\sqrt{2\gamma h}) & \textrm{otherwise}
	\end{cases}
\]
for each $x \in D_{\kappa}$, respectively.

Next, we design the flow cost $L_C$ and jump cost $L_D$ so that the CLF condition in Assumption~\ref{assmp:CLF} holds with $V$ and $X$ as designed above. To that end, we define the flow cost as
\[
	L_C(x,u)	:=\theta\gamma\frac{(W(x)-\gamma h)^2}{1+2W(x)} \quad \forall (x,u)\in C
\]
The jump cost is chosen as follows:
\[
	L_D(x,u)=\frac{1}{2}\left(1-\frac{\theta\pi}{2}\right)\gamma h\left(x_2+\sqrt{2\gamma h}\right)^2
\]
if~$ x_2\geq -\sqrt{2\gamma h}/\lambda$, and
\begin{multline*}
	L_D(x,u)=\min\Big\{\frac{1}{2}\left(1-\frac{\theta\pi}{2}\right)\gamma h\left(x_2+\sqrt{2\gamma h}\right)^2,\\
	\left(1-\frac{\theta\pi}{2}\right)\left(\frac{x_2^2}{2}-\gamma h\right)^2-\left(1+\frac{\theta\pi}{2}\right)\left(\frac{\lambda^2x_2^2}{2}-\gamma h\right)^2\Big\}
\end{multline*}
otherwise, for all~$(x,u)\in D$.

Next, we check that the CLF conditions in \eqref{eq:CLF} hold. 
For starters, $L_C$ satisfies \ref{item:LC} due to radial unboundedness of~$W$ in~$\Pi(C)$, that is, $W(x)$ approaches infinity as~$\Anorm{x}$ approaches infinity in~$\Pi(C)$. Furthermore, since~$\langle\nabla W(x),f_{\kappa}(x)\rangle=0$, it can be shown by direct computation that~$\langle\nabla V(x),f_{\kappa}(x)\rangle\leq-L_C(x,\kappa_C(x))$ for all~$x\in C_{\kappa}$. 
Following similar steps, $L_D$ is 
nonnegative and satisfies \eqref{eq:CLF}.
\IfJournal{Details are available in \cite{HybridMPCTR}.}{More specifically, the condition on~$\theta$ ensures that the coefficient of the~$x_2^4$ term is positive and~$L_D$ is positive definite, in the sense that, for each $(x,u)\in D$,~$L_D(x,u)$ is equal to zero if~$x_2=-\sqrt{2\gamma h}$ and positive otherwise.  It can be shown that~$V(g_{\kappa}(x))-V(x)\leq-L_D(x,\kappa_D(x))$ for each~$x\in D_{\kappa}$. In particular, since~$g_{\kappa}(x)=(0,\sqrt{2\gamma h})\in\A$ for each~$x\in D_{\kappa}$ satisfying~$x_2\geq-\sqrt{2\gamma h}/\lambda$, and~$V(x)=0$ for each~$x\in\A$,
\begin{multline*}
	V(g_{\kappa}(x))-V(x)\\
	\leq-\frac{1}{4}\left(1-\frac{\theta\pi}{2}\right)\left(x_2-\sqrt{2\gamma h}\right)^2\left(x_2+\sqrt{2\gamma h}\right)^2\\
	\leq -\frac{1}{2}\left(1-\frac{\theta\pi}{2}\right)\gamma h\left(x_2+\sqrt{2\gamma h}\right)^2\leq-L_D(x,\kappa_D(x))
\end{multline*}
for each~$x\in D_{\kappa}$ satisfying~$x_2\geq-\sqrt{2\gamma h}/\lambda$, where we use the fact that~$x_1=0$ and~$x_2\leq 0$ on~$D_{\kappa}$. On the other hand, given any~$x\in D_{\kappa}$ satisfying~$x_2\leq-\sqrt{2\gamma h}/\lambda$, since~$\kappa_D(x)=0$, we have
\begin{multline*}
	V(g_{\kappa}(x))-V(x)\\
	=(1+\theta\arctan (-\lambda x_2))\left(\frac{\lambda^2x_2^2}{2}-\gamma h\right)^2\\
	-(1+\theta\arctan x_2)\left(\frac{x_2^2}{2}-\gamma h\right)^2\\
	\leq(1+\theta\pi /2)\left(\frac{\lambda^2x_2^2}{2}-\gamma h\right)^2-(1-\theta\pi /2)\left(\frac{x_2^2}{2}-\gamma h\right)^2\\
	\leq -L_D(x,\kappa_D(x))
\end{multline*}
Hence,~\eqref{eq:CLF} holds. 
}

Note that since $L_{C}$ 
satisfies Assumption~\ref{assmp:OCP},
and 
\ref{item:flobound}
holds for each solution to $\HS$,
from Theorem~\ref{thm:PDofJ} and \cite[Section 6.2]{286} we have that 
there exists a class-$\classK$ 
function
$\alpha^*:\realsgeq\to\realsgeq$ such that 
the value function $\J^*$ in \eqref{eq:value} satisfies \eqref{eqn:ValueFunctionLowerBound}
for each $\xinitial\in \X$ and that 
an optimal solution pair~$(x,u)\in\widehat{\sol}_{\HS}(\xinitial)$ exists from each such $\xinitial$. 
Furthermore, one can adapt the steps  in~\cite[Example~2.12]{65} to show that~$X$ is 
forward pre-invariant for~$\HS_{\kappa}$,
leading to $\X = X$, which is equal to $\Pi(C) \cup \Pi(D)$, 
implying that every maximal solution to the closed-loop system
using hybrid MPC is complete.

Next, we simulate the bouncing ball controlled by hybrid MPC.
The parameters used in the simulation are~$\gamma=9.81 \mbox{m/s}$,~$\lambda=0.9$,~$h=3\  \mbox{m}$,
which leads to $c^* = 1.835\ \mbox{m}^2\mbox{/s}$, and a prediction horizon of the form~\eqref{eq:genericT} with~$N=5$ and $\delta=0.5$.
Problem~\ref{prob:optimum} is solved in MATLAB using the \texttt{fmincon} function by converting it into a finite-dimensional nonlinear program, for which we exploit the fact that the total energy~$W$---and therefore the flow cost~$L_C$---is invariant during flows (namely, $W$ is conserved during flows) and that the state trajectory of the bouncing ball during flows can be written in closed form.
The closed-loop system is simulated using the Hybrid Equations Toolbox~\cite{74}. After every optimization event, if the predicted state trajectory jumps, the next optimization is triggered at the next jump time; otherwise, the next optimization occurs at the terminal time of the predicted state trajectory. 
\IfJournal{
Figure~\ref{fig:simu2} shows the total energy associated with trajectories from
initial conditions 1) $x(0,0) = (0,0)$, shown in blue,  2) $x(0,0) = (1,-1)$, shown in black, 3) $x(0,0) = (3,4)$, shown in green.
The trajectory from  $x(0,0) = (0,0)$ exhibits a jump at $(t,j) = (0,0)$ since a jump is possible from such initial condition. The hybrid MPC controller injects energy equal to $c^*$ at that jump, and the trajectory converges and stays at the desired level of energy from there on. 
The trajectory from $x(0,0) = (1,-1)$ exhibits a similar behavior, but after the first jump at $(t_1,0)$ where $t_1 \approx 0.36 \mbox{sec}$. 
Such ``dead beat'' convergence property is due to the only constraint on the control input being $u \geq 0$.
The effect of this constraint is visualized in the trajectory from 
$x(0,0) = (3,4)$, which has initial energy much larger than $c^*$
and, due to the control input having to be nonnegative, 
cannot remove energy in one jump.}
{Figure~\ref{fig:simu} depicts closed-loop state trajectories for the controlled bouncing ball
from three initial conditions: 
1) $x(0,0) = (0,0)$, shown in blue, 
2) $x(0,0) = (1,-1)$, shown in black, 
3) $x(0,0) = (3,4)$, shown in green.
Figure~\ref{fig:simu2} shows the total energy associated with these trajectories.
The trajectory from  $x(0,0) = (0,0)$ exhibits a jump at $(t,j) = (0,0)$ since a jump is possible from such initial condition. The hybrid MPC controller injects energy equal to $c^*$ at that jump, and the trajectory converges and stays at the desired level of energy from there on. 
The trajectory from $x(0,0) = (1,-1)$ exhibits a similar behavior, but after the first jump at $(t_1,0)$ where $t_1 \approx 0.36 \mbox{sec}$. 
Such ``dead beat'' convergence property is due to the only constraint on the control input being $u \geq 0$.
The effect of this constraint is visualized in the trajectory from 
$x(0,0) = (3,4)$, which has initial energy much larger than $c^*$
and, due to the control input having to be nonnegative, 
cannot remove energy in one jump.}

\NotForJournal{
\begin{figure}[tbhp]
\centering
\includegraphics[width=1\columnwidth,trim=1.6cm 0.3cm 1.6cm 0.7cm,clip]{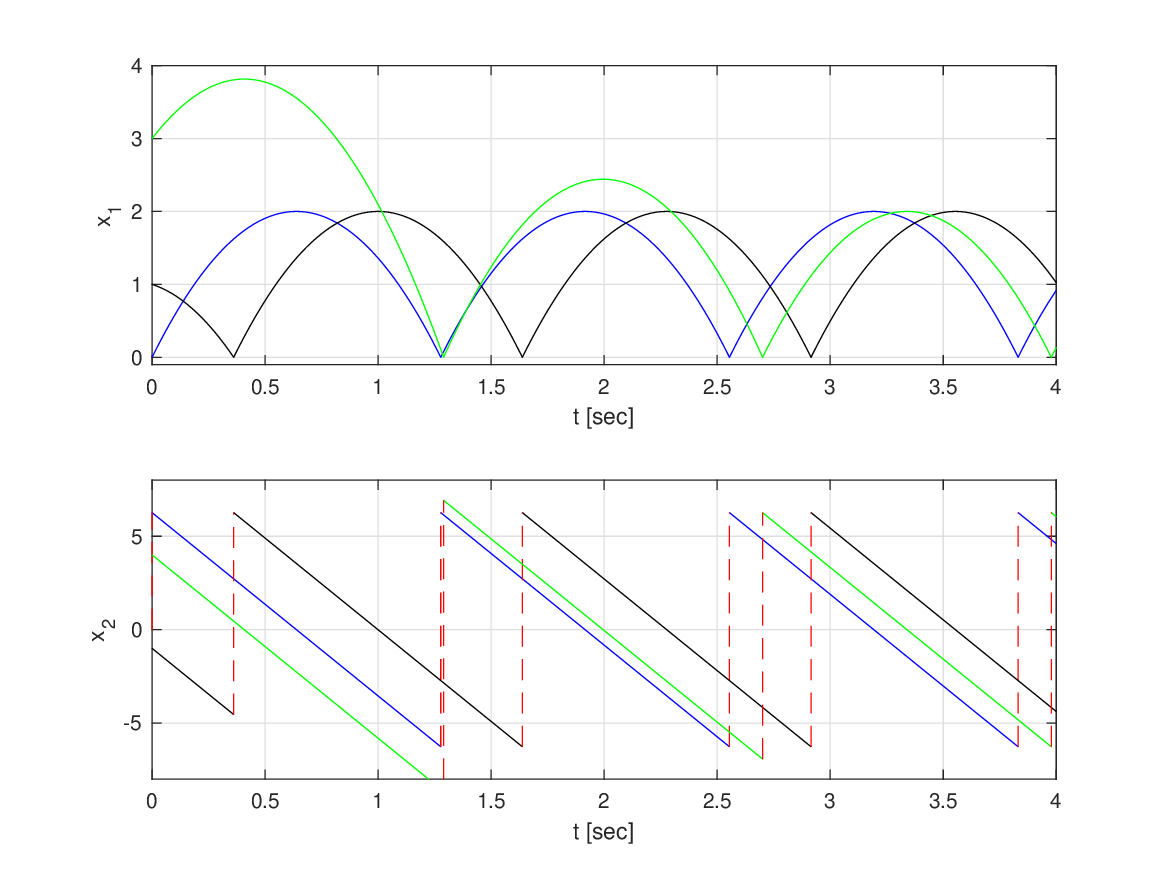}
\caption{Position and velocity trajectories of the bouncing ball controlled by hybrid MPC projected onto ordinary time~$t$.}
\label{fig:simu}
\end{figure}
}

\begin{figure}[tbhp]
\centering
\vspace{-0.15in}
\includegraphics[width=1\columnwidth,trim=1.3cm 0cm 1.6cm 0.3cm,clip]{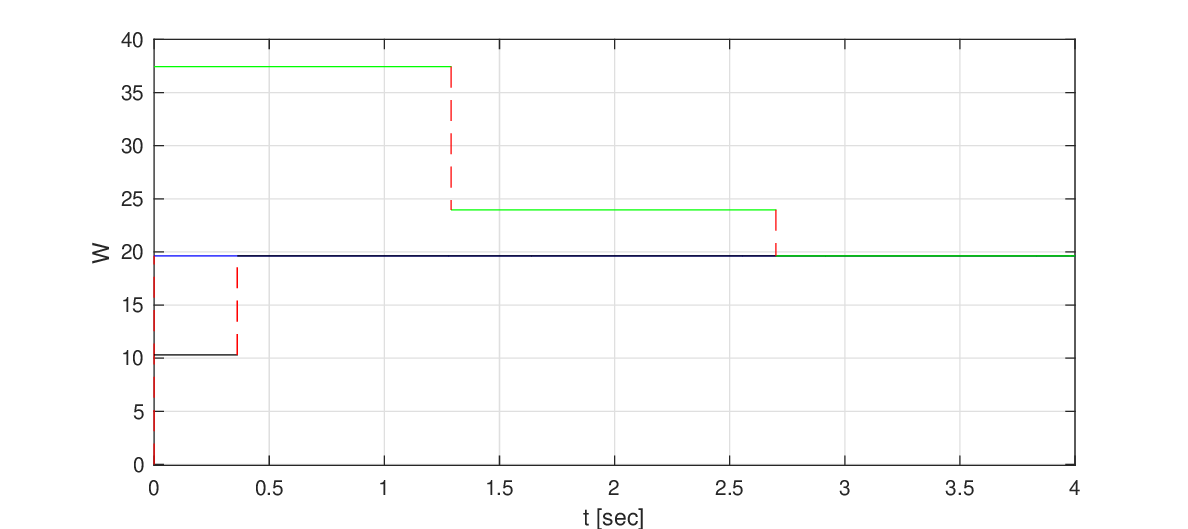}
\vspace{-0.2in}
\caption{Evolution of the total energy of the bouncing ball controlled by hybrid MPC 
projected onto ordinary time~$t$.}
\label{fig:simu2}
\vspace{-0.15in}
\end{figure}
\end{example}

\begin{example}[Sample-and-Hold Control (revisited)]
\label{ex:samplefinal}
We synthesize our hybrid MPC algorithm to control the system in  Example~\ref{ex:sample} using the sample-and-hold paradigm.  With the flow and jump sets given in Example~\ref{ex:sample}, and the compact set $\A$ and linear plant dynamics defined by $\tilde{f}$ in Example~\ref{ex:samplejump},
since $\A$ is such that both $z$ and $\eta$ are zero, the control objective for hybrid MPC 
is to drive the state $z$ and the input $\eta$ of $\dot z = A z + B \eta$ to zero asymptotically, and render the set $\A$ Lyapunov stable.  To this end, following the structure of the stage cost used in the linear quadratic regulator, 
the stage cost for flows is chosen as
\IfJournal{$L_C(x,u) := x_1^\top Q_C x_1$ for all $(x,u) \in C$,}{
$$L_C(x,u) := x_1^\top Q_C x_1 \qquad \forall (x,u) \in C$$}
where $x_1 := (z,\eta)$, $x_2 := \tau_s$, and $Q_C$ is a symmetric positive definite matrix that is to be chosen (see below).
Assuming that jumps do not accrue cost, the stage cost for jumps is chosen to be the zero function.
The terminal cost $V$ is chosen as
$V(x) = \exp(-\sigma x_2) \left(x_1^\top \exp(A_f^\top(T_s-x_2)) P  \exp(A_f(T_s-x_2))  x_1 \right)$,
where $\sigma > 0$, $A_f=\matt{A & B \\0 & 0}$, and $P$ is a symmetric positive definite matrix, chosen next.

The matrices $Q_C$ and $P$ are chosen as follows.
Given matrices $A \in \reals^{n_z} \times \reals^{n_z}$ and $B \in \reals^{n_z} \times \reals^{m_z}$, and the compact set $U \subset\reals^{m_z}$, choose a matrix $K \in \reals^{m_z} \times \reals^{n_z}$, a positive definite symmetric matrix $P \in \reals^{n_z} \times \reals^{n_z}$,  $c > 0$, and $T_s > 0$ such that
$$
H(K)^\top P H(K) - P < 0
$$
$X := \defset{x}{x_2 \in [0,T_s], V(x) \leq c}$ is such that
\IfJournal{$(x_1,x_2) = (z,\eta,\tau_s) \in X$ implies $Kz \in U$}
{
$$
(x_1,x_2) = (z,\eta,\tau_s) \in X \ \ \Rightarrow \  \ Kz \in U
$$
}
and, for each $s \in [0,T_s]$,
\begin{equation}\label{ex:sampleLcCondition}
Q_C \leq \sigma\exp(-\sigma s) ^\top \exp(A_f^\top(T_s-s)) P  \exp(A_f(T_s-s)) 
\end{equation}
where $H(K) := \exp(A_f T_s) A_g(K)$ and $A_g(K) := \matt{I & 0 \\ K & 0}$.
The choice of $V$ stems from the fact that the hybrid closed-loop system resulting from using the state-feedback controller $u = \kappa_D(x):=K z$ is such that
$\langle\nabla V(x),f_{\kappa}(x)\rangle \leq - \sigma V(x)$ on $C_{\kappa}$
and that
$V(g_{\kappa}(x))-V(x) = x_1^\top (H(K)^\top P H(K) - \exp(-\sigma x_2)P) x_1$.
Then, Assumption~\ref{assmp:CLF} holds if $L_c(x,u) \leq \sigma V(x)$ on $C$, which, since $f$ and $L_C$ do not depend on $u$, is 
guaranteed from \eqref{ex:sampleLcCondition}.
Recall that Example~\ref{ex:sampleflow}  establishes via Proposition~\ref{prop:mess} that \ref{item:flobound} holds.
Then,
from item~\ref{item2-pd} in
Theorem~\ref{thm:PDofJ}, \eqref{eqn:ValueFunctionLowerBound} holds.\NotForJournal{\footnote{Note that Example~\ref{ex:samplejump}  establishes that 
\ref{item:jumbound} holds with a class-$\classKinfty$ function $\alpha$, but since
$L_D$ is the zero function, 
Theorem~\ref{thm:PDofJ} cannot be employed to establish that \eqref{eqn:ValueFunctionLowerBound} holds.}}
Note that from the construction of the hybrid plant in Example~\ref{ex:sample}, every complete solution has persistent flows and \ref{item:AS-RelaxFlow} holds. 
Since the other assumptions in Theorem~\ref{thm:AS-MPC} hold, 
the hybrid MPC algorithm renders $\A$ asymptotically stable for $\HS$ in Example~\ref{ex:sample}.
\end{example}

\NotForJournal{
\begin{example}[Thermostat Control]
The evolution of the temperature of a room controlled by a heater that can either be {\em on} or {\em off} is given by
$$
\dot z = - z + z_o + z_\Delta q,
$$
where $z \in \reals$ is the temperature of the room, 
$z_o$ denotes the effective temperature outside of the room,
$z_\Delta$ represents the capacity of the heater,
and the logic state $q \in \{0,1\}$ represents whether the heater is {\em on} or {\em off}.  
When $q = 1$, the heater is 
{\em on} and when $q=0$, the heater is {\em off}.
Our goal is to design a control algorithm that 
steers the temperature to a desired temperature range $$[z_{\min},z_{\max}]$$ 
where $z_{\min} < z_{\max}$,
where $z_o <  z_{\min} < z_{\max} < z_o + z_\Delta$.
To facilitate the formulation of the optimization problem, we treat $q$ as 
an additional logic state and incorporate an input, denoted  $u \in\{0,1\}$, 
playing the role of the decision variable 
for the optimization problem.
The resulting system is given as in \eqref{eqn:Hp}, 
with state $x = (q,z) \in  \{0,1\} \times \reals$,
input $u \in \{0,1\}$, 
and data $(C,f,D,g)$ given by
\[
	\begin{aligned}
		C					&:=\{0,1\}\times \reals \times \{0\},\\
		f(x,u)	&:=	(0,-z + z_o + z_{\Delta} q)	\quad \forall (x,u)\in \{0,1\} \times \reals \times \{0,1\},\\
		D					&:=\{0,1\}\times \reals \times \{1\},\\
		g(x,u) 	&=	(1-q,z)											\quad \forall (x,u)\in\{0,1\} \times \reals  \times \{0,1\}.
	\end{aligned}
\]
With this data, flows of the plant are allowed when $u$ is zero.
In this regime, the temperature $z$ evolves according to its continuous-time
model and $q$ remains constant as $f$ leads to $\dot q = 0$.
At jumps, which are triggered when $u$ is equal to one, 
the update law $1 - q$ toggles the value of $q$ from $0$ to $1$ or from $1$ to $0$.
According to the goal stated above, the set to stabilize is given by the range of desired temperature, namely, 
$$\A := \defset{x}{z \in [z_{\min},z_{\max}], q \in \{0,1\}} 
$$

Next, we define the data $(\T,L_{C},L_{D},V,X)$ associated with hybrid MPC  (see Section~\ref{sec:overview}) so as to satisfy the assumptions in Theorem~\ref{thm:AS-MPC}.
The hybrid prediction horizon $\T$ is chosen as in \eqref{eq:genericT}, with $\delta > 0$ and $N \in \{1,2,\ldots\}$.  This set is such that there exists $c > 0$ such that $T+J \geq c$ for each $(T,J) \in \T$, from where it follows that item a.i in Proposition~\ref{prop:mess} holds. 
The flow cost $L_C$ is defined to penalize the cost of flow when the temperature is not in the desired
range, as follows: for each $(x,u) \in C$,
$$L_C(x,u) := \left\{
\begin{array}{cl}
a_{\scriptsize \rm hot}(z - z_{\max}) + a_{\scriptsize \rm on} q & \mbox{ if } z \geq z_{\max} \\
0 & \mbox{ if } z \in (z_{\min},z_{\max}) \\
a_{\scriptsize \rm cold}(z_{\min} - z)  & \mbox{ if } z \leq z_{\min}
\end{array}
\right.$$
where $0 < a_{\scriptsize \rm hot} \leq z_{\max}-z_o$ penalizes the temperature when larger than the desired maximum,
$0<  a_{\scriptsize \rm cold}$ penalizes the temperature when smaller than the desired minimum,
and $0 \leq a_{\scriptsize \rm on}$ penalizes the on position of the heater when the temperature is larger than the desired maximum. The jump cost $L_D$ also penalizes the temperature when outside the desired range and,
in addition, aims at reducing the number of switches of the heater.
It is chosen as follows: for each $(x,u) \in D$,
$$L_D(x,u) := \left\{
\begin{array}{ll}\displaystyle
b_{\scriptsize \rm hot}\frac{(z - z_{\max})^2}{2} + b_{\scriptsize \rm on,hot} (1-q) & \\ \hspace{1.4in} \mbox{ if } z \geq z_{\max} & \\ \displaystyle
b_{\scriptsize \rm ss}\frac{(z_{\max} - z)(z - z_{\min})}{2} + b_{\scriptsize \rm on, ss} (1-q) & \\ \hspace{1.4in} \mbox{ if } z \in (z_{\min},z_{\max}) \\ \displaystyle
b_{\scriptsize \rm cold}\frac{(z - z_{\min})^2}{2} & \\ \hspace{1.4in} \mbox{ if } z \leq z_{\min} & 
\end{array}
\right.$$
where $0 < b_{\scriptsize \rm hot} \leq 1$ penalizes the temperature when larger than the desired maximum,
$b_{\scriptsize \rm on,hot}$ penalizes turning heater on,
$b_{\scriptsize \rm ss}$ penalizes temperature variations within the desired range,
$b_{\scriptsize \rm on,ss}$ penalizes turning heater on if  within the desired range,
and $b_{\scriptsize \rm cold}$ penalizes the temperature when smaller than the desired minimum.

The terminal constraint set $X$ is defined as 
$$
X := \widetilde{C} \cup \widetilde{D}
$$
where the sets $\widetilde{C}$ and $\widetilde{D}$ are given by
$$\widetilde{C} : = \left(\{0\}\times \widetilde{C}_0\right) \cup \left(\{1\}\times \widetilde{C}_1\right)$$
$$\widetilde{D} : = \left(\{0\}\times \widetilde{D}_0\right) \cup \left(\{1\}\times \widetilde{D}_1\right)$$
with
$$\widetilde{C}_0 := \defset{z \in \reals}{z \geq z_{\min}}, \ \widetilde{C}_1 := \defset{z \in \reals}{z \leq z_{\max}}$$
$$\widetilde{D}_0 := \defset{z \in \reals}{z \leq z_{\min}},\  
\widetilde{D}_1 := \defset{z \in \reals}{z \geq z_{\max}}$$
This construction of $X$ constrains the temperature $z$ and 
the heater state $q$ to assure stabilizability of the set $\A$.
The terminal cost function $V$ is defined as 
$$V(x) := \left\{
\begin{array}{cl}\displaystyle
\frac{(z - z_{\max})^2}{2}(1+q) & \mbox{ if } z \geq z_{\max} \\
0 & \mbox{ if } z \in (z_{\min},z_{\max}) \\
\displaystyle
\frac{(z - z_{\min})^2}{2}(2-q) & \mbox{ if } z \leq z_{\min}
\end{array}
\right.$$
where the terms
$(1 + q)$ and $(2-q)$ penalize being at the ``wrong'' mode.
Picking the feedback $\kappa_C$ as 
$\kappa_C(x) = 0$ if $x \in \widetilde{C}$, and one otherwise,
and the feedback $\kappa_D$ as
$\kappa_D(x) = 1$ if $x \in \widetilde{D}$, and zero otherwise,
we obtain the following properties for the variation of $V$ along flows
and jumps of the hybrid closed-loop system $\HS_{\kappa}$ in \eqref{eq:Hk}: for every $x \in X \cap \widetilde{C} (= \widetilde{C})$ we have
that $\langle \nabla V(x), f(x,\kappa_C(x))\rangle$ is given by
$$
\left\{
\begin{array}{cl}
-(z_{\max}- z_o)(z - z_{\max})  & \mbox{ if } z \geq z_{\max} \\
0  & \mbox{ if } z \in (z_{\min},z_{\max}) \\
-(z_o + z_\Delta - z_{\min})(z_{\min} - z) & \mbox{ if } z \leq z_{\min}
\end{array}
\right.
$$
and, for every $x\in X\cap \widetilde{D} (=\widetilde{D})$, 
$V(g(x,\kappa_D(x))) - V (x)$ is equal to
$$
 \left\{
\begin{array}{cl}\displaystyle
-\frac{(z-z_{\max})^2}{2}  & \mbox{ if } z \geq z_{\max} \\
0  & \mbox{ if } z \in (z_{\min},z_{\max}) \\ \displaystyle
-\frac{(z-z_{\min})^2}{2}  & \mbox{ if } z \leq z_{\min}
\end{array}
\right.
$$
It follows that, with the choices of $L_C$ and $L_D$ above,
Assumption~\ref{assmp:CLF} holds and that~$X$ is 
forward pre-invariant for~$\HS_{\kappa}$.
From item~\ref{item:flowjumpbound} in Theorem~\ref{thm:PDofJ},
there exists a class-$\classK$ function
$\alpha^*:\realsgeq\to\realsgeq$ such that 
the value function $\J^*$ in \eqref{eq:value} satisfies \eqref{eqn:ValueFunctionLowerBound}
for each $\xinitial\in \X$.
Finally, from \cite[Section 6.1]{286}, we have that 
an optimal solution pair~$(x,u)\in\widehat{\sol}_{\HS}(\xinitial)$ exists from each such $\xinitial$.
\end{example}
}

\section{Conclusion}
\label{sec:conclusion}

This paper formalizes an MPC framework for hybrid dynamical systems.
The formulation allows to quantify the cost and predict trajectories over hybrid time domains.
Conditions guaranteeing recursive feasibility, positive definiteness of the value function and its decrease along solutions,
leading to asymptotic stability of a closed set are proposed.
The conditions involve the data of the system and the OCP, leading to checkable design conditions.
Future work will focus on the hybrid MPC design to satisfy other dynamical properties of interest, such as safety, temporal logic, and robustness.  An add-on to our Hybrid Equations Toolbox (HyEQ) for Matlab/Simulink \cite{74} to compute numerical solutions of Problem~\ref{prob:optimum} is currently being developed.
Employing the results in \cite{camilli2008control} to relax the conditions for positive definiteness of the value function as suggested by an anonymous reviewer is also part of future work.

\ifCLASSOPTIONcaptionsoff
  \newpage
\fi

\bibliographystyle{IEEEtran}
\bibliography{bibs/long,bibs/IEEEabrv,bibs/HybridMPC,bibs/RGSweb}

@article{camilli2008control,
  title={Control Lyapunov functions and Zubov's method},
  author={Camilli, Fabio and Gr{\"u}ne, Lars and Wirth, Fabian},
  journal={SIAM Journal on Control and Optimization},
  volume={47},
  number={1},
  pages={301--326},
  year={2008},
  publisher={SIAM}
}

@article{Bemporad2002,
  author  = {Alberto Bemporad and Manfred Morari and Vivek Dua and Efstratios N. Pistikopoulos},
  title   = {The explicit linear quadratic regulator for constrained systems},
  journal = {Automatica},
  volume  = {38},
  number  = {1},
  pages   = {3--20},
  year    = {2002}
}

@article{Bemporad2006,
  author  = {Alberto Bemporad and Carlo Filippi},
  title   = {Suboptimal explicit receding horizon control with guaranteed stability},
  journal = {Automatica},
  volume  = {42},
  number  = {3},
  pages   = {457--462},
  year    = {2006}
}

@article{BemporadMorari1999,
  title={Control of systems integrating logic, dynamics, and constraints},
  author={Bemporad, Alberto and Morari, Manfred},
  journal={Automatica},
  volume={35},
  number={3},
  pages={407--427},
  year={1999}
}

@article{colaneri2007robust,
  title={Robust model predictive control of discrete-time switched systems},
  author={Colaneri, Patrizio and Scattolini, Riccardo},
  journal={IFAC Proceedings Volumes},
  volume={40},
  number={14},
  pages={208--212},
  year={2007},
  publisher={Elsevier}
}

@article{muller2012model,
  title={Model predictive control of switched nonlinear systems under average dwell-time},
  author={M{\"u}ller, Matthias A and Martius, Pascal and Allg{\"o}wer, Frank},
  journal={Journal of Process Control},
  volume={22},
  number={9},
  pages={1702--1710},
  year={2012},
  publisher={Elsevier}
}

@article{mhaskar2008robust,
  title={Robust predictive control of switched systems: Satisfying uncertain schedules subject to state and control constraints},
  author={Mhaskar, Prashant and El-Farra, Nael H and Christofides, Panagiotis D},
  journal={International Journal of Adaptive Control and Signal Processing},
  volume={22},
  number={2},
  pages={161--179},
  year={2008},
  publisher={Wiley Online Library}
}

@Article{Tabuada.07,
  Title                    = {Event-triggered real-time scheduling of stabilizing control tasks},
  Author                   = {P. Tabuada},
  Journal                  = {IEEE Transactions on Automatic Control},
  Year                     = {2007},
  Number                   = {9},
  Pages                    = {1680--1685},
  Volume                   = {52},

  Publisher                = {IEEE}
}

@Article{AubinLygerosQuincampoixSastrySeube02,
  Title                    = {Impulse Differential Inclusions: a Viability Approach to Hybrid Systems},
  Author                   = {J.-P. Aubin and J. Lygeros and M. Quincampoix and S. S. Sastry and N. Seube},
  Journal                  = IEEETACz,
  Year                     = {2002},
  Number                   = {1},
  Pages                    = {2-20},
  Volume                   = {47}
}

@Book{vanderSchaftSchumacher00,
  Title                    = {An Introduction to Hybrid Dynamical Systems},
  Author                   = {A. van der Schaft and H. Schumacher},
  Publisher                = {LNCIS, Springer},
  Year                     = {2000}
}

@phdthesis{findeisen2006nonlinear,
  title={Nonlinear model predictive control: a sampled-data feedback perspective},
  author={Findeisen, Rolf},
  year={2006},
  school={Stuttgart Univ.}
}

@book{GrunePannek2017,
  title={Nonlinear model predictive control},
  author={Gr{\"u}ne, Lars and Pannek, J{\"u}rgen},
  year={2017},
  edition={2nd},
  publisher={Springer}
}

@article{Mayne.ea.00.Automatica,
  title={Constrained model predictive control: Stability and optimality},
  author={Mayne, David Q and Rawlings, James B and Rao, Christopher V and Scokaert, Pierre OM},
  journal={Automatica},
  volume={36},
  number={6},
  pages={789--814},
  year={2000},
  publisher={Elsevier}
}

@incollection{nodozi2025solving,
  title={Solving Hybrid Model Predictive Control Problems via a Mixed-Integer Approach},
  author={Nodozi, Iman and Sanfelice, Ricardo G},
  booktitle={Model Predictive Control},
  pages={83--109},
  year={2025},
  publisher={Springer}
}

@article{camacho,
title = "Model predictive control techniques for hybrid systems",
journal = "Annual Reviews in Control",
volume = "34",
number = "1",
pages = "21--31",
year = "2010",
issn = "1367-5788",
author = "Camacho, E.F. and Ramirez, D.R. and Limon, D. and Mu\~{n}oz de la Pe\~{n}a, D. and Alamo, T.",
}

@Inbook{pereira,
author="Pereira, Fernando Lobo and Fontes, Fernando A. C. C. and Aguiar, Ant\'{o}nio Pedro and de Sousa, {Jo\~{a}o} Borges",
editor="Olaru, Sorin and Grancharova, Alexandra and Lobo Pereira, Fernando",
title="An Optimization-Based Framework for Impulsive Control Systems",
bookTitle="Developments in Model-Based Optimization and Control: Distributed Control and Industrial Applications",
year="2015",
publisher="Springer International Publishing",
address="Cham",
pages="277--300",
isbn="978-3-319-26687-9",
}

@ARTICLE{pakniyat, 
author={A. Pakniyat and P. E. Caines}, 
journal={IEEE Transactions on Automatic Control}, 
title={On the Relation Between the Minimum Principle and Dynamic Programming for Classical and Hybrid Control Systems}, 
year={2017}, 
volume={62}, 
number={9}, 
pages={4347--4362}, 
ISSN={0018-9286}, 
month={Sept},}

@article{maynesurvey,
title = "Model predictive control: Recent developments and future promise ",
journal = "Automatica ",
volume = "50",
number = "12",
pages = "2967--2986",
year = "2014",
note = "",
issn = "0005-1098",
author = "David Q. Mayne"
}

@ARTICLE{sopasakis, 
author={P. Sopasakis and P. Patrinos and H. Sarimveis and A. Bemporad}, 
journal={IEEE Transactions on Automatic Control}, 
title={Model Predictive Control for Linear Impulsive Systems}, 
year={2015}, 
volume={60}, 
number={8}, 
pages={2277--2282}, 
doi={10.1109/TAC.2014.2380672}, 
ISSN={0018-9286}, 
month={Aug},}

@ARTICLE{magni, 
author={L. Magni and R. Scattolini}, 
journal={IEEE Transactions on Automatic Control}, 
title={Model predictive control of continuous-time nonlinear systems with piecewise constant control}, 
year={2004}, 
volume={49}, 
number={6}, 
pages={900--906}, 
doi={10.1109/TAC.2004.829595}, 
ISSN={0018-9286}, 
month={June},}

@book{aubin,
 author = {Jean-Pierre Aubin and H{\'é}l{\`}ene Frankowska},
 title = {Set-Valued Analysis},
 year = {2009},
 publisher = {Birkh{\"a}user},
 address = {Basel},
}

@InProceedings{Cutler.Ramaker.80.ACC,
  author    = {Cutler, Charles R and Ramaker, Brian L},
  title     = {Dynamic matrix control -- a computer control algorithm},
  booktitle = {Joint Aut. Cont. Conf.},
  year      = {1980},
  number    = {17},
  pages     = {72},
}

@Article{Morari.Lee.99.CCE,
  author    = {Morari, Manfred and Lee, Jay H},
  title     = {Model predictive control: past, present and future},
  journal   = {Comp. \& Chem. Eng.},
  year      = {1999},
  volume    = {23},
  number    = {4-5},
  pages     = {667--682},
  publisher = {Elsevier},
}

@Article{Qin.Badgwell.03.CEP,
  author    = {Qin, S Joe and Badgwell, Thomas A},
  title     = {A survey of industrial model predictive control technology},
  journal   = {Control Engineering Practice},
  year      = {2003},
  volume    = {11},
  number    = {7},
  pages     = {733--764},
  publisher = {Elsevier},
}

@Book{Rawlings.Mayne.09,
  title     = {Model predictive control: Theory and design},
  publisher = {Nob Hill Pub.},
  year      = {2009},
  author    = {Rawlings, J. B. and Mayne, D. Q.},
}

@Book{Borrelli.ea.17,
  title     = {Predictive control for linear and hybrid systems},
  publisher = {Cambridge University Press},
  year      = {2017},
  author    = {Borrelli, F. and Bemporad, A. and Morari, M.},
}

@Article{Mayne.ea.90.TAC,
  author    = {Mayne, D. Q. and Michalska, H.},
  title     = {Receding horizon control of nonlinear systems},
  journal   = {IEEE Trans. Aut. Cont.},
  year      = {1990},
  volume    = {35},
  number    = {7},
  pages     = {814--824},
  publisher = {IEEE},
}

@Article{Chen.Allgower.98.Automatica,
  author       = {Chen, H. and Allg{\"o}wer, F.},
  title        = {A quasi-infinite horizon nonlinear model predictive control scheme with guaranteed stability},
  journal      = {Automatica},
  year         = {1998},
  volume       = {34},
  number       = {10},
  pages        = {1205--1217},
  booktitle    = {Control Conference (ECC), 1997 European},
  organization = {IEEE},
}

@Article{Richalet.ea.76.IFAC,
  author  = {Richalet, JLTJ},
  title   = {Algorithmic control of industrial processes},
  journal = {Proc. of the 4th IFAC Sympo. on Identification and System Parameter Estimation},
  year    = {1976},
  pages   = {1119--1167},
}

@article{mhaskar2005predictive,
  title={Predictive control of switched nonlinear systems with scheduled mode transitions},
  author={Mhaskar, Prashant and El-Farra, Nael H and Christofides, Panagiotis D},
  journal={IEEE Trans.  Aut. Cont.},
  volume={50},
  number={11},
  pages={1670--1680},
  year={2005},
  publisher={IEEE}
}

@TechReport{HybridMPCTR,
  Title                    = {Model Predictive Control of Hybrid Dynamical
Systems},
  Author                   = {R. G. Sanfelice and B. Altin},
  url={ArXiv},
  Year                     = {2026}
}

@article{goebel2019existence,
  title={Existence of optimal controls on hybrid time domains},
  author={Goebel, Rafal},
  journal={Nonlinear Analysis: Hybrid Systems},
  volume={31},
  pages={153--165},
  year={2019},
  publisher={Elsevier}
}

@book {220,
	title = {Hybrid Feedback Control},
	year = {2021},
	publisher = {Princeton University Press},
	organization = {Princeton University Press},
	address = {New Jersey},
	author = {R. G. Sanfelice}
}

@book {65,
	title = {Hybrid Dynamical Systems: Modeling, Stability, and Robustness},
	year = {2012},
	publisher = {Princeton University Press},
	organization = {Princeton University Press},
	doi = {http://www.u.arizona.edu/~sricardo/index.php?n=Main.Books},
	author = {R. Goebel and R. G. Sanfelice and A. R. Teel}
}

@mastersthesis {2,
	title = {Novel current control for {AC} motors},
	year = {2001},
	school = {Universidad Nacional de Mar del Plata},
	url = {https://hybrid.soe.ucsc.edu/files/preprints/2.pdf},
	author = {R. G. Sanfelice}
}

@inbook {154,
	title = {Hybrid Model Predictive Control},
	booktitle = {Handbook of Model Predictive Control},
	year = {2018},
	month = {09/2018},
	pages = {199-220},
	publisher = {Birkh{\"a}user },
	organization = {Birkh{\"a}user },
	edition = {Edition 1},
	address = {Basel},
	issn = {978-3-319-77489-3},
	doi = {10.1007/978-3-319-77489-3},
	author = {R. G. Sanfelice}
}

@article {286,
	title = {Regularity of Optimal Solutions and the Optimal Cost for Hybrid Dynamical Systems via Reachability Analysis},
	journal = {Automatica},
	volume = {152},
	year = {2023},
	month = {March},
	author = {B. Altin and R. G. Sanfelice}
}

@article {209,
	title = {Hybrid Dynamical Systems with Hybrid Inputs: Definition of Solutions and Applications to Interconnections},
	journal = {International Journal of Robust and Nonlinear Control},
	volume = {30},
	year = {2019},
	month = {October},
	pages = {5892{\textendash}5916},
	author = {P. Bernard and R. G. Sanfelice}
}

@article {185,
	title = {Forward Invariance of Sets for Hybrid Dynamical Systems ({Part I})},
	journal = {IEEE Transactions on Automatic Control},
	volume = {64},
	number = {6},
	year = {2019},
	month = {June},
	pages = {2426-2441},
	author = {J. Chai and R. G. Sanfelice}
}

@article {52,
	title = {Interconnections of Hybrid Systems: Some Challenges and Recent Results},
	journal = {Journal of Nonlinear Systems and Applications},
	volume = {2},
	number = {1-2},
	year = {2011},
	pages = {111{\textendash}121},
	doi = {http://jnsaonline.watsci.org/abstract_pdf/2011v2/v2n1-pdf/13.pdf},
	author = {R. G. Sanfelice}
}

@conference {228,
	title = {Model Predictive Control for Hybrid Dynamical Systems: Sufficient Conditions for Asymptotic Stability with Persistent Flows or Jumps},
	booktitle = {Proc. Amer. Cont. Conf.},
	year = {2020},
	month = {July},
	pages = {1791-1796},
	author = {B. Altin and R. G. Sanfelice}
}

@conference {205,
	title = {A Model Predictive Control Framework for Asymptotic Stabilization of Discretized Hybrid Dynamical Systems},
	booktitle = {Proceedings of the 2019 IEEE Conference on Decision and Control},
	year = {2019},
	month = {December},
	pages = {2356 - 2361},
	author = {P. Ojaghi and B. Altin and R. G. Sanfelice}
}

@conference {193,
	title = {Asymptotically Stabilizing Model Predictive Control for Hybrid Dynamical Systems},
	booktitle = {Proceedings of the American Control Conference},
	year = {2019},
	month = {July},
	pages = {3630-3635},
	author = {B. Altin and R. G. Sanfelice}
}

@conference {179,
	title = {A Model Predictive Control Framework for Hybrid Systems},
	booktitle = {Proceedings of the 6th IFAC Conference on Nonlinear Model Predictive Control (NMPC)},
	volume = {51},
	year = {2018},
	month = {August},
	pages = {128-133},
	author = {B. Altin and P. Ojaghi and R. G. Sanfelice}
}

@conference {74,
	title = {A Toolbox for Simulation of Hybrid Systems in {M}atlab/{S}imulink: {H}ybrid {E}quations ({HyEQ}) {T}oolbox},
	booktitle = {Proceedings of Hybrid Systems: Computation and Control Conference},
	year = {2013},
	pages = {101{\textendash}106},
	doi = {http://doi.org/tz9},
	url = {https://hybrid.soe.ucsc.edu/files/preprints/74.pdf},
	author = {R. G. Sanfelice and D. A. Copp and P. Nanez}
}

@conference {4,
	title = {Results on convergence in hybrid systems via detectability and an invariance principle},
	booktitle = {Proc. 24th American Control Conference},
	series = {NULL},
	year = {2005},
	pages = {551{\textendash}556},
	doi = {http://ieeexplore.ieee.org/iel5/9861/31519/01470014.pdf?tp=\&arnumber=1470014\&isnumber=31519},
	url = {https://hybrid.soe.ucsc.edu/files/preprints/4.pdf},
	author = {R. G. Sanfelice and R. Goebel and A.R. Teel}
}

@conference {1,
	title = {Novel current control for {AC} motors with low torque ripple},
	booktitle = {Proc. IX Workshop on Information Processing and Control RPIC},
	series = {NULL},
	year = {2001},
	url = {https://hybrid.soe.ucsc.edu/files/preprints/1.pdf},
	author = {M. Benedetti and J.F. Rovira and R. G. Sanfelice}
}

@STRING{ieeetacz = ieeetacl}

\NotForJournal{
\appendices
\section{Proofs}
	\label{sec:proofs}

\subsection{Proof of Lemma~\ref{prop:Lyaplobound}}

Take $\delta \in (0,\varepsilon)$ such that ${\widetilde{\alpha}_1^{-1}}\circ\widetilde{\alpha}_2(\delta) \leq \varepsilon$.
Let $(x,u)$ be a solution pair to $\HS$ with compact hybrid time domain and terminal time $(T,0)$; hence, it is a continuous solution pair with $\dom (x,u) = [0,T] \times \{0\}$.
Suppose $|x(0,0)|_\A \leq \delta$.
Pick any $s \geq 0$ and the largest interval $[t^1,t^2] \subset [0,T]$ containing $s$ such that $|x(s,0)|_\A \leq \delta$ and $|x(s,0)|_\A \leq |x(0,0)|_\A$ for all $t \in [t^1,t^2]$.
From differentiability of $\widetilde{V}$, the bounds in
\eqref{eq:candidate} and \eqref{eq:lodescent} hold on $\cl(\Pi(C))$.
Then, since by the choice of $\delta$, $|x(s,0)|_\A \leq \varepsilon$ for all $t \in [t^1,t^2]$, using \eqref{eq:candidate} and \eqref{eq:lodescent}, we have
\begin{align*}
\widetilde{\alpha}_2(|x(t,0)|_\A &\geq 
\widetilde{V}(x(t,0)) \geq \exp(\lambda (t-t^1))\widetilde{V}(x(t^1,0))
 \\
& \geq \exp(\lambda (t-t^1)) \widetilde{\alpha}_1(|x(t^1,0)|_\A)
\end{align*}
for all $t \in [t^1,t^2]$,
which implies
\begin{align*}
|x(t,0)|_\A &\geq \widetilde{\alpha}_2^{-1}(\exp(\lambda (t-t^1)) \widetilde{\alpha}_1(|x(t^1,0)|_\A))
\end{align*}
for all $t \in [t^1,t^2]$. Since 
$|x(t^1,0)|_\A  = |x(0,0)|_\A$ by the very definition of the interval $[t^1,t^2]$,
\begin{align*}
|x(t,0)|_\A &\geq \widetilde{\alpha}_2^{-1}(\exp(\lambda (t-t^1)) \widetilde{\alpha}_1(|x(0,0)|_\A))\\
 &\geq \widetilde{\alpha}_2^{-1}(\exp(\lambda (t^2-t^1)) \widetilde{\alpha}_1(|x(0,0)|_\A))\\
& \geq 
\widetilde{\alpha}_2^{-1}(\min\{1,\exp(\lambda T)\}\widetilde{\alpha}_1(|x(0,0)|_\A))
\end{align*}
for all $t \in [t^1,t^2]$.
Note that this property holds for $t$'s on any interval such as $[t^1,t^2]$.
At other $t$'s,  $|x(t,0)|_\A > \delta$.
Then, combining the lower bounds on $|x(t,0)|_\A$ above and the case when $|x(0,0)|_\A > \delta$ gives 
\begin{align*}
|x(t,0)|_\A 
& \geq \min\{
\widetilde{\alpha}_2^{-1}(\min\{1,\exp(\lambda T)\}\widetilde{\alpha}_1(|x(0,0)|_\A)),\delta\}
\end{align*}
for all $(t,0) \in \dom (x,u)$.
Since ${\widetilde{\alpha}_1}$ and ${\widetilde{\alpha}_2}$ are class-$\classK$ functions and $\delta$ is positive, there exists 
class-$\classK$  function $\alpha$ lower bounding $r \mapsto \min\{
\widetilde{\alpha}_2^{-1}(\min\{1,\exp(\lambda T)\}\widetilde{\alpha}_1(r)),\delta\}$.

\subsection{Proof of Lemma~\ref{prop:nice}}

Without loss of generality, suppose that there exists~$c>0$ such that~$\sigma(r)\geq c$ for all~$r>0$. Let
$\varphi$ be given as in \eqref{eqn:varphi}, which is continuous.
Let~$\tau(0):=0$ and~$\tau(r):=r/(2\varphi(r))$ for all~$r\in(0,\varepsilon]$, and note that~$\tau$ is continuous, positive definite, and lower bounds the ordinary time to reach the~$(r/2)$-disc around~$\A$ from the~$r$-disc around~$\A$ via flows for~$r\leq\varepsilon$. Let~$\rho(r):=\min\{r,\varepsilon\}$ and~$\widetilde{\tau}(r):= \tau(\rho(r))$ for all~$r\geq 0$. Observing that~$\Anorm{x(t,0)}>\rho(\Anorm{x(0,0)})/2$ for any continuous solution pair~$(x,u)$ and~$t< \widetilde{\tau}(\Anorm{x(0,0)})$ satisfying~$(t,0)\in\dom(x,u)$, the proof is completed by picking~$\alpha$ to be any class-$\classK$ function lower bounding both $\rho$ and~$\widetilde{\tau}$.

\subsection{Upper Semicontinuous Maps}

\begin{lemma}
\label{lem:cont}
Let~$h:S\to\reals$ be a continuous function, where~$S\subset\reals^n$ is closed. Then, given a compact set~$\A\subset\reals^n$, there exists~$\bar{r}\geq 0$
such that
\begin{equation}
\label{eq:min}
	\bar{r}=\min\{r\in\realsgeq: (\A+\ubar{r}\ball)\cap S\neq\varnothing\}.\
\end{equation}
Moreover, the function~$\psi:[\bar{r},\infty)\to\reals$ given by
\[
	\psi(r):=\max_{x\in\Gamma(r)}h(x) \quad \forall r\in[\bar{r},\infty),
\]
where~$\Gamma:\realsgeq\rightrightarrows S$ is the set-valued mapping defined as~$\Gamma(r)=(\A+r\ball)\cap S$ for all~$r\in\realsgeq$, is upper semicontinuous.
\end{lemma}
\begin{IEEEproof}
	Take any~$x\in S$ and let~$\delta:=\Anorm{x}$. Since~$\A$ is compact, there exists~$a\in\A$ such that~$|x-a|=\delta$. Hence,~$(\A+\delta\ball)\cap S$ is nonempty. Consequently, the infimum of the continuous function~$x\mapsto\Anorm{x}$ on~$(\A+\delta\ball)\cap S$ is attained, so~\eqref{eq:min} holds for some~$\bar{r}\geq 0$.

Now, let~$h':\realsgeq\times S\to\reals$ be such that~$h'(r,x)=h(x)$ for all~$(r,x)\in\realsgeq\times S$. Then,
\[
	\psi(r)=\max_{x\in\Gamma(r)}h'(r,x) \quad \forall r\in[\bar{r},\infty).
\]
Since~$\realsgeq\times S\supset \gph \Gamma$,~$\Gamma$ is compact-valued~(as~$\A$ is compact and~$S$ is closed), and~$h'$ is continuous, by Berge's maximum theorem~\cite[Theorem~1.4.16]{aubin},~$\psi$ is upper semicontinuous if~$\Gamma$ is upper semicontinuous. To show upper semicontinuity of~$\Gamma$, let~$\Gamma':\realsgeq\rightrightarrows \reals^n$ be the set-valued mapping such that~$\Gamma'(r)=\varnothing$ for all~$r<0$, and~$\Gamma'(r):=(\A+r\ball)$ for all~$r\geq 0$. Then, since~$\Gamma'$ is upper semicontinuous and~$\Gamma'(r)$ is closed for all~$r\in\reals$, by~\cite[Lemma~5.15]{65},~$\Gamma'$ is outer semicontinuous. Moreover,~\cite[Lemma~5.10]{65} implies that~$\gph \Gamma'$ is closed, so~$\gph \Gamma=\gph \Gamma'\cap (\realsgeq\times S)$ is closed. Therefore,~$\Gamma$ is upper semicontinuous, again as a result of Lemmas~5.10 and~5.15 of~\cite{65}, and consequently,~$\psi$ is upper semicontinuous.
\end{IEEEproof}

}

\begin{IEEEbiography}[{\includegraphics[width=1in,height=1.25in,clip,keepaspectratio]{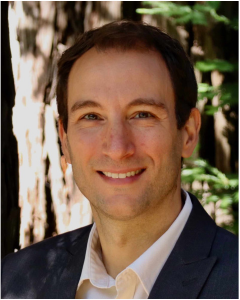}}]{Ricardo G. Sanfelice}
(SM'01, F'22) received the B.S. degree in Electronics Engineering from the Universidad de Mar del Plata, Buenos Aires, Argentina, in 2001, and the M.S. and Ph.D. degrees in Electrical and Computer Engineering from the University of California, Santa Barbara, in 2004 and 2007, respectively. In 2007 and 2008, he held postdoctoral positions at the Laboratory for Information and Decision Systems at the Massachusetts Institute of Technology and at the Centre Automatique et Systemes at the Ecole de Mines de Paris. In 2009, he joined the faculty of the Department of Aerospace and Mechanical Engineering at the University of Arizona, Tucson, where he was an Assistant Professor. In 2014, he joined the University of California, Santa Cruz, where he is currently Professor and Chair in the Department of Electrical and Computer Engineering. Prof. Sanfelice is the recipient of the 2013 SIAM Control and Systems Theory Prize, the National Science Foundation CAREER award, the Air Force Young Investigator Research Award, the 2010 IEEE Control Systems Magazine Outstanding Paper Award, and the 2020 Test-of-Time Award from the Hybrid Systems: Computation and Control Conference. He is Associate Editor for Automatica, Communicating Editor for the Journal of Nonlinear Science, and a Fellow of the IEEE. His research interests are in modeling, stability, robust control, observer design, and simulation of nonlinear and hybrid systems with applications to power systems, aerospace, and biology.
\end{IEEEbiography}
\begin{IEEEbiography}[{\includegraphics[width=1in,height=1.25in,clip,keepaspectratio]{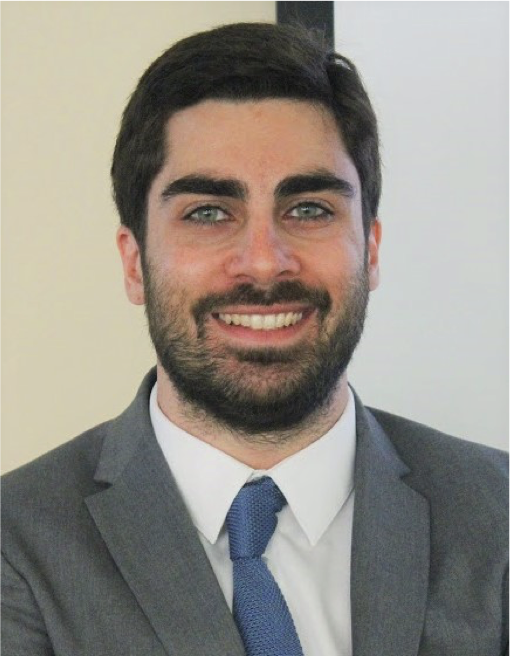}}]{Berk Alt{\i}n}
(M'17) received the B.S. degree in mechatronics from Sabanci University, Istanbul, Turkey, in 2011, and the M.S. and Ph.D. degrees in electrical engineering: systems and the M.S. degree in mathematics from the University of Michigan, Ann Arbor, MI, USA, in 2013, 2016, and 2016, respectively. From 2011 to 2016, he was a Fulbright Fellow with the University of Michigan. In 2016, he joined the Hybrid Systems Laboratory, University of California, Santa Cruz, CA, USA, as a Post-Doctoral Researcher. Since 2021, he has been working in the autonomous vehicles industry as a Technical Lead and Planning and Control Software Engineer. His research interests include hybrid systems, model predictive control, iterative learning control, repetitive processes, and multidimensional systems, with applications in cyber-physical systems, power systems, robotics, and additive manufacturing.
\end{IEEEbiography}
\vfill

\end{document}